\documentclass[11pt]{amsart}
\usepackage{amsmath,amsfonts,amscd,amsthm}
\usepackage{pgfplots}
\pgfplotsset{compat=1.17}
\usepackage{mathtools}
\usepackage{bm}
\usepackage{amssymb}
\usepackage[flushmargin]{footmisc}
\usepackage{color} 
\usepackage{ragged2e}
\usepackage{esvect}
\usepackage{mathrsfs}
\usepackage{typearea}
\usepackage[a4paper, left=1in, right=1in, top=1.5in, bottom=1in]{geometry}

\usepackage{xcolor}
\usepackage[colorlinks,
linkcolor=blue,
anchorcolor=blue,
citecolor=blue]{hyperref} %for show the ref

\newcommand{\dbar}{\ensuremath{\overline\partial}}
\newcommand{\dbarstar}{\ensuremath{\overline\partial^*}}
\newcommand{\C}{\ensuremath{\mathbb{C}}}

\newcommand{\R}{\ensuremath{\mathbb{R}}}

\newcommand{\reg}{\mathrm{reg}}
\newcommand{\sing}{\mathrm{sing}}

\newcommand{\norm}[1]{\left\Vert#1\right\Vert}
\newcommand{\abs}[1]{\left\vert#1\right\vert}

\newcommand{\set}[1]{\left\{#1\right\}}
\DeclareMathOperator{\Ker}{Ker}

\newcommand{\cali}[1]{\mathscr{#1}}

\DeclareFontFamily{U}{mathx}{\hyphenchar\font45}
\DeclareFontShape{U}{mathx}{m}{{n-1}}{
      <5> <6> <7> <8> <9> <10>
      <10.95> <12> <14.4> <17.28> <20.74> <24.88>
      mathx10
      }{}
\DeclareSymbolFont{mathx}{U}{mathx}{m}{{n-1}}
\DeclareFontSubstitution{U}{mathx}{m}{{n-1}}
\DeclareMathAccent{\widecheck}{0}{mathx}{"71}
\DeclareMathAccent{\wideparen}{0}{mathx}{"75}

\def\dom{\operatorname{Dom}\,}
\def\re{\operatorname{Re}\,}
\def\eps{\varepsilon}

\def\omz{\Omega}

\makeatletter
\newcommand{\sumprime}{\if@display\sideset{}{'}\sum%
	\else\sum'\fi}
\makeatother

\newtheorem{thm}{Theorem}[section]
\newtheorem{prop}[thm]{Proposition}
\newtheorem{lem}[thm]{Lemma}
\newtheorem{cor}[thm]{Corollary}

\theoremstyle{definition}
\newtheorem{defin}[thm]{Definition}

\theoremstyle{remark}
\newtheorem{rem}[thm]{Remark}
\numberwithin{equation}{section}

\providecommand\ufootnote[1]{{\let\thefootnote\relax\footnote[0]{#1}}}

\newcommand{\N}{\mathbb{N}}

\newcommand{\ol}{\overline}

\newcommand{\pa}{\partial}

 \DeclareMathOperator{\im}{Im}
\DeclareMathOperator{\ke}{Ker}

 \DeclareMathOperator{\Dom}{Dom}
  
\DeclareMathOperator{\Tr}{Tr} 

\begin{document}

\title[Heat kernel asymptotics]
{Semi-classical heat kernel asymptotics on complex manifolds with boundary}

\author[]{Chin-Yu Hsiao}
	\address{Department of Mathematics, National Taiwan University, 
	Taipei, Taiwan, {\bf ORCID iD:} 0000-0002-1781-0013}
	\email{chinyuhsiao@ntu.edu.tw or chinyu.hsiao@gmail.com}
	\thanks{Chin-Yu Hsiao was partially supported by National Science and Technology Council project 113-2115-M-002-011-MY3.}
    \author[]{George Marinescu}
\address{Department of Mathematics and Computer Science, 
Universit{\"a}t zu K{\"o}ln\\ 
Weyertal 86-90, \newline
    \mbox{\quad}\,50931 K{\"o}ln, Germany\\
    \newline
\mbox{\quad}\, Institute of Mathematics `Simion Stoilow', Romanian Academy,
Bucharest, Romania}
\thanks{George Marinescu was partially supported 
by the DFG-funded projects SFB TRR 191 `Symplectic Structures in Geometry, 
Algebra and Dynamics' (Project-ID 281071066\,--\,TRR 191), 
the ANR-DFG project `QuaSiDy\,--\,Quantization, Singularities, 
and Holomorphic Dynamics' (Project-ID 490843120)}
\email{gmarines@math.uni-koeln.de}
\author[]{Weixia Zhu}
\address{Fakult\"{a}t f\"{u}r Mathematik, Universit\"{a}t 
Wien, Oskar-Morgenstern-Platz 1, 1090 Wien, Austria}
\thanks{Weixia Zhu was partially supported by FWF projects 
10.55776/ESP367 and 10.55776/PAT1879425.} 
\email{weixia.zhu@univie.ac.at; zhuvixia@gmail.com}

\begin{abstract} 
Let $M$ be a relatively compact open subset of a complex manifold $M'$ 
with smooth boundary $X$ and let $L$ be a holomorphic line bundle over $M'$.
Assuming that condition $Z(q)$ holds, we establish the semi-classical asymptotic behavior 
of $e^{-\frac{t}{k}\Box^{q}_k}$ near the boundary $X$ as $k\to\infty$, 
where $\Box^{q}_k$ is the $\bar{\partial}$-Neumann Laplacian acting on 
$(0,q)$-forms on $M$ with values in $L^k$. 
Our results extend the seminal work of Bismut to complex manifolds with boundary. 
As applications of our results, we provide a heat kernel-based proof 
of the holomorphic Morse inequalities for complex manifolds with boundary 
and derive a semi-classical Weyl law for the $\dbar$-Neumann Laplacian.
	
\bigskip

\noindent{{\sc Mathematics Subject Classification}: 32W05, 32G05, 35J25, 35P15.}	

	\smallskip
	
	\noindent{{\sc Keywords}:  holomorphic line bundles, 
	$\dbar$-Neumann Laplacian, heat kernel asymptotics}
\end{abstract}

\maketitle

\tableofcontents

\section{Introduction}\label{sec:intro}

Heat kernel asymptotics provide a fundamental link between analysis, 
geometry, and spectral theory. Beginning with the work of 
Minakshisundaram and Pleijel \cite{MP49} for the scalar Laplacian, 
and later developments by McKean and Singer \cite{MS67}, 
Gilkey \cite{Gi75} and Patodi \cite{Pa71} have played a central role 
in index theory, spectral asymptotics, and global analysis. 
In several complex variables, Beals and Stanton \cite{BS87,BS88} 
studied the heat equation for the $\dbar$-Neumann problem and 
established a Weyl-type theorem for the $\dbar$-Neumann Laplacian under 
the condition $Z(q)$.

A different but closely related direction concerns heat kernel asymptotics 
for high tensor powers $L^k$ of a holomorphic line bundle over a 
complex manifold without boundary. Such asymptotics form one of the standard 
approaches to holomorphic Morse inequalities \cite{De85}. 
Important contributions were made by Bismut \cite{Bi87}, 
who used probabilistic methods, and later by Demailly \cite{De91} 
and Bouche \cite{Bo96}, who developed classical heat kernel arguments. 
In \cite[\S 1.6--1.7]{MM07}, Ma and Marinescu gave a new proof of heat kernel 
asymptotics and holomorphic Morse inequalities inspired by the work of Bismut and Lebeau. 
More recently, Hsiao--Zhu \cite{HZ23} established analogous asymptotic 
expansions for the Kohn Laplacian on CR manifolds with values in high tensor 
powers of a line bundle, with applications to CR and $\R$-equivariant Morse inequalities.

The purpose of this paper is to develop the corresponding heat kernel theory 
for complex manifolds with boundary. More precisely, we study the asymptotic 
behavior, as $k\to\infty$, of the heat kernel of the $\dbar$-Neumann Laplacian acting 
on $(0,q)$-forms with values in $L^k$ on a compact complex manifold with boundary. 
Our main emphasis is on the boundary scaling regime, where the non-coercive nature 
of the $\dbar$-Neumann boundary condition introduces substantial new difficulties 
compared with the elliptic or CR settings.

The starting point of our analysis is the model domain
\begin{equation*}
\ol{M}_0=\Big\{{\bf z}=(z,z_n)\in \C^n:\ \im z_n+ 
\sum_{j=1}^{n-1}\lambda_j|z_j|^2\le 0\Big\},
\end{equation*}
equipped with a weighted $\dbar$-Neumann Laplacian $\Box_{\phi_0}^q$. 
Building on ideas from Stanton \cite{S80}, Tie \cite{T97}, 
and Hsiao-Zhu \cite{HZ23}, we derive an explicit formula for the heat kernel 
of $\Box_{\phi_0}^q$. In contrast to the classical works \cite{S80, T97}, 
where one assumes either $\lambda_j=1$ or $\lambda_j>0$ for all $j$, 
we allow the coefficients $\lambda_j$ to have arbitrary signs. 
Thus, the model domain is no longer necessarily strongly pseudoconvex. 
Moreover, we incorporate a nontrivial quadratic weight $\phi_0$, 
which is essential for applications to high powers of a line bundle.

The second main ingredient is a new scaling procedure adapted to the 
$\dbar$-Neumann problem near the boundary. By combining anisotropic scaling 
with a suitable $k$-dependent Hermitian metric, we show that the rescaled 
heat kernel converges to the heat kernel of the model operator on $\ol M_0$. 
This is the boundary counterpart of the classical scaling picture for heat kernels 
in the interior, but the argument is considerably subtler here because the 
$\dbar$-Neumann boundary condition is non-coercive, and the standard CR scaling 
methods do not directly apply. The specific choice of metric 
is tailored to capture the boundary contribution in the asymptotic 
regime relevant for Morse inequalities.

Our first main result identifies the boundary scaling limit of the heat kernel. 
Let $M$ be a relatively compact domain with smooth boundary $X$ in an 
$n$-dimensional complex manifold $M'$, and let $(L,h^L)$ be a holomorphic 
Hermitian line bundle over $M'$. For each $k\in\N$, we consider the 
$\dbar$-Neumann Laplacian $\Box_k^q$ acting on $(0,q)$-forms with 
values in $L^k$, with respect to a suitable $k$-dependent Hermitian metric 
near the boundary. 

We recall that condition $Z(q)$ means that at every boundary point the 
Levi form has either at least $n-q$ positive eigenvalues or at least $q+1$ negative eigenvalues; 
see Definition~\ref{zq}.
Assuming condition $Z(q)$, the heat operator
$$
e^{-t\Box_k^q}:L^2_{(0,q)}(M,L^k)\to \Dom\Box_k^q
$$
has a smooth kernel on $\R_+\times \ol M\times \ol M$. 
Near a boundary point $p\in X$, after choosing suitable 
local holomorphic coordinates and performing anisotropic scaling, 
we obtain a family of rescaled kernels $A_{(k)}(t,{\bf z},{\bf w})$. 
Our first theorem asserts that these kernels converge in the 
$\mathscr C^\infty$ topology on compact subsets to the heat kernel of the model operator.

\begin{thm}[Boundary scaling limit]
Assume that condition $Z(q)$ holds. Then, in local boundary coordinates 
and after the anisotropic scaling introduced in Section~\ref{heatasymptotic}, 
the rescaled heat kernel $A_{(k)}(t,{\bf z},{\bf w})$ converges in the 
$\mathscr{C} ^\infty$ topology on compact subsets to the heat kernel 
$e^{-t\Box_{\phi_0}^q}({\bf z},{\bf w})$ of the model weighted 
$\dbar$-Neumann Laplacian on $\ol M_0$.
In particular, the diagonal heat kernel of $e^{-\frac{t}{k}\Box_k^q}$ 
admits an explicit boundary asymptotic formula involving the curvature 
of $L$ and the Levi form of $X$; see Theorem~\ref{main1} for the precise statement.
\end{thm}

Away from the boundary, where the metric is independent of $k$, 
one recovers the familiar interior asymptotics governed by the curvature of $L$. 
Thus, the leading behavior of the heat kernel is described by two different model geometries: the standard complex-geometric model in the interior and the weighted $\dbar$-Neumann model near the boundary.

To make the boundary asymptotics effective, it is necessary to compute explicitly the heat kernel of the model operator $\Box_{\phi_0}^q$. This turns out to be a nontrivial problem even in the unweighted case, and explicit formulas were previously known only in special situations \cite{S80,T97}. Our second main result gives such a formula in full generality for the diagonal; in fact, we also obtain the off-diagonal kernel; see Theorem~\ref{heatkernelmodel}. The diagonal expression takes the following form.

\begin{thm}[Model heat kernel]
Let 
$$
\ol M_0=\Big\{(z,z_n)\in \C^{n-1}\times \C:\ \im z_n+
\sum_{j=1}^{n-1}\lambda_j|z_j|^2\le 0\Big\},
$$
and let $\Box_{\phi_0}^q$ be the weighted $\dbar$-Neumann Laplacian on 
$\ol M_0$ defined in Section~\ref{bdd}. Then the heat kernel of $\Box_{\phi_0}^q$ 
admits an explicit formula in terms of the curvature of the weight and the Levi form 
of the boundary; in particular, for the diagonal kernel, one obtains the integral 
expression \eqref{modelasymoptptic}. Moreover, an explicit formula for 
the off-diagonal kernel is given in Theorem~\ref{heatkernelmodel}; 
see also Theorem~\ref{hmd}.
\end{thm}

Let us denote by $H^\bullet(\ol M,L^k)$ the cohomology
of the $\overline\partial$-complex $(\Omega^{0,\bullet}(\ol M,L^k),\overline\partial)$
of forms with values in $L^k$ smooth up to the boundary of $M$.
If condition $Z(q)$ holds, by Kohn's subelliptic theory \cite{FK72,K65}, the 
$\dbar$-Neumann Laplacian has a discrete spectrum in bidegree $(0,q)$, 
and there exists a canonical isomorphism
$H^q(\ol M,L^k)\cong \Ker \Box_k^q$\,.

Combining the boundary asymptotics of Theorem~\ref{main1}, 
the interior asymptotics, and the standard heat-kernel trace inequalities 
as in \cite[\S1.6]{MM07}, we obtain a heat-kernel proof of 
holomorphic Morse inequalities for complex manifolds with boundary.
To state the resulting asymptotic formula, let $(L,h^L)$ be a semipositive holomorphic 
Hermitian line bundle over $M$, let $\dot{\mathcal{R}}^L$ and
$\dot{\mathcal L}$ denote 
the Hermitian endomorphisms induced by the curvature of $L$ and by the Levi form of
$X=\partial M$. Set
\begin{equation}\label{rjz}
M(q)=\{\mathbf z\in M \mid \dot{\mathcal R}^L_{\mathbf z}\,
\text{has exactly $q$ negative and $n-q$ positive eigenvalues}\}.
\end{equation}
For the Hermitian form $\dot{\mathcal R}^L_b-2\eta\dot{\mathcal L}$
on the holomorphic tangent space of $X$, we set,
\begin{equation}\label{rj}
\begin{aligned}
\R_x(q)=\{\eta<0 \mid &\dot{\mathcal R}^L_{b,x}-2\eta\dot{\mathcal L}_x\,
\text{has exactly $q$ negative and $n-q-1$ positive eigenvalues}\}.
\end{aligned}
\end{equation}

\begin{thm}[Holomorphic Morse inequalities]
Suppose that condition $Z(q)$ holds. Then, as $k\to\infty$, the dimensions 
$\dim H^q(\ol M,L^k)$ satisfy weak holomorphic Morse inequalities with both interior and boundary contributions, %given explicitly by \eqref{weak}. 
\begin{equation}\label{weak0}
\begin{aligned}
\dim H^{q}(\ol M,L^k) &\le\frac{k^{n}}{(2\pi)^{n}}\Big(\int_{M(q)}
\big|\det(\dot{\mathcal{R}}^L_{\mathbf z})\big|dv_{M'}(\mathbf z)+
\int_X\int_{\R_{x}(q)}\big|\det(\dot{\mathcal{R}}^L_{b,x}-2\eta\dot{\mathcal L}_x)\big|
\,d\eta\,dv_X(x)\Big)\\
&\quad+o(k^{n}),
\end{aligned}
\end{equation}
If, in addition, $Z(j)$ holds for all $j=0,\ldots,q$, then the strong Morse inequalities \eqref{strong} 
also hold; see Theorem~\ref{thmmorse}.
\end{thm}

These inequalities were previously obtained by Berman \cite{Be05} 
using Bergman kernel methods. Our result provides a new proof based on heat kernel asymptotics. In this sense, the present work may be viewed as a boundary analog of Bismut's heat-kernel proof of Demailly's holomorphic Morse inequalities.

A main application of these results is obtained in Section~\ref{5.2}, 
where we pass from manifolds with boundary to $q$-convex and 
$q$-concave manifolds via Andreotti-Grauert theory \cite{AG:62},
see Theorems \ref{miqconv} and \ref{miqconc}. As immediate consequences, we obtain asymptotic vanishing results for semipositive line bundles on $q$-convex manifolds and for seminegative line bundles on $q$-concave manifolds, namely \(\dim H^j(M,L^k)=o(k^n)\), $k\to\infty$,   in the Andreotti--Grauert range; see Corollaries~\ref{miqconvsp} and \ref{miqconcsn}. 
In particular, for a $1$-convex manifold endowed with a semipositive line bundle $L$
in the neighborhood of the exceptional analytic set, we have 
\(\dim H^j(M,L^k)=o(k^n)\), $k\to\infty$, 
for all $j\geq1$ (Corollary \ref{mi1convsp}).

We further apply the holomorphic Morse inequalities on manifolds with boundary 
to obtain asymptotic lower bounds for the
dimension of the spaces of holomorphic sections
of a semipositive line bundle over a $1$-concave manifold. 
These lower bounds generalize the well-known estimates obtained by 
Siu \cite{Si1:84} and Demailly \cite{De85} in their work on the Grauert–Riemen-schneider conjecture, which provides a characterization of Moishezon manifolds in terms of semipositive line bundles.
The corresponding problem for $1$-concave manifolds was considered in
\cite{Be05,HLM22,LMW25,Mar96,Mar16}.
Recall that a compact or $1$-concave complex manifold $M$ is 
called \textit{Moishezon} if the transcendence degree of the field 
of meromorphic functions on $M$ equals its dimension (see 
\cite{De85,Mar96}). 
We state here the result for 
a strictly pseudoconcave domain, that is, a relatively compact smooth
domain whose Levi form is negative definite, and refer to 
Corollary~\ref{cor:one-concave-positive} for the general case.
%===
\begin{cor}\label{cor:one-concave-positive0}
Let $M$ be a strictly pseudoconcave domain in a complex manifold $M'$
of dimension $n\geq 3$. 
Let $(L,h^L)$ be a semipositive holomorphic Hermitian line bundle 
on $M'$.
Assume that along $X$, the curvature of $L$ and the Levi form of $X$
are conformally equivalent in the sense that there exists a smooth 
positive function $g\in C^\infty(X)$ such that
$\dot{\mathcal R}^L_{b,x}=-g(x)\dot{\mathcal L}_{x}$ 
and $x\in X$ are Hermitian forms on \(T^{1,0}X\cap T^{1,0}M'\). 
Then, as $k\to\infty$,
\begin{equation}\label{eq:h0-one-concave0}
\begin{split}
\dim H^0(M,L^k)
&=\frac{k^n}{n!}\int_{M}c_1(L,h^L)^n
+\frac{k^n}{(2\pi)^n}
\int_{X}\int_{\mathbb R_x(0)}
\big|\det(\dot{\mathcal R}^L_{b,x}
-2\eta\dot{\mathcal L}_{x})\big|\,d\eta\,dv_{X}(x)\\
&\qquad+o(k^n).
\end{split}
\end{equation}
Furthermore, the manifold $M$ can be compactified to
a Moishezon manifold; that is, there exists a compact Moishezon 
$\widehat{M}$ such that $M$ is biholomorphic to an open set of $\widehat{M}$.
\end{cor}
%===
%Our asymptotic formula is a heat-kernel counterpart of the boundary-curvature formulas \cite[Corollary~6.7]{Be05}, \cite[Theorem 1.8]{LMW25}, while the conformal boundary assumption above is precisely the situation in which the upper and lower bounds match and yield an exact leading term. The novelty here is that the boundary contribution is obtained from the explicit model heat kernel and from the boundary scaling analysis developed in this paper. Furthermore, the asymptotic growth rate $\dim H^0(M,L^k)\sim k^n$ implies that $M$ admits a compactification to a Moishezon manifold; see Corollary~\ref{cor:one-concave-positive} for details. The assertion concerning the existence of a Moishezon compactification constitutes one of the main geometric applications of this work.
%
Our asymptotic formula is a heat-kernel counterpart of the 
boundary-curvature formulas in \cite[Corollary~6.7]{Be05} and 
\cite[Theorem~1.8]{LMW25}, while the conformal boundary assumption above is precisely the situation in which the upper and lower bounds match and yield an exact leading term. The novelty here is that the boundary contribution is obtained from the explicit model heat kernel and from the boundary scaling analysis developed in this paper.
We obtain, in addition to \eqref{eq:h0-one-concave0}, that the higher cohomology groups satisfy 
$\dim H^j(M,L^k)=o(k^n)$ for $1\le j\le n-2$. As a consequence, one obtains 
an explicit formula for the volume of $L$, consisting of the usual interior term 
together with a boundary contribution; in particular, $L$ is big.
This implies that $M$ admits a compactification to a Moishezon manifold; 
see Corollary~\ref{cor:one-concave-positive}. 
This compactification result is one of the main geometric applications of the paper.
We also discuss two further geometric applications: examples of \(1\)-concave manifolds arising from isolated singularities, and a deformation result showing that the positivity of the leading coefficient
in the Morse inequality, and hence the Moishezon property, 
persists under sufficiently small perturbations of the complex structure.
 
In addition to these applications to the holomorphic Morse inequalities, our heat trace asymptotic expansions also lead to a semi-classical Weyl law for the $\dbar$-Neumann Laplacian. More precisely, let $N_k^q(\lambda)$, $\lambda\ge0$, denote the spectral counting function of
$\frac{1}{k}\Box_k^q$. We show that there exists a positive Borel measure $\mu^q$ on $[0,\infty)$ whose Laplace transform coincides with the limiting heat trace $\lim_{k\to\infty}k^{-n}\Tr_q\big(e^{-\frac{t}{k}\Box^q_k}\big)$, and such that the spectral counting functions satisfy
\begin{equation}
N^q_k(\lambda)=k^n\mu^q([0,\lambda])+o(k^n)   
\end{equation}
at every continuity point of the distribution function $\lambda\mapsto\mu^q([0,\lambda])$, see Theorem~\ref{thm:scwl}. This is closely connected to the holomorphic Morse inequalities. Demailly \cite[Th\'eor\`eme 0.6]{De85} (see also \cite[\S3.2.2]{MM07}) computed the limit $\lim_{k\to\infty}k^{-n}N^q_{k,M}(\lambda)$ of the spectral counting function $N^q_{k,M}(\lambda)$ for the  Kodaira Laplacian with Dirichlet boundary conditions on $M$, defined with respect to a fixed Hermitian metric on $M'$. He then applied this computation to establish the holomorphic Morse inequalities on a compact manifold without boundary. Our method adopts the heat kernel based proof of Demailly's formula given in \cite[\S3.2.2]{MM07}.

% Let $\Box^q_{k,M}$ be the Kodaira Laplacian with
% Dirichlet boundary conditions on $M$, constructed
% with respect to a fixed Hermitian metric on $M'$, and let
% $N^q_{k,M}(\lambda)$ be its spectral counting function. 
% Demailly established the 
% semiclassical Weyl law $\lim_{k\to\infty}k^{-n}N^q_{k,M}(\lambda)$.

% This is in the spirit
% of the asymptotic distribution of eigenvalues in \cite[\S3.2.2]{MM07}.

Let us briefly comment on the proofs of our results. 
The first step is to analyze the model operator on $\ol M_0$ and derive an 
explicit formula for its heat kernel. Although the works \cite{S80,T97} 
provide important starting points, the passage to an arbitrary Levi signature 
and a nontrivial weight requires additional ideas. The second step is to 
introduce a boundary scaling adapted to the $\dbar$-Neumann condition 
and to prove uniform estimates for the corresponding rescaled heat kernels. 
We then show that these kernels converge to the model heat kernel. Finally, we combine the local boundary analysis with the standard interior asymptotics to derive global heat trace estimates, from which the Morse-type applications and the semi-classical Weyl law follow.
%Finally, we combine the local boundary analysis with  the standard interior asymptotics to derive global trace estimates,  first on manifolds with boundary and then, by exhaustion, on $q$-convex and $q$-concave manifolds.

We now explain the organization of the paper. In Section~\ref{sec:prelim}, we collect the standard notation, terminology, and background material used throughout the paper. In Section~\ref{bdd}, we analyze the model domain and derive explicit formulas for the heat kernel of the weighted $\dbar$-Neumann Laplacian, proving Theorem~\ref{hmd} and the more precise off-diagonal formula in Theorem~\ref{heatkernelmodel}. In Section~\ref{heatasymptotic}, we develop the boundary scaling procedure, establish uniform estimates for the scaled heat kernels, and prove the asymptotic convergence to the model kernel, leading to Theorem~\ref{main1}. 
In Section~\ref{morseinequality}, we apply these heat kernel asymptotics to prove the holomorphic Morse inequalities of Theorem~\ref{thmmorse}. Finally, in Section~\ref{5.2}, we use these inequalities to deduce Morse inequalities for $q$-convex and $q$-concave manifolds and to derive the $1$-concave Moishezon compactification criterion. In Section~\ref{5.3}, we show that the heat trace asymptotics also yield a semi-classical Weyl law for the $\dbar$-Neumann Laplacian.
%Finally, in Section~\ref{5.2}, we deduce Morse inequalities for $q$-convex and  $q$-concave manifolds and derive the $1$-concave Moishezon compactification criterion. 
%semipositive application culminating in the Moishezon compactification criterion.

\section{Preliminaries}\label{sec:prelim}

\subsection{Standard notations} \label{s-ssna}
We shall use the following notations: $\mathbb N=\{1,2,\ldots\}$
is the set of natural numbers, $\mathbb N_0=\mathbb N\cup\set{0}$, $\mathbb R$ 
is the set of real numbers, $\mathbb R_+:=\{x\in\mathbb R\mid x>0\}$,
$\overline{\mathbb R}_+:=\{x\in\mathbb R\mid x\geq0\}$. 
Let $\alpha = (\alpha_1,\ldots,\alpha_n) \in \mathbb{N}_0^{n}$ be a multiindex. We define its order by $\lvert \alpha \rvert := \alpha_1 + \cdots + \alpha_n$ and its length by $l(\alpha) := n$.
For $m\in\mathbb N$, write $\alpha\in\set{0,1,\ldots,m}^{n}$ if $\alpha_j\in\set{0,1,\ldots,m}$, 
$j=1,\ldots,{n}$. $\alpha$ is strictly increasing if $\alpha_1<\alpha_2<\ldots<\alpha_n$. 
Let $z = (z_1, \ldots, z_n) \in \mathbb{C}^n$ be the standard complex coordinate system on $\mathbb{C}^n$, where each component is given by
$z_j = x_{2j-1} + i x_{2j}$, $j = 1, \ldots, n,$
with $(x_1, \ldots, x_{2n})$ denoting the corresponding real coordinates on $\mathbb{R}^{2n}$.
We write
\[
\begin{split}
&z^\alpha=z_1^{\alpha_1}\ldots z^{\alpha_n}_n\,,\quad\ol z^\alpha=\ol z_1^{\alpha_1}\ldots\ol z^{\alpha_n}_n\,,\\
&\partial_{z_j}=\frac{\partial}{\partial z_j}=
\frac{1}{2}\Big(\frac{\partial}{\partial x_{2j-1}}-i\frac{\partial}{\partial x_{2j}}\Big)\,,\quad\partial_{\overline z_j}=
\frac{\partial}{\partial\overline z_j}=\frac{1}{2}\Big(\frac{\partial}{\partial x_{2j-1}}+i\frac{\partial}{\partial x_{2j}}\Big),\\
&\partial^\alpha_z=\partial^{\alpha_1}_{z_1}\ldots\partial^{\alpha_n}_{z_n}=\frac{\partial^{\abs{\alpha}}}{\partial z^\alpha}\,,\quad
\partial^\alpha_{\overline z}=\partial^{\alpha_1}_{\overline z_1}\ldots\partial^{\alpha_n}_{\overline z_n}=
\frac{\partial^{\abs{\alpha}}}{\partial\overline z^\alpha}\,.
\end{split}
\]
For $j, s\in\mathbb Z$, set $\delta_{j,s}=1$ if $j=s$, $\delta_{j,s}=0$ if $j\neq s$.

Let $W$ be a $\cali{C}^\infty$ paracompact manifold.
We let $TW$ and $T^*W$ denote the tangent bundle of $W$
and the cotangent bundle of $W$ respectively.
The complexified tangent bundle of $W$ and the complexified cotangent bundle of $W$ are denoted by $\mathbb C TW$
and $\mathbb C T^*W$, respectively. Write $\langle\,\cdot\,,\cdot\,\rangle$ to denote the pointwise
duality between $TW$ and $T^*W$.
We extend $\langle\,\cdot\,,\cdot\,\rangle$ bilinearly to $\mathbb C TW\times\mathbb C T^*W$.
Let $G$ be a $\cali{C}^\infty$ vector bundle over $W$. The fiber of $G$ at $x\in W$ will be denoted by $G_x$.
Let $E$ be another vector bundle over $W$. We write $E\boxtimes G^*$ for the vector bundle over $W\times W$ whose fiber at $(x,y)\in W\times W$ is the space of linear maps from $G_y$ to $E_x$.   Let $Y\subset W$ be an open set. 
From now on, the spaces of distribution sections of $G$ over $Y$ and
smooth sections of $G$ over $Y$ will be denoted by $\mathscr D'(Y, G)$ and $\cali{C}^\infty(Y, G)$ respectively.
Let $\mathscr{E}'(Y, G)$ be the subspace of $\mathscr D'(Y, G)$ 
whose elements have compact support in $Y$ and 
let $\cali{C}^\infty_c(Y, G)$ be the subspace of $\cali{C}^\infty(Y, G)$ whose elements have compact support in $Y$.

\subsection{\texorpdfstring{$\dbar$}{dbar}-Neumann Laplacian}\label{s-neumann}

In this section, we recall the definition of the $\dbar$-Neumann Laplacian (with values in $L^k$) and study some of its properties. 

Let $M$ be an $n$-dimensional compact complex manifold with a smooth boundary $X$. We may assume that $\ol M$ is the closure of a relatively compact open subset $M$ of a smooth complex manifold $M'$ in which $\ol M$ has a smooth boundary $X$. Fix a  Hermitian metric $\langle\,\cdot\mid\cdot\,\rangle$ on $\mathbb CTM'$ such that $T^{1,0}M'\perp T^{0,1}M'$. Let $dv_{M'}$ be the volume form on $M'$ induced by $\langle\,\cdot\,|\,\cdot\,\rangle$. The Hermitian metric $\langle\,\cdot\mid\cdot\,\rangle$ induces Hermitian metrics on $\mathbb CT^*M'$ 
and $\oplus^n_{q=1}T^{*0,q}M'$. We also denote all these Hermitian metrics by $\langle\,\cdot\,|\,\cdot\,\rangle$ and let $|\cdot|$ be the corresponding norm.

Fix a defining function $r\in\mathscr{C}^\infty(M',\mathbb R)$ so that $M=\{{\bf z}\in M'\,:\, r({\bf z})<0\}$ and $|dr|=1$ on $X$. The boundary $X$ is a CR manifold with CR structure $T^{1,0}X:=T^{1,0}M'\cap\mathbb CTX$. Let $\omega_0:=J(dr)|_X$, 
where $J$ is the complex structure map on $M'$. 
For $x\in X$, the Levi form $\mathcal{L}_x$ of $X$ at $x$ is the Hermitian form on $T^{1,0}_xX$ given by 
\begin{equation}\label{e-gue260131yyda}
\mathcal{L}_x(U,\ol V):=-\frac{1}{2i}d\omega_0(U\wedge\ol V),\ \ U, V\in T^{1,0}_xX.
\end{equation}
\begin{defin}[Condition $Z(q)$]\label{zq}
Let $q\in\{0,\ldots,n\}$. We say that condition $Z(q)$ holds for $M$ if  the 
Levi form has either at least $n-q$ positive eigenvalues or at least $q+1$ negative eigenvalues at every boundary point.
\end{defin}

Let $A$ be a $\mathscr{C}^\infty$ vector bundle over $M'$. 
Let $U$ be an open set in $M'$. Let 
\[
\begin{split}
&\mathscr{C}^\infty(U\cap \ol M,A),\ \ \mathscr D'(U\cap \ol M,A),\ \ \mathscr{C}^\infty_c(U\cap \ol M,A),\ \ 
\mathscr E'(U\cap \ol M,A),
\end{split}
\]
denote the spaces of restrictions to $U\cap\ol M$ of elements in 
\[
\begin{split}
\mathscr{C}^\infty&(M',A),\ \ \mathscr D'(M',A),\ \ \mathscr{C}^\infty_c(M',A),\ \ 
\mathscr E'(M',A),\\  
\end{split}
\]
respectively. When $U=M$, write $\mathscr{C}^\infty(\ol M,A):=\mathscr{C}^\infty(U\cap \ol M,A)$, $\mathscr D'(\ol M,A):=\mathscr D'(U\cap \ol M,A)$, $\mathscr{C}^\infty_c(\ol M,A):=\mathscr{C}^\infty_c(U\cap \ol M,A)$,
$\mathscr E'(\ol M,A):=\mathscr E'(U\cap \ol M,A)$. 

%Suppose that $D\subset M$ is an open set. We define $\omz^{0,q}(D)$ as the space of smooth sections of $T^{*0,q}M$ over $D$, $\omz_c^{0,q}(D)$ the subspace of $\omz^{0,q}(D)$ whose elements have compact support in $D$. %$\omz_c^{0,q}(\ol D)$ the subspace of $\omz^{0,q}(D)$ whose elements can be extended to compactly supported forms in $\ol M$. 

 For $0\le q\le n$, we define $\omz^{0,q}(M')$ as the space of smooth sections of $T^{*0,q}M'$ and $\omz_c^{0,q}(M')$ as the subspace of $\omz^{0,q}(M')$ whose elements have compact support in $M'$. Let $U\subset M'$ be an open set, we also define $\omz^{0,q}(U\cap\ol M):=\mathscr C^\infty(U\cap\ol M,T^{*0,q}M')$, $\omz^{0,q}_c(U\cap\ol M):=\mathscr C^\infty_c(U\cap\ol M,T^{*0,q}M')$, $\omz^{0,q}(\ol M):=\mathscr C^\infty(\ol M,T^{*0,q}M')$, and $\omz^{0,q}_c(M):=\mathscr C^\infty_c(M,T^{*0,q}M')$.
 
 %$\omz_c^{0,q}(\ol D)$ the subspace of $\omz^{0,q}(D)$ whose elements can be extended to compactly supported forms in $\ol M$. 
Let 
$$
\dbar=\dbar_q: \omz^{0,q}(M')\to \omz^{0,q+1}(M')
$$ 
be the Cauchy-Riemann operator. 
The Hermitian metric on $\C TM'$ induces a volume form $dv_{M'}$ on $M$, which allows us to define natural global $L^2$ inner products $(\,\cdot\, ,\,\cdot\,)_{\ol{M}}$ and $(\,\cdot\, ,\,\cdot\,)_{M'}$ on $\omz^{0,q}(\ol M)$ and $\omz^{0,q}_c(M')$. These inner products are given by
$$
(f|g)_{\ol M}=\int_{\ol M} \langle f\mid g\rangle dv_{M'},\quad f,g\in \omz^{0,q}(\ol M),
$$
and
$$
(f|g)_{M'}=\int_{M'} \langle f\mid g\rangle dv_{M'},\quad f,g\in \omz_c^{0,q}(M').
$$
We then define $L^2_{(0,q)}(M)$ as the completion of $\omz^{0,q}(\ol M)$ with respect to $(\,\cdot\, ,\,\cdot\,)_{\ol M}$ and $\|\cdot\|_{\ol M}$ as the associated norm.
Extend $\dbar_q$ to $L^2_{(0,q)}(M)$ by
\begin{equation}\label{e-gue260130yyd}
\dbar_q: {\rm Dom\,}\dbar_{q}\subset L^2_{(0,q)}(M)\to L^2_{(0,q+1)}(M),\end{equation}
where $u\in {\rm Dom\,}\dbar_{q}$ if we can find $u_j\in\Omega^{0,q}(\ol M)$, $j=1,2,\ldots$, such that 
$\lim_{j\to+\infty}u_j=u$ in $L^2_{(0,q)}(M)$, and $\lim_{j\to+\infty}\dbar_qu_j=v$ in $L^2_{(0,q+1)}(M)$ for some $v\in L^2_{(0,q+1)}(M)$. We set $\dbar_qu:=v$. 
Let $\dbarstar_q$ be its adjoint operator, whose domain is given by
\begin{equation}
{\rm Dom\,}\dbarstar_q=\set{f\in L^2_{(0, q+1)}(M) \mid \exists \,C>0,\,
|( f, \dbar_q g)_{\ol M}|\le C\|g\|_{\ol M},\ \forall g\in\Dom \dbar_{q} }.
\end{equation} 
%Let $\omL_{(0,q)}(\ol M)$ denote the space of $(0,q)$-forms on $M$ smooth up to the boundary. 
The condition for $f\in \omz^{0,q}(\ol M)$ to be in $\Dom\dbarstar_q$ gives rise, by integration by parts, to a boundary condition. In fact, if $M$ is a bounded domain in $\C^n$ with a $\mathscr C^1$ defining function $r$ such that $|d r|=1$ on $\partial M$, the boundary of $M$, and write
$$
f=\sumprime_{|J|=q} f_J \,d\bar z_J \in\Omega^{0,q}(\ol M).
$$
Then $f\in\Dom \dbarstar_{q-1}$ if and only if 
$$
(\dbar r\wedge)^* f:=(\dbar r)^* \lrcorner f=\sumprime_{|K|=q-1}\Big(\sum_{j=1}^{n} f_{jK}\frac{\partial r}{\partial z_j}\Big) d\bar z_K=0
$$
on $\partial M$, where
$$
(\dbar r)^*=\sum_{j=1}^{n} \frac{\partial r}{\partial z_j}\frac{\partial}{\partial \bar z_j}
$$
is the dual $(0, 1)$-vector field of $\dbar r$ and $\lrcorner$ denotes the contraction operator. 
Let 
$$
Q_q(f, g)=(\dbar_q f, \dbar_q
g)_M+(\dbarstar_{q-1} f,
\dbarstar_{q-1} g)_M
$$
with
$\Dom Q_{q}=\Dom \dbar_q\cap \Dom \dbarstar_{q-1}$. Then $Q_q$ defines a closed sesquilinear form  on $L^{2}_{(0,q)}(M)$.
By standard theory there is a unique self-adjoint operator $\Box_q$, called {Gaffney extension of \it $\dbar$-Neumann Laplacian}, corresponding to $Q_q$ that we can write as 
\begin{equation}\label{neumann}
\begin{split}
&\Box^q: {\rm Dom\,}\Box^q\subset L^2_{(0,q)}(M)\to L^2_{(0,q)}(M),\\
& \Box^q=\dbar_{q-1}\dbarstar_{q-1}+\dbarstar_q\dbar_{q},\\
&{\rm Dom\,}\Box^q=\set{f\in{\rm Dom\,}\dbar_{q}\cap{\rm Dom\,}\dbarstar_{q-1} \,\big|\, \dbar_{q}f\in{\rm Dom\,}\dbarstar_{q},\, \dbarstar_{q-1} f\in{\rm Dom\,}\dbar_{q-1}}. 
\end{split}
\end{equation}
The operator $\Box^q$ is elliptic, but its natural boundary conditions, the $\dbar$-Neumann boundary conditions, are non-coercive. We will provide a detailed representation of them in Section \ref{bdd} using local coordinates. In what follows, for simplicity of notation, we will suppress the subscript $q$.

Let $(L,h^L)$ be a holomorphic line bundle over $M'$, where $h^L$ is the Hermitian fiber metric on $L$. Assume that the local weight of $L$ on $U$ is $\phi$. More precisely, let $s$ be a local holomorphic trivializing section of $L$ on $U$, then locally,
\begin{equation}\label{trivia}
|s(\mathbf z)|^2_{h^L}=e^{-\phi(\mathbf z)},\quad \mathbf z\in U.
\end{equation}
For $k>0$, let $L^k$ be the $k$-th tensor power of the line bundle $L$ over $M'$, then $h^L$ induces a Hermitian fiber metric $h^{L^k}$ on $L^k$ and $s^k$ is a local holomorphic trivializing section of $L^k$.
Fix $R\gg1$, $0<\varepsilon\ll1$. For every $k\in\mathbb N$, let $\langle\,\cdot\,|\,\cdot\,\rangle_k$ be the Hermitian metric on $T^{*0,1}_{\bf z}M'$ for ${\bf z}\in\ol M$, defined as follows: %On the region $\set{{\bf z}\in M': 0\leq r({\bf z})\leq\frac{1}{\sqrt{k}}}$, we define it by
\begin{equation}\label{e-gue260125yyd}
\begin{split}
&\langle\,\dbar r\,|\,\dbar r\,\rangle_k=\chi_k^{R,\eps}(r)\langle\,\dbar r\,|\,\dbar r\,\rangle,\\
& \langle\,u\,|\,v\,\rangle_k=\langle\,u\,|\,v\,\rangle, \ \ \text{for } u,v \in T^{*0,1}_{\bf z}M', \text{ $\langle\,u\,|\,\dbar r\,\rangle=0$,} \text{ and $\langle\,v\,|\,\dbar r\,\rangle=0$,}\\
&\langle\,u\,|\,\dbar r\,\rangle_k=0, \ \ \text{for } u\in T^{*0,1}_{\bf z}M' \text{ and $\langle\,u\,|\,\dbar r\,\rangle=0$}.
\end{split}
\end{equation}
Here, $\chi_k^{R,\varepsilon}(r)$ is given by
\begin{equation}\label{newmetric}
\begin{aligned}
\chi_k^{R,\varepsilon}(r)=&
\frac{1}{k}\chi\left(\frac{k}{\eps}\Big(-r-\frac{R}{k}\Big)\right)+\chi\left(\frac{\sqrt k}{\eps}\Big(\frac{1}{\sqrt k}+r\Big)\right)\\
&+kr^2\left[1-\chi\left(\frac{k}{\eps}\Big(-r-\frac{R}{k}\Big)\right)\right]\left[1-\chi\left(\frac{\sqrt k}{\eps}\Big(\frac{1}{\sqrt k}+r\Big)\right)\right],
\end{aligned}
\end{equation}
where $\chi\in C^\infty(\mathbb{R})$ satisfies that
\begin{equation*}
\chi=1 \text{ on } (-\infty,0],\text{  and }
\chi=0 \text{ on } [1, +\infty].
\end{equation*}
For clarity, a schematic graph of $\chi_k^{R,\varepsilon}(r)$ is shown in the figure below.
%%%%%%%%%%
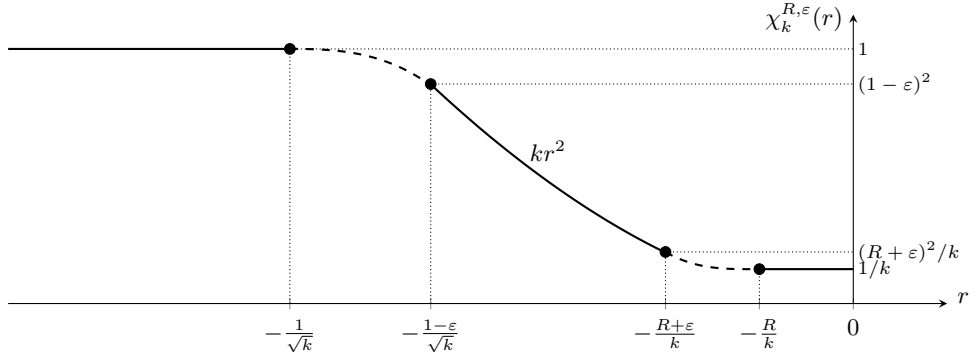
\begin{figure}[htbp]
  \centering
\begin{tikzpicture}
\begin{axis}[
    axis x line=bottom,
    axis y line=middle,
    ytick pos=right,
    yticklabel pos=right,
    xlabel={\footnotesize $r$},
    xmin=-9, xmax=1,
    ymin=0, ymax=4.2,
    clip=false,
    xtick={-6,-4.5,-2,-1,0},
    xticklabels={
        $-\frac{1}{\sqrt{k}}$,
        $-\frac{1-\varepsilon}{\sqrt{k}}$,
        $-\frac{R+\varepsilon}{k}$,
        $-\frac{R}{k}$,
        0
    },
    ytick={0.5,0.75,3.19,3.7},
    yticklabel style={font=\tiny},
    xticklabel style={font=\footnotesize},
    yticklabels={
        $1/k$,
        $(R+\varepsilon)^2/k$,
        $(1-\eps)^2$,
        $1$
    },
    width=14cm,
    height=5.4cm,
    x label style={
        at={(axis description cs:1,0.02)}, 
        anchor=west,
    },
    y label style={
        at={(axis description cs:0,1.02)}, 
        anchor=west,
        rotate=-90,
    },
]

% constant 1 region
\addplot[
    thick,
    domain=-9:-6
] {3.7};

% middle dashed region
\addplot[
    dashed, 
    thick,
    domain=-6:-4.5,
    samples=200
] {0.12593*(-x)^3 - 2.31333*(-x)^2 + 14.16*(-x) - 25.18};

% quadratic region
\addplot[
    thick,
    domain=-4.5:-2
] {0.2125*exp(1.8*ln(-x))};

\node[above, font=\footnotesize] at (axis cs:-3.25, {0.2125*exp(1.8*ln(3.25))+0.15}) {$kr^2$};

% middle dashed region
\addplot[
    dashed, 
    thick,
    domain=-2:-1,
    samples=200
] {0.3*(-x)^3 - 0.95*(-x)^2 + (-x) + 0.15};

% constant 1/k region
\addplot[
    thick,
    domain=-1:0
] {0.5};

% points
\addplot[only marks] coordinates {(-1,0.5)};
\addplot[only marks] coordinates {(-2,0.75)};
\addplot[only marks] coordinates {(-4.5,3.19)};
\addplot[only marks] coordinates {(-6,3.7)};

% dotted guide lines
\addplot[thin, densely dotted, samples=100] coordinates {(-1,0) (-1,0.5)};
\addplot[thin, densely dotted, samples=100] coordinates {(-2,0) (-2,0.75)};
\addplot[thin, densely dotted, samples=100] coordinates {(-2,0.75) (0,0.75)};
\addplot[thin, densely dotted, samples=100] coordinates {(-4.5,0) (-4.5,3.19)};
\addplot[thin, densely dotted, samples=100] coordinates {(-4.5,3.19) (0,3.19)};
\addplot[thin, densely dotted, samples=100] coordinates {(-6,0) (-6,3.7)};
\addplot[thin, densely dotted, samples=100] coordinates {(-6,3.7) (0,3.7)};
\node[font=\footnotesize, anchor=south east] at (axis cs:0,3.8)
{$\chi_k^{R,\varepsilon}(r)$};
\end{axis}
\end{tikzpicture}
\end{figure}
%%%%%%%%%
%On $\set{z\in M': 0\leq r(z)\leq\frac{2R}{\sqrt{k}}}$, we define $\langle\,\cdot\,|\,\cdot\,\rangle_k:=\langle\,\cdot\,|\,\cdot\,\rangle$. 

\noindent The idea behind the construction of \eqref{newmetric} is that the metric has to be compatible with the anisotropic scale \eqref{scalinging} near the boundary: tangential variables have a scale of $1/{\sqrt k}$, while the normal variable $r$ has a scale of $1/k$. To keep the normal and tangential parts of the rescaled $\dbar$-Neumann Laplacian balanced, we need to ensure
$$
\langle\,\dbar r\,|\,\dbar r\,\rangle_k
\sim
\frac{1}{k}\langle\,\dbar r\,|\,\dbar r\,\rangle
$$
on $\set{{\bf z}\in \ol M: -1/k\lesssim r({\bf z})\leq 0}$ and
$$
\langle\,\dbar r\,|\,\dbar r\,\rangle_k
\sim
\langle\,\dbar r\,|\,\dbar r\,\rangle\quad\mbox{on}\quad 
\set{{\bf z}\in \ol M: r({\bf z})\lesssim -1/\sqrt{k}}.
$$
Thus, the $k$-dependent metric is not merely a technical convenience; 
it encodes the natural boundary scale and leads to a stable model limit after rescaling.
In addition, in the figure, the region between these two regimes is further divided 
into three parts. This decomposition is needed both to construct a smooth metric 
that interpolates between these two regimes and to later prove the Morse inequalities.
The Hermitian metric $\langle\,\cdot\,|\,\cdot\,\rangle_k$ 
induces a Hermitian metric on $\oplus^n_{q=1}T^{*0,q}M'$ on $\ol M$, 
which we still denote by $\langle\,\cdot\,|\,\cdot\,\rangle_k$. 

Let $\omz^{0,q}(\ol M,L^k):=\cali{C}^\infty(\ol M,T^{*0,q}M'\otimes L^k)$.
Let $$\dbar_{k} \colon \omz^{0,q}(M',L^k)\to \omz^{0,q+1}(M',L^k)$$
be the Cauchy-Riemann operator acting on forms with values in $L^k$ such that
\begin{equation}\label{dbarbk}
\dbar_{k}(s^kf)=s^k\dbar f,
\end{equation}
where $s$ is a local holomorphic trivializing section of $L$ on an open set $U\subset M'$ and $f\in\omz^{0,q}(U)$. Then the $q-$th $\dbar_{k}$-cohomology on $\omz^{0,q}(\ol M,L^k)$ is given by
\begin{equation}
H^{q}(\ol M,L^k):=\frac{\ke \dbar_{k}|_{\omz^{0,q}(\ol M,L^k)}}{\im \dbar_{k}|_{\omz^{0,q-1}(\ol M,L^k)}},\quad 0\le q\le {n-1}.
\end{equation}

Fix $R\gg1$, $0<\varepsilon\ll1$. 
For every $k\in\mathbb N$, let $(\,\cdot\, ,\,\cdot\,)_{k,\ol M}$ be the $L^2$ inner product on $\Omega^{0,q}(\ol M,L^k)$ induced by the Hermitian metric $\langle\,\cdot\,|\,\cdot\,\rangle_k$ on $T^{*0,q}M'$, the volume form $dv_{M'}$, and the Hermitian metric $h^{L^k}$ on $L^k$. Let $\norm{\,\cdot\,}_{k,\ol M}$ be the corresponding $L^2$ norm. 
We denote by $L^2_{(0,q)}(M,L^k)$ the completion of $\omz^{0,q}(\ol M,L^k)$ with respect to $(\,\cdot\, ,\,\cdot\,)_{k,\ol M}$.
As in \eqref{e-gue260130yyd}, we can extend $\dbar_{k}$ to $L^2_{(0,q)}(M,L^k)$: 
$$
\dbar_{k}: {\rm Dom\,}\dbar_{k}\subset L^2_{(0,q)}(M,L^k)\to L^2_{(0,q+1)}(M,L^k).
$$
Let 
$$\dbarstar_{k}:{\rm Dom\,}\dbarstar_{k}\subset L^2_{(0,q+1)}(M,L^k)\to L^2_{(0,q)}(M,L^k)$$
be the adjoint of $\dbar_{k}$ with respect to $(\,\cdot\, ,\,\cdot\,)_{k,\ol M}$. 
The Gaffney extension of $\dbar$-Neumann Laplacian on $\ol M$ with values in $L^k$ is given by
\begin{equation}\label{be-gue210228yyd}
\begin{split}
&\Box^q_{k}: {\rm Dom\,}\Box^q_{k}\subset L^2_{(0,q)}(M,L^k)\to L^2_{(0,q)}(M,L^k),\\
&\Box^q_{k}=\dbarstar_{k}\dbar_{k}+\dbar_{k}\dbarstar_{k}\ \ \mbox{on ${\rm Dom\,}\Box^q_{k}$},\\
& {\rm Dom\,}\Box^q_{k}=\set{f\in{\rm Dom\,}\dbar_{k}\cap{\rm Dom\,}\dbarstar_{k} \,\big|\, \dbar_{k}f\in{\rm Dom\,}\dbarstar_{k}, \dbarstar_{k} f\in{\rm Dom\,}\dbar_{k}}. 
\end{split}
\end{equation}
Define the space of harmonic $(0,q)$-forms as  
\begin{equation}  
\mathscr{H}^q(\ol M,L^k) := \ker \Box_k^q,  
\end{equation}  
consisting of harmonic forms smooth up to the boundary that satisfies the $\dbar$-Neumann boundary conditions. 
According to Kohn's estimates \cite{K65} (see also \cite{FK72}), 
the subellipticity of $\Box^q$ and $\Box_k^q$ is ensured under condition $Z(q)$.

We now fix the convention for Sobolev norms. For each $k$, let $\langle \,\cdot\,,\,\cdot\,\rangle_k$
denote the Hermitian metric on
$\Lambda^{0,q}T^*M' \otimes L^k$. Notice that this metric may depend on $k$ not only through the tensor power $L^k$, but also through Hermitian metric on $T^{*0,q}M'$. The volume form $dv_{M'}$ is fixed throughout. Choose a smooth connection $\nabla^{\Lambda}$ on
$\Lambda^{0,q}T^*M'$, and a smooth connection $\nabla^L$ on $L$. They induce, for each $k$, a smooth connection $\nabla_k$ on $\Lambda^{0,q}T^*M' \otimes L^k$. For $m\in\mathbb N_0$, we define
\begin{equation}\label{sobolevnorm}
\|f\|^2_{m,k,\overline M}
:=\sum_{\ell=0}^m
\int_{\overline M}
\left|\nabla_k^\ell f(x)\right|_k^2\,dv_{M'}({\bf z}),
\qquad f\in\Omega^{0,q}(\overline M,L^k).    
\end{equation}
When $m=0$, we simply write
$\|f\|_{k,\overline M}$. 
The following subelliptic estimates hold.
\begin{thm}[\cite{K65, FK72}]\label{t-gue210228yyd}  
Let $m\in\mathbb N_0$. If condition $Z(q)$ holds on $M$, then for every
$k$ there exists a constant $C_{m,k}>0$ such that
\begin{equation}\label{be-gue210228yydI}  
\|f\|_{m+1,k,\overline M}
\leq C_{m,k}\left(\|\Box_k^q f\|_{m,k,\overline M}
+\|f\|_{k,\overline M}\right),
\quad
\forall f\in\Omega^{0,q}(\overline M,L^k).
\end{equation}  
\end{thm}
It follows that under condition $Z(q)$, $\Box_k^q$ has compact resolvent and the strong Hodge decomposition holds. In particular, we have  
$$
\dim \mathscr{H}^q(\ol M,L^k) < \infty \quad \text{and} \quad \mathscr{H}^q(\ol M,L^k) \cong H^q(\ol M,L^k).
$$

\subsection{Heat kernel distribution}\label{s-gue230213yyd}

Since $\Box^q_{k}$ is non-negative and self-adjoint, the heat operator 
$$
e^{-t\Box^q_{k}}: L^2_{(0,q)}(M,L^k)\to{\rm Dom\,}\Box^q_{k}
$$ 
for $t>0$ exists. Let 
$$A_k(t,\mathbf z, \mathbf w):=e^{-t\Box^q_{k}}(\mathbf z,\mathbf w)\in\mathscr D'(\mathbb R_+\times M\times  M, (T^{*0,q}M'\otimes L^k)\boxtimes(T^{*0,q}M'\otimes L^k)^*)$$
be the distribution kernel of $e^{-t\Box^q_{k}}$ with respect to $dv_{M'}$, i.e.,
$$
e^{-t\Box^q_{k}}f(\mathbf z)=\int_{M}A_k(t,\mathbf z, \mathbf w)f(\mathbf w)dv_{M'}(\mathbf w)\quad \text{for $f\in \omz^{0,q}_c(M,L^k)$}.
$$
We also write $A_k(t):=e^{-t\Box^q_{k}}$.
For $f\in \omz^{0,q}_c(M,L^k)$, $A_k(t)$ satisfies the following 
\begin{equation}\label{be-gue210531ycda}
\begin{dcases}
A_k(t)f\in \Dom(\Box^q_{k}),\\
A_k(t)f \mbox{ is differentiable in $t$},\\
\left(\dfrac{\pa}{\pa t}+\Box^q_{k}\right)A_k(t)f=0,\\
\lim_{t\to0}A_k(t)f=f\ \ \mbox{in $L^2_{(0,q)}(M,L^k)$}.\\
\end{dcases}
\end{equation}
Now, assume that $Z(q)$ holds. On the diagonal, by the canonical identification ${\rm End\,}(L^k_{\mathbf z})=\mathbb C$, we have
\begin{equation}\label{be-gue210303yyd}
A_k(t,\mathbf z, \mathbf z)\in\cali{C}^\infty(\mathbb R_+\times \ol M, {\rm End\,}(T^{*0,q}M')).
\end{equation}
 For any $p\in \ol M$, let $s$ be a local holomorphic trivializing section of $L$ defined on $U$, a small open neighborhood of $p$ in $M'$, such that $\abs{s}^2_{h^L}=e^{-\phi}$. Set $D:=U\cap M$. When $p\in X$, let $\partial D:=U\cap X$ and $\ol D:=U\cap \ol M$.
Notice that we have the following unitary identification:
\[\begin{split}
U: L^2_{(0,q)}(D,L^k)&\longleftrightarrow L^2_{(0,q)}(D),\\
s^k\otimes f&\longleftrightarrow e^{-\frac{k\phi}{2}} f.
\end{split}\]
Hence there exists $A_{k,s}(t,\mathbf z,\mathbf w)\in\cali{C}^\infty(\mathbb R_+\times \ol D\times \ol D, T^{*0,q}M'\boxtimes(T^{*0,q}M')^*)$ such that for every $\hat f=s^k f\in\Omega^{0,q}_c(D,L^k)$ with $f\in\Omega^{0,q}_c(D)$, we have 
\begin{equation}\label{be-gue210303yydI}
(e^{-t \Box^q_{k}}\hat f)(\mathbf z)=s^k(\mathbf z)  e^{\frac{k\phi(\mathbf z)}{2}}\int_DA_{k,s}(t,\mathbf z, \mathbf w)e^{-\frac{k\phi(\mathbf w)}{2}} f(\mathbf w)dv_{M'}(\mathbf w),
\end{equation}
which implies 
\begin{equation}\label{be-gue210303yydII}
A_k(t,\mathbf z, \mathbf w)=s^k(\mathbf z)  A_{k,s}(t,\mathbf z, \mathbf w)e^{\frac{k\phi(\mathbf z)-k\phi(\mathbf w)}{2}} s^k(\mathbf w)^*. 
\end{equation}
In particular, 
\begin{equation}\label{be-gue210303yydIII}
\mbox{$e^{-t \Box^q_{k}}(\mathbf z,\mathbf z)=A_k(t,\mathbf z,\mathbf z)=A_{k,s}(t,\mathbf z,\mathbf z)$ on $\ol D$}.
\end{equation}
We also notice that 
\begin{equation}\label{be-gue210303yydIV}
\lim_{t\to0}\int A_{k,s}(t,\mathbf z, \mathbf w)e^{\frac{k\phi(\mathbf z)-k\phi(\mathbf w)}{2}} f(\mathbf w)dv_{M'}(\mathbf w)= f(\mathbf z)
\end{equation}
in $L^2_{(0,q)}(D)$, for every $ f\in\Omega^{0,q}_c( D)$.
 Assume that $\ol\omega^1,\cdots,\ol\omega^{n}$ is an orthogonal basis for $T^{*0,1}M'$ on $U$.
For a fixed $0\le q\le n$, note that the set
$\{\overline\omega^J(\mathbf z) ,  J=(j_1,\ldots,j_q), 1\leq j_1<\cdots<j_q\leq n\}$ 
forms an orthogonal frame for $T^{*0,q}_{\mathbf z}M$, for every $\mathbf z\in D$. Then locally, we write 
\begin{equation}\label{be-gue230306yydII}
A_{k,s}(t,\mathbf z,\mathbf w):=\sumprime_{|I|=|J|=q}A_{k,s,I,J}(t,\mathbf z,\mathbf w)\,\ol\omega^I(\mathbf z)\otimes\ol\omega^J(\mathbf w)    
\end{equation}
in the sense that for $f=\sumprime_{|J|=q}f_J\ol\omega^J\in\omz^{0,q}_c(D)$,
\begin{equation*}
(A_{k,s}(t)f)(\mathbf z)=\sumprime_{|I|=|J|=q}\int A_{k,s,I,J}(t,\mathbf z,\mathbf w)f_J(\mathbf w)\langle\,\ol\omega^J(\mathbf w)\mid\ol\omega^J(\mathbf w)\,\rangle_k dv_{M'}(\mathbf w)\,\ol\omega^I(\mathbf z).
\end{equation*}

Since the $\dbar$-Neumann Laplacian $\Box^q_{k}$ is self-adjoint and non-negative, we can apply the spectral theorem (see \cite{Da95}) to identify $L^2_{(0,q)}(M,L^k)$ with $L^2(\mathbb{S} \times \mathbb{N}, d\mu)$, where $\mathbb{S}$ is the spectrum of $\Box^q_{k}$ and $\mu$ is a regular Borel measure on $\mathbb{S} \times \mathbb{N}$. Under this identification, $\Box^q_{k}$ corresponds to the multiplication operator $M_s$, and the heat operator $e^{-t \Box^q_{k}}$ acts as  
\[
\begin{split}  
e^{-t \Box^q_{k}}: L^2(\mathbb{S} \times \mathbb{N}, d\mu) &\to L^2(\mathbb{S} \times \mathbb{N}, d\mu), \\  
\varphi(s, {n}) &\mapsto e^{-ts} \varphi(s, {n}).  
\end{split}  
\] 
For the reader's convenience, we record the following standard consequence of the spectral theorem, which will later play a key role in obtaining uniform estimates in Proposition \ref{uniformly} in Section \ref{scaledL}: 
\begin{lem}\label{l-gue210331yyd}  
For every $N \in \mathbb{N}_0$ and $f \in \Omega^{0,q}(\ol{M}, L^k) \cap \Dom((\Box^q_k)^N)$, we have
\begin{equation}\label{e-gue210331yyd}  
\norm{e^{-\frac{t}{k} \Box^q_{k}} ( \Box^q_{k})^N f}_{k,\ol M}  
= \norm{( \Box^q_{k})^N e^{-\frac{t}{k} \Box^q_{k}} f}_{k,\ol M}  
\leq \left(1 + \frac{k^N}{t^N}\right) C_N \norm{f}_{k,\ol M},  
\end{equation}  
for all $t>0$ and $k\geq1$, 
where $C_N > 0$ is a constant independent of $k$ and $t$.  
\end{lem}

\begin{proof}  
Let $N \in \mathbb{N}_0$. For $f \in \Omega^{0,q}(\ol{M}, L^k) \cap \Dom(\Box^q_k)$, we identify it with $\varphi(s, {n}) \in L^2(\mathbb{S} \times \mathbb{N}, d\mu)$. Then
\[
\begin{split}  
\|e^{-\frac{t}{k} \Box^q_{k}} ( \Box^q_{k})^N f\|^2_{k}  
&= \|( \Box^q_{k})^N e^{-\frac{t}{k} \Box^q_{k}} f\|^2_{k}= \int \abs{e^{-\frac{t}{k}s} s^N \varphi(s, {n})}^2 d\mu \\  
&\leq \int_{\{0 \leq s < 1\}} \abs{\varphi(s, {n})}^2 d\mu  
+ \hat{C}_N \int_{\{s \geq 1\}} \abs{\frac{k^N}{t^N s^N} s^N \varphi(s, {n})}^2 d\mu \\  
&\leq \left(1 + \hat{C}_N \frac{k^{2N}}{t^{2N}}\right) \int \abs{\varphi(s, {n})}^2 d\mu,  
\end{split}  
\] 
where $\hat{C}_N > 0$ is a constant independent of $k$ and $t$. The lemma follows.  
\end{proof}

%Our main goal of this work is to study the asymptotic behavior of $A_k(\frac{t}{k},\mathbf z,\mathbf z)$ as $k\to\infty$ when condition $Z(q)$ holds. 

%We will divide the manifold $M$ into two regions: the interior region $\dot{M}$ and the boundary region. In the interior region, $\Box$ is elliptic and we have the G\aa{}rding's inequality. The asymptotic behavior of $A_k(\frac{t}{k},z,z)$ as $k\to\infty$ is well-established and can be found in \cite[Theorem 1.5]{B87} or \cite[Theorem 1.6.1]{MM07}. In this paper, we will provide a version of proof using the scaling technique in section \ref{estimate}. For the boundary region, we will use 

\section{The heat equation of the weighted \texorpdfstring{$\dbar$}{bar}-Neumann 
Laplacian on model domains}\label{bdd}

In this section, we analyze the model operators that govern the heat kernel asymptotics. Two model geometries appear naturally in our problem. The first one is the flat model on $\C^n$, which describes the interior regime away from the boundary. The second one is the boundary model on a Siegel-type domain $M_0$, which arises after anisotropic rescaling near a boundary point. Our goal is to compute the corresponding heat kernels explicitly, with particular emphasis on the boundary model, since this is the new ingredient needed in the proof of the main asymptotic theorem.

We begin with the flat model on $\C^n$. In this case the weighted Laplacian is given by
\begin{equation}\label{modelLa}
\mathring{\Box}_{\phi_0}^{q}=\sum_{j=1}^{n}\Big(-\frac{\pa}{\pa z_j}+\sum_{l=1}^{n}\mu_{j,l}\ol  z_l\Big)\frac{\pa}{\pa \ol z_j}+\sum_{j,l=1}^{n}\mu_{j,l}d\ol z_j\wedge (d\ol z_l\wedge)^*,
\end{equation}
where $\phi_0$ is as defined in \eqref{mphi0}. Its heat kernel is classical. For later comparison with the boundary model, we record the diagonal formula
\begin{equation}\label{model0}
e^{-t\mathring{\Box}_{\phi_0}^q}({\bf z},{\bf z})=\frac{1}{(2\pi)^{n}}\frac{\det(\dot{R}^0)e^{-t\varTheta^0}}{\det(1-e^{-t\dot{R}^0})},
\end{equation}
where $\dot{R}^{0}$ and $\varTheta^{0}$ are given by \eqref{Reta} and \eqref{omega} with $\eta=0$. The corresponding off-diagonal expression follows from Mehler's formula; see \cite[Theorem 1.5]{Bi87} and \cite[Appendix~E.2]{MM07}.

We now pass to the boundary model. In contrast with the interior case, the appropriate scaling near the boundary leads to the domain
\begin{equation}
 \ol M_0 = \big\{{\bf z}=(z, z_{n}) \in \C^{n} : \im z_{n}+\sum_{j=1}^{n-1} \lambda_j |z_j|^2\le 0 \big\}.
\end{equation}
Here the heat kernel reflects both the Levi geometry of the boundary and the weighted potential. The next theorem gives the explicit diagonal formula that will be used later in the proof of the boundary scaling limit. Together with the off-diagonal formula (Theorem \ref{heatkernelmodel}) established later in this section, it provides the model kernel appearing in Theorem~\ref{main1}.

\begin{thm}\label{hmd}
Let 
$$
\ol M_0=\Big\{(z,z_{n})\in \C^{n-1}\times \C: \im z_{n}+\sum_{j=1}^{{n-1}}\lambda_j|z_j|^2\le 0\Big\}
$$
with boundary 
$$
X_0=\big\{(z,z_{n})\in \C^{n}: \im z_{n}+\sum_{j=1}^{{n-1}}\lambda_j|z_j|^2=0\big\}.
$$
Then the heat kernel of the weighted $\dbar$-Neumann Laplacian $\Box_{\phi_0}^q$ as defined in \eqref{mboxrhok} is given on the diagonal by
\begin{equation}\label{modelasymoptptic}
\begin{aligned}
e^{-t\Box_{\phi_0}^q}([x,r],[x,r])
&=\frac{1+e^{-2r^2/t}}{(2\pi)^{n}\sqrt{2\pi t}}\lim_{\delta\to \infty}\int_{-\delta}^\delta e^{-t\eta^2/2}\dfrac{\det\dot R^\eta}{\det\big(1-e^{-t \dot R^{\eta}}\big)}
e^{-t\varTheta^{\eta,\nu}}d\eta\\
&\quad  +\frac{2}{(2\pi)^{n} \sqrt{\pi}}\lim_{\delta\to \infty}\int_{-\delta}^\delta \eta e^{-2r\eta} \dfrac{\det\dot R^\eta}{\det\big(1-e^{-t \dot R^{\eta}}\big)}e^{-t\varTheta^{\eta,\tau}}\Big(\int^{\frac{2r-t\eta}{\sqrt{2t}}}_{-\infty}e^{-\gamma^2}d\gamma\Big) d\eta\\
&\quad  +\frac{1-e^{-2r^2/t}}{(2\pi)^{n}\sqrt{2\pi t}}\lim_{\delta\to \infty}\int_{-\delta}^\delta e^ {-t\eta^2/2}\dfrac{\det\dot R^\eta}{\det\big(1-e^{-t \dot R^{\eta}}\big)}
e^{-t\varTheta^{\eta,\nu}}d\eta,
\end{aligned}
\end{equation}
where $x=(z,\re z_{n})$, $r=\im z_{n}+\sum_{j=1}^{{n-1}}\lambda_j|z_j|^2$, $\dot{R}^{\eta}$ and $\varTheta^{\eta}$ are given by \eqref{Reta} and \eqref{omega}, respectively.
\end{thm}
\begin{rem}\label{rem:modelasym}
Assume that condition $Z(q)$ holds. Then the above integrals in $\eta$ converge on $\R$; see Remark~\ref{rem:zq} and the proof of Theorem~\ref{main1}.
\end{rem}

The proof occupies the remainder of this section. We first derive explicit formulas for the fundamental solutions of $\Box_{\phi_0}^q$ on $M_0$, and then add suitable correction terms in order to enforce the $\dbar$-Neumann boundary condition. The construction combines ideas from \cite{S80,T97,HZ23}, but the present setting is more general: we allow mixed Levi signature and a nontrivial weight $\phi_0$, so the model is no longer confined to the strongly pseudoconvex or unweighted cases.

\subsection{Heisenberg coordinates}\label{M0} 

Consider $\C^{n}$ with coordinates $\mathbf z=(z,z_{n})$, where $z\in\C^{n-1}$ and $z_{n}\in\C$.
Let 
$$
X_0=\big\{(z,z_{n})\in \C^{n}: \text{Im}z_{n}=-\sum_{j=1}^{{n-1}}\lambda_j|z_j|^2\big\}
$$
be the boundary of $M_0$. We can identify $X_0$ with the Heisenberg group $\mathbb H_{n-1}=\C^{n-1}\times\R$ equipped with the group law 
$$
(z,z_{n})(z',z'_{n})=(z+z',z_{n}+z'_{n}-2i\sum_{j=1}^{{n-1}}\lambda_j z_j\ol z'_j).
$$
More precisely, we associate each $[\xi,\theta]\in\mathbb H_{n-1}$  a holomorphic affine self-mapping of $M_0$:
$$
T_{[\xi,\theta]}:(z,z_{n})\to\big(z+\xi,z_{n}+\theta-i\sum_{j=1}^{n-1}\lambda_j(|\xi_j|^2+2z_j\ol \xi_j)\big).
$$
Observe that $T_{[\xi,\theta]}$ preserves the defining function 
$
r_0(z,z_{n})=\text{Im}z_{n}+\sum_{j=1}^{{n-1}}\lambda_j|z_j|^2,
$
and hence preserves the boundary $X_0$. %Moreover, the mappings are simply transitive on $X_0$.
Now we set
$$
\xi=z,\quad \theta=\operatorname{Re} z_n,\quad r=\operatorname{Im} z_n+\sum_{j=1}^{n-1}\lambda_j|z_j|^2.
$$
It is convenient to use the Heisenberg coordinates $[\xi,\theta,r]$ instead of the ambient coordinates $(z,z_n)$. Here, $r$ can be viewed as the signed height of $(z,z_n)\in M_0$ relative to the boundary, and $[\xi,\theta]$ represents its projection onto $X_0$. Via the identification
$$
\mathbb H_{n-1}\ni [\xi,\theta]\longmapsto \Big(\xi,\theta-i\sum_{j=1}^{n-1}\lambda_j|\xi_j|^2\Big)\in X_0,
$$
we identify $X_0$ with $\mathbb H_{n-1}$. In these coordinates, $M_0$ and $X_0$ take the form
$$
M_0 = \{ [\xi, \theta, r] \in \mathbb{H}_{n-1} \times \mathbb{R} \mid r < 0 \} \quad \text{and} \quad X_0 = \{ [\xi, \theta, r] \in \mathbb{H}_{n-1} \times \mathbb{R} \mid r = 0 \}.
$$
The vector fields
\begin{equation*}
\begin{aligned}
&\overline{Z}_j = \frac{\partial}{\partial \overline{\xi}_j} + i\lambda_j \xi_j \frac{\partial}{\partial \theta}, \quad j = 1, \ldots, {n-1}, \\
&\overline{Z}_{n} = \frac{1}{\sqrt 2} \left( \frac{\partial}{\partial r} - i\frac{\partial}{\partial \theta} \right)
\end{aligned}
\end{equation*}
form a basis for $ T^{0,1}\C^n $, where $\lambda_j \in \mathbb{R}$, $j = 1, \ldots, {n-1}$. Noting that $\frac{\partial}{\partial \overline{\xi}_j} = \frac{\partial}{\partial \overline{z}_j} - \lambda_j z_j \frac{\partial}{\partial r}$, we have $\frac{\partial r_0}{\partial \overline{\xi}_j} = 0$. Hence $\overline{Z}_1, \ldots, \overline{Z}_{n-1}$ are tangent to the level surface of $ r_0 $, while $\overline{Z}_{n}$ is a complex normal vector field. Furthermore, setting $\lambda_{n}:=0$, we have
\begin{equation}\label{mlocal2}
\begin{aligned}
[Z_j,\ol Z_l]=2i\lambda_j\delta_{jl}\frac{\pa}{\pa\theta},\quad
[Z_j,Z_l]=[\ol Z_j,\ol Z_l]=0,
 \end{aligned}
\end{equation}
for $j,l=1,\cdots,n$. The $(1,0)$-forms basis dual to these vector fields are then given by
\begin{equation}\label{mlocal3}
\begin{aligned}
&\ol\omega_0^j=d\ol \xi_j, \quad j=1,\cdots,{n-1},\\
&\ol\omega_0^{n}=\sqrt 2\, \dbar r(=\frac{i}{\sqrt 2}d\ol z_{n}+\sqrt{2}\sum_{j=1}^{{n-1}}\lambda_jz_jd\ol z_j).
\end{aligned}
\end{equation}
We give $\C^{n}$ the invariant Hermitian metric $\langle\,\cdot\, |\,\cdot\,\rangle_{\C^n}$ for which $\set{Z_1,\cdots,Z_{n}}$  forms orthonormal basis for $T^{1,0}\mathbb C^n$ and $\set{\omega_0^1,\cdots,\omega_0^{n}}$ forms orthonormal basis for the cotangent bundle space $T^{*1,0}\mathbb C^n$. Let 
\begin{equation}\label{mphi0}
\begin{aligned}
\phi_0(z)=\sum_{j,l=1}^{{n-1}}\mu_{j,l}\xi_j\ol \xi_l,
\end{aligned}
\end{equation}
where $\mu_{j,l}\in\R$, $\mu_{j,l}=\overline{\mu_{l,j}}$, $j,l=1,\cdots, {n-1}$.

The Cauchy-Riemann operator $\dbar$ on $\C^{n}$ is given by 
\begin{equation}\label{mdbarbh}
\begin{aligned}
\dbar=\sum_{j=1}^{n}\ol\omega_0^j\wedge \ol Z_j:\omz^{0,q}(\C^{n})\to \omz^{0,q+1}(\C^{n}).
\end{aligned}
\end{equation}
Let $(\,\cdot \,, \,\cdot\,)_{\phi_0}$ be the  inner product on $\omz^{0,q}(\ol M_0)$ with weight $\phi_0$. Namely,
\begin{equation}\label{mbe-gue210506yyd}
\begin{aligned}
(\,f ,  g\,)_{\phi_0}=\int_{M_0}\langle\,f \mid  g\,\rangle_{\C^n}e^{-\phi_0}dv(\mathbf z), \quad f,g\in\omz^{0,q}(\ol M_0),
\end{aligned}
\end{equation}
where the volume form $dv(\mathbf z)=i^ndz_1d\overline z_1\cdots dz_nd\overline z_n$, is $2^{n}$ times the standard Euclidean volume form. Let $L^2_{(0,q)}(M_0,\phi_0)$ be the completion of $\omz^{0,q}(\ol M_0)$ with respect to $(\,\cdot \,, \,\cdot\,)_{\phi_0}$. 
%Let $\|\cdot\|_{\phi_0}$ be the corresponding norm. Let $W\subset M_0$ be an open set. For $f\in L^2_{(0,q)}(M_0,\phi_0)$, let 
%$$
%\|f\|^2_{\phi_0,W}:=\int_W\langle\,u ,  u\,\rangle_{M_0}e^{-\phi_0}dv(z).
%$$
We extend $\dbar$ to $L^2_{(0,q)}(M_0,\phi_0)$ as in \eqref{e-gue260130yyd} and still denote by $\dbar$. Let $\dbar^{*,\phi_0}$ be the adjoint of $\dbar$ with respect to $(\,\cdot \,, \,\cdot\,)_{\phi_0}$. The (Gaffney extension) of $\dbar$-Neumann Laplacian $\Box^q_{\phi_0}$ is given by
\[\begin{split}
&\Box^q_{\phi_0}=\dbar^{*,\phi_0}\dbar+\dbar\,\dbar^{*,\phi_0}: L^2_{(0,q)}(M_0,\phi_0)\to L^2_{(0,q)}(M_0,\phi_0)\\
& {\rm Dom\,}\Box^q_{\phi_0}=\{f\in L^2_{(0,q)}(M_0,\phi_0) ,  f\in{\rm Dom\,}\dbar\cap{\rm Dom\,}\dbar^{*,\phi_0}, \dbar f\in{\rm Dom\,}\dbar^{*,\phi_0}, \dbar^{*,\phi_0} f\in{\rm Dom\,}\dbar\}. 
\end{split}\]
Given a $(0,q)$-form 
$$
f=\sumprime_{J}f_J\ol\omega_0^J,
$$
we write its tangential part $f^\tau$ and normal part $f^\nu$ as:  
\begin{equation}\label{tannor}
f^\tau = \sumprime_{n \notin J} f_J \ol\omega^J \quad \text{and} \quad f^\nu = \sumprime_{n \in J} f_J \ol\omega^J,
\end{equation}
where $f^\tau$ consists of terms without $\ol\omega^{n}$, and $f^\nu$ includes terms containing $\ol\omega^{n}$. One can check that $f\in\omz^{0,q}(\ol M_0)\cap \Dom{\Box_{\phi_0}^q}$ if and only if
\begin{equation}\label{mbdc}
\begin{aligned}
f_J=0&\quad \text{on $X_0$, \  if $n\in J$},\\
\ol Z_{n} f_J=0 &\quad \text{on $X_0$, \ if $n\notin J$}.\\
%Z_j(f_J)=\ol Z_j(f_J)=0 &\quad \text{on $X_0$, if $j<n$ and $n\in J$}.
\end{aligned}
\end{equation}
Moreover, a direct calculation shows that
\begin{equation}\label{mboxrhok}
\begin{aligned}
\Box_{\phi_0}^{q}=&\sum_{j=1}^{{n-1}}\Big(-\frac{\pa}{\pa \xi_j}+i\lambda_j\ol \xi_j\frac{\pa}{\pa  \theta}+\sum_{l=1}^{n-1}\mu_{j,l}\ol  \xi_l\Big)\Big(\frac{\pa}{\pa \ol \xi_j}+i\lambda_j \xi_j\frac{\pa}{\pa \theta}\Big)-\frac{1}{2}\Big(\frac{\pa^2}{\pa\theta^2}+\frac{\pa^2}{\pa r^2}\Big)\\
&+\sum_{j,l=1}^{{n-1}}
\Big(2i\lambda_j\frac{\pa}{\pa \theta}+\mu_{l,j}\Big)d\ol \xi_j\wedge \big( d\ol \xi_l\wedge\big)^*.
\end{aligned}
\end{equation}
Let $e^{-t\Box^q_{\phi_0}}$, $t>0$, be the heat operator for $\Box^q_{\phi_0}$. Let 
$e^{-t\Box^q_{\phi_0}}(\mathbf z,\mathbf w)$
be the distribution kernel of $e^{-t\Box^q_{\phi_0}}$ with respect to $(\,\cdot \,, \,\cdot\,)_{\phi_0}$, i.e.,
$$(e^{-t\Box^q_{\phi_0}}u)(\mathbf z)=\int e^{-t\Box^q_{\phi_0}}(\mathbf z,\mathbf w)f(\mathbf w)dv(\mathbf w),\ \ f\in\Omega^{0,q}_c(M_0).$$
If condition $Z(q)$ holds on $M_0$, we have
$$e^{-t\Box^q_{\phi_0}}(\mathbf z,\mathbf w)\in\cali{C}^\infty(\mathbb R_+\times \ol M_0\times \ol M_0,T^{*0,q}\C^{n}\boxtimes(T^{*0,q}\C^{n})^*).$$
For notation convenience, we write $A_{\phi_0}(t,\mathbf z, \mathbf w):=e^{-t\Box^q_{\phi_0}}(\mathbf z,\mathbf w)$ and
$$
A_{\phi_0}(t,\mathbf z, \mathbf w):=\sumprime_{|I|=|J|=q}A_{\phi_0,I,J}(t,\mathbf z, \mathbf w)\,\ol\omega_0^I(\mathbf z)\otimes\ol\omega_0^J(\mathbf w).
$$
It is clear that 
\begin{equation}\label{heatmodel}
\begin{aligned}
&\Big(\frac{\pa}{\pa t}+\Box^q_{\phi_0}\Big)A_{\phi_0}(t)f= 0,
\mbox{\quad $A_{\phi_0}(t)f\in\dom \Box^q_{\phi_0}$},\\
&\mbox{\qquad$\lim_{t\to0}A_{\phi_0}(t)f=f$ in $L^2_{(0,q)}(M_0)$.}
\end{aligned}
\end{equation}

\subsection{The fundamental solution of \texorpdfstring{$\Box_{\phi_0}^q$}{Box}}  

In this subsection, we begin by finding a kernel $K(t, {\bf z}, {\bf w})$ that solves the following initial value problem (without imposing boundary conditions):  
\begin{equation}\label{hh1}  
\begin{cases}  
\left(\frac{\pa}{\pa t} + \Box^q_{\phi_0}\right) K(t) f = 0, \\  
\displaystyle\lim_{t \to 0} K(t) f = f, 
\end{cases}  
\end{equation}  
where $f\in\omz^{0,q}_c(M_0)$. We rewrite  \eqref{mboxrhok} as 
\begin{equation}
\Box^q_{\phi_0}=\Box_{b,\phi_0}^q-\frac{1}{2}\left(\frac{\pa^2}{\pa\theta^2}+\frac{\pa^2}{\pa r^2}\right),   
\end{equation}
where the operator $\Box_{b,\phi_0}^q$, called the Kohn Laplacian on $X_0$, is defined by 
\begin{equation*}
\begin{aligned}
\Box_{b,\phi_0}^q:=&\sum_{j=1}^{{n-1}}\Big(-\frac{\pa}{\pa \xi_j}+i\lambda_j\ol \xi_j\frac{\pa}{\pa  \theta}+\sum_{l=1}^{n-1}\mu_{j,l}\ol  \xi_l\Big)\Big(\frac{\pa}{\pa \ol \xi_j}+i\lambda_j \xi_j\frac{\pa}{\pa \theta}\Big)\\
&+\sum_{j,l=1}^{{n-1}}
\Big(2i\lambda_j\frac{\pa}{\pa \theta}+\mu_{l,j}\Big)d\ol \xi_j\wedge \big( d\ol \xi_l\wedge\big)^*.
\end{aligned}
\end{equation*}
The Hermitian metric $\langle\,\cdot\,|\,\cdot\,\rangle_{\C^n}$ induces Hermitian metrics $\langle\,\cdot\,|\,\cdot\,\rangle_{\mathbb H_{n-1}}$
on $\mathbb CT\mathbb H_{n-1}$ and on the bundle $\oplus^{2n-1}_{q=1}\Lambda^q(\mathbb CT^*\mathbb H_{n-1})$. Let $T^{*0,q}\mathbb H_{n-1}$ be the bundle of $(0,q)$-forms on $\mathbb H_{n-1}$ and let 
$\Omega^{0,q}(\mathbb H_{n-1}):=\cali{C}^\infty(\mathbb H_{n-1},T^{*0,q}\mathbb H_{n-1})$, $\Omega^{0,q}_c(\mathbb H_{n-1}):=\cali{C}^\infty_c(\mathbb H_{n-1},T^{*0,q}\mathbb H_{n-1})$. 
Let $dv_{\mathbb H_{n-1}}$ be the volume form on $\mathbb H_{n-1}$ induced by $\langle\,\cdot\,|\,\cdot\,\rangle_{\mathbb H_{n-1}}$. Let $(\,\cdot\,,\,\cdot\,)_{\mathbb H_{n-1},\phi_0}$ be the $L^2$ inner product on $\Omega^{0,q}_c(\mathbb H_{n-1})$ induced by $\langle\,\cdot\,|\,\cdot\,\rangle_{\mathbb H_{n-1}}$ and $e^{-\phi_0}dv_{\mathbb H_{n-1}}$. Let $L^2_{(0,q)}(\mathbb H_{n-1},\phi_0)$ be the completion of $\Omega^{0,q}_c(\mathbb H_{n-1})$ with respect to $(\,\cdot\,,\,\cdot\,)_{\mathbb H_{n-1},\phi_0}$.

We begin by calculating the heat kernel of $\Box_{b,\phi_0}^q$. According to \cite[Theorem 4.3]{HZ23},  we claim that the heat kernel of $\Box_{b,\phi_0}^q$ can be expressed as:
\begin{equation}\label{be-gue210523yydI}
\begin{split}
e^{-t\Box_{b,\phi_0}^q}(x,y)=\frac{1}{2\pi}\lim_{\delta\to \infty}\int_{-\delta}^{\delta} e^{i<\theta-\sigma,\eta>+\phi_0(\xi)/2-\phi_0(\zeta)/2}e^{-t\Box_\eta}(\xi,\zeta)d\eta,
\end{split}
\end{equation}
where $x=[\xi,\theta]\in\mathbb H_{n-1}$, %=(x_1,\ldots,x_{2n},x_{2n-1}) 
$y=[\zeta,\sigma]\in\mathbb H_{n-1}$, $\eta\in\mathbb R$,
%=(y_1,\ldots,y_{2n},y_{2n-1})
and $e^{-t\Box_\eta}(\xi,\zeta)$ is given by
\begin{equation}\label{be-gue210521yyd}
\begin{aligned}
e^{-t\Box_\eta}(\xi,\zeta)=&\frac{1}{(2\pi)^{n-1}}\dfrac{\det\dot R^\eta}{\det\big(1-e^{-t \dot R^\eta}\big)}
\exp\bigg\{-t\varTheta^\eta_0-\frac{1}{2}\Big\langle\,\frac{\dot R^\eta/2}{\tanh(t\dot R^\eta/2)}\xi\mid \xi\,\Big\rangle_{\mathbb C^{n-1}}\\
&-\frac{1}{2}\Big\langle\,\frac{\dot R^\eta/2}{\tanh(t\dot R^\eta/2)}\zeta\mid \zeta\,\Big\rangle_{\mathbb C^{n-1}}
+\Big\langle\,\frac{\dot R^\eta/2}{\sinh(t\dot R^\eta/2)}e^{t\dot R^\eta/2}\xi\mid \zeta\,\Big\rangle_{\mathbb C^{n-1}}\bigg\}.
\end{aligned}
\end{equation}
For $\eta\in\mathbb R$, $j, l=1,\ldots,{n-1}$, set
\begin{equation}\label{Phieta}
\Phi_{\eta}:=-2\eta\sum_{j=1}^{{n-1}}\lambda_j|\xi_j|^2+\sum_{j,l=1}^{{n-1}}\mu_{j,l}\xi_j\ol \xi_l.
\end{equation} 
From \eqref{mphi0}, we have that $\Phi_{0}=\phi_0$. In \eqref{be-gue210521yyd}, $\dot{R}^\eta:T^{1,0}\C^{{n-1}}\to T^{1,0}\C^{{n-1}}$ represents the linear map defined by
\begin{equation}\label{Reta}
\langle\,\dot{R}^\eta U\mid\ol V\,\rangle_{\mathbb C^{n-1}}=\pa\dbar\Phi_{\eta}(U,\ol V),\quad U,V\in T^{1,0}\C^{n-1}
\end{equation}
and 
\begin{equation}\label{omega}
\varTheta^\eta:=\sum_{j,l=1}^{n-1}\frac{\pa^2\Phi_{\eta}}{\pa\ol \xi_j\pa \xi_l}d\ol \xi_j\wedge (d\ol \xi_l\wedge)^*.
\end{equation}
For the computation of \eqref{be-gue210521yyd}, we refer the reader to \cite[Section 4]{HZ23} for details. Now, we provide a proof of \eqref{be-gue210523yydI}. %as we replace the condition $Y(q)$ used in \cite{HZ23} with the weaker condition $Z(q)$. 
%%%%%%%%%%
Let $\mathcal{S}^{0,q}(\mathbb H_{n-1})$ be the space of Schwartz test $(0,q)$-forms. Let 
\begin{equation}\label{e-gue210518yydI}
\begin{split}
G: \Omega^{0,q}_c(\mathbb H_{n-1})&\to\mathcal{S}^{0,q}(\mathbb H_{n-1}),\\
u(\xi,\theta)&\mapsto\int_{\mathbb R}u(\xi,\theta)e^{-i\theta\eta-\phi_0(z)/2}d\theta,
\end{split}
\end{equation}
where $\eta\in\mathbb R$ is the Fourier dual variable to $\theta$. 
By a straightforward calculation, we have 
$$
G(\Box^q_{b,\phi_0}u)=\Box_\eta(Gu),\ \mbox{for $u\in\Omega^{0,q}_c(\mathbb H_{n-1})$,}
$$
where $\Box_\eta$ is defined as in \cite[(4.2)]{HZ23}, with $n$ replaced by $n-1$. On $\mathbb H_{n-1}$, let $T=-\frac{\partial}{\partial\theta}$ and consider 
$$-iT: {\rm Dom\,}(-iT)\subset L^2_{(0,q)}(\mathbb H_{n-1},\phi_0)\to L^2_{(0,q)}(\mathbb H_{n-1},\phi_0),$$
where ${\rm Dom\,}(-iT)=\set{u\in L^2_{(0,q)}(\mathbb H_{n-1},\phi_0)\mid -iTu\in L^2_{(0,q)}(\mathbb H_{n-1},\phi_0)}$. It is not difficult to see that $-iT$ is self-adjoint. For $\delta_1<\delta_2$, $\delta_1, \delta_2\in\mathbb R$, let 
\begin{equation}\label{e-gue210604yydIx}
Q_{[\delta_1,\delta_2]}: L^2_{(0,q)}(\mathbb H_{n-1},\phi_0)\to E_{-iT}([\delta_1,\delta_2])
\end{equation}
be the orthogonal projection with respect to $(\,\cdot\,,\,\cdot\,)_{\mathbb H_{n-1},\phi_0}$, where $E_{-iT}([\delta_1,\delta_2])$ denotes the spectral measure of $-iT$.  For $\delta>0$, 
let 
\begin{equation}\label{e-gue210604yydu}
e^{-t\Box^{q,\delta}_{b,\phi_0}}:=e^{-t\Box^q_{b,\phi_0}}\circ Q_{[-\delta,\delta]}: L^2_{(0,q)}(\mathbb H_{n-1},\phi_0)\to L^2_{(0,q)}(\mathbb H_{n-1},\phi_0)
\end{equation}
and 
let $e^{-t\Box^{q,\delta}_{b,\phi_0}}(x,y)\in\mathscr D'(\mathbb R_+\times \mathbb H_{n-1}\times \mathbb H_{n-1}, T^{*0,q}\mathbb H_{n-1}\boxtimes(T^{*0,q}\mathbb H_{n-1})^*)$ be the 
distribution kernel of $e^{-t\Box^{q,\delta}_{b,\phi_0}}$. For every $\delta>0$ and $t>0$, set 
\begin{equation}\label{e-gue210521ycdII}
\begin{split}
P_\delta(t,x,y):&=\frac{1}{2\pi}\int_{-\delta}^{\delta} e^{i<\theta-\sigma,\eta>+\frac{\phi_0(z)-\phi_0(w)}{2}}e^{-t\Box_\eta}(z,w)(\eta)d\eta\\
&\in\cali{C}^\infty(\mathbb R_+\times \mathbb H_{n-1}\times \mathbb H_{n-1},T^{*0,q}\mathbb H_{n-1}\boxtimes(T^{*0,q}\mathbb H_{n-1})^*).
\end{split}
\end{equation}
%where $1_{[\delta_1,\delta_2]}(\eta)=1$ if $\eta\in[\delta_1,\delta_2]$, $1_{[\delta_1,\delta_2]}(\eta)=0$ if $\eta\notin[\delta_1,\delta_2]$.
Let
$$P_\delta(t): \Omega^{0,q}_c(\mathbb H_{n-1})\to\Omega^{0,q}(\mathbb H_{n-1})$$
be the continuous operator given by 
\begin{equation}\label{e-gue210521ycdIII}
(P_\delta(t)u)(x)=\int P_\delta(t,x,y)u(y)dv_{\mathbb H_{n-1}}(y),\ \ u\in\Omega^{0,q}_c(\mathbb H_{n-1}).
\end{equation} 
From \cite[(4.9)-(4.19)]{HZ23}, we deduce that
\begin{equation}\label{modelheat}
 e^{-t\Box^{q,\delta}_{b,\phi_0}}=P_\delta(t) \quad
\text{and}\quad
e^{-t\Box^{q}_{b,\phi_0}}=\lim_{\delta\to \infty}P_\delta(t).   
\end{equation}
Let $\ell\in\mathbb N$ and  $K\subset\mathbb R_+\times \mathbb H_{n-1}\times \mathbb H_{n-1}$ be a compact set, and $\delta_2>\delta_1\gg 1$. Note that
\begin{equation}\label{e-unib}
P_{\delta_2}-P_{\delta_1}=\frac{1}{2\pi}\int_{[-\delta_2,-\delta_1]\cup [\delta_1,\delta_2]} e^{i<\theta-\sigma,\eta>+\frac{\phi_0(z)-\phi_0(w)}{2}}e^{-t\Box_\eta}(z,w)d\eta.
\end{equation} 
\begin{rem}\label{rem:zq} The derivation of \eqref{modelheat} does not require condition $Z(q)$.
Assume now that $Z(q)$ holds. It can be verified that, for every $\ell \in \mathbb{N}$, there exist constants $C > 0$ and $\varepsilon > 0$, such that the integral on the right-hand side of \eqref{e-unib} over the negative region of $\eta$ remains uniformly bounded. Specifically, we have
\begin{equation}\label{e-unip}
\Big\|\int_{-\delta_2}^{-\delta_1} e^{i\langle\theta-\sigma,\eta\rangle}e^{-t\Box_\eta}(z,w)d\eta\Big\|_{\cali{C}^\ell(K,T^{*0,q}\mathbb H_{n-1}\boxtimes(T^{*0,q}\mathbb H_{n-1})^*)}\leq Ce^{-\varepsilon\delta_1},\qquad \delta_2>\delta_1\gg 1.
\end{equation}
\end{rem}
The estimate \eqref{e-unip} will be used in the proof of Theorem~\ref{main1} and in Remark~\ref{rem:modelasym}.
We include a sketch of the proof of \eqref{e-unip}. Since $|\eta|\gg1$, the dominant contribution to $\Phi_\eta$ in \eqref{e-unip} is given by $-2\eta\lambda|z|^2$, the terms involving $\phi_0$ are of lower order in $\eta$ and can be absorbed into the constants below. The leading term in $e^{-t\Box_\eta}(z, w)$ can be expressed as  
\begin{equation*}
\dfrac{\det\dot R^\eta}{\det\big(1-e^{-t \dot R^\eta}\big)}e^{-t\varTheta^\eta_0}\approx\dfrac{\det\dot R^\eta}{(1-e^{2t\eta\lambda_1})\cdots(e^{-2t\eta\lambda_{j_1}}-1)\cdots(e^{-2t\eta\lambda_{j_q}}-1)\cdots(1-e^{2t\eta\lambda_n})},
\end{equation*}
where the indices $j_1, \dots, j_q$ correspond to directions related to the degree $q$.
When $\eta \ll 0$, condition $Z(q)$ ensures that at least one of the following holds:

1. There exists $k \in J$ such that $\lambda_k > 0$, which implies $1/(e^{-2t\eta\lambda_k} - 1) \leq e^{-\varepsilon\delta_1}$ for some $\varepsilon > 0$.

2. There exists $k \notin J$ such that $\lambda_k < 0$, implying that $1/(1 - e^{2t\eta\lambda_k}) \leq e^{-\varepsilon\delta_1}$ for some $\varepsilon > 0$.\\
In either case, the integral in \eqref{e-unip} is bounded uniformly for $\delta_2 > \delta_1 \gg 1$, and the estimate follows as stated. This completes the proof of \eqref{e-unip}. 

%%%%%%%%%
We now proceed with the fundamental solution of $\Box_{\phi_0}^q$. Let $P:=\Box_{b,\phi_0}^q-\frac{1}{2}\frac{\pa^2}{\pa\theta^2}$, and rewrite the first line of \eqref{hh1} as
$$
\left(\frac{\pa}{\pa t}+P-\frac{1}{2}\frac{\pa^2}{\pa r^2}\right)K(t)f=0.
$$ 
Note that $P$ and $-\frac{1}{2}\frac{\partial^2}{\partial r^2}$ commute. We therefore expect that $K(t)$ can be written as the product of the two corresponding kernels, that is,
\begin{equation}\label{ker1}
K(t, \mathbf z, \mathbf w)=p(t,x,y)\cdot g(t,r,s),
\end{equation}
where $\mathbf z=[x,r]=[\xi,\theta,r]$, $\mathbf w=[y,s]=[\zeta,\sigma,s]$, $g(t,r,s)=\frac{1}{\sqrt{2\pi t}}e^{-|r-s|^2/(2t)}$ is the Gaussian kernel on $\R$, and $p(t)$ is the kernel for the fundamental solution of the initial value problem for $P$ on $\mathbb H_{n-1}$. 
Note that the coefficient of $-\frac{1}{2}\frac{\pa^2}{\pa \theta^2}$ is independent of $z$ and $\theta$. Let

\begin{equation}\label{con1}
\begin{aligned}
p(t,x,y)&:=\frac{1}{\sqrt{2\pi t}}\int e^{-t\Box_{b,\phi_0}^q}((\xi,\theta+\mu),(\zeta,\sigma))e^{-|\mu|^2/(2t)} d\mu\\
&=\frac{1}{2\pi\sqrt{2\pi t}}\lim_{\delta\to \infty}\int_{-\delta}^\delta\int e^{i<\theta+\mu-\sigma,\eta>+\phi_0(\xi)/2-\phi_0(\zeta)/2-|\mu|^2/(2t)}e^{-t\Box_\eta}(\xi,\zeta)d\eta d\mu\\
&=\frac{e^{\phi_0(\xi)/2-\phi_0(\zeta)/2}}{2\pi}\lim_{\delta\to \infty}\int_{-\delta}^\delta e^{i<\theta-\sigma,\eta>-t\eta^2/2}e^{-t\Box_\eta}(\xi,\zeta)d\eta ,
\end{aligned}
\end{equation}
where the last equality is obtained by the Fourier transform of a Gaussian: 
$$
\int e^{i\langle \mu,\eta\rangle-|\mu|^2/(2t)} d\mu=\sqrt{2\pi t}\,e^{-t\eta^2/2}.
$$
One can verify that 
\begin{equation}\label{hh3}
\begin{dcases}
\left(\frac{\pa}{\pa t}+\Box^q_{b,\phi_0}-\frac{1}{2}\frac{\pa^2}{\pa \theta^2}\right)p(t)f=0,\\
\displaystyle\lim_{t\to0}p(t)f=f.
\end{dcases}
\end{equation}
Combining \eqref{ker1} and \eqref{con1}, we have
\begin{equation}\label{ker2}
\begin{aligned}
K(t,\mathbf z,\mathbf w)&=p(t,x,y)\cdot g(t,r,s)\\
&=\frac{e^{\phi_0(\xi)/2-\phi_0(\zeta)/2}}{2\pi\sqrt{2\pi t}}e^{-|r-s|^2/(2t)}\lim_{\delta\to \infty}\int_{-\delta}^\delta e^{i<\theta-\sigma,\eta>-t\eta^2/2}e^{-t\Box_\eta}(\xi,\zeta)d\eta
\end{aligned}
\end{equation}
is the fundamental solution of the operator $\Box_{\phi_0}^q$ on $\mathbb H_{n-1}\times \mathbb R$.
In particular,
\begin{equation}\label{ker0}
\begin{aligned}
K(t,0,0)=\frac{1}{(2\pi)^{n}\sqrt{2\pi t}}\lim_{\delta\to \infty}\int_{-\delta}^\delta e^{-t\eta^2/2}\dfrac{\det\dot R^\eta}{\det\big(1-e^{-t \dot R^\eta}\big)}
e^{-t\varTheta^\eta_0}d\eta,
\end{aligned}
\end{equation}
where $\dot R^\eta$ and $\varTheta_0^\eta$ are defined by \eqref{Reta} and \eqref{omega}, respectively.

\subsection{Boundary correction terms} We now 
 proceed to construct the correction terms to ensure the associated heat kernel $A(t)$ satisfies the desired boundary conditions. 

Let $f=\sumprime_{\abs{J}=q}f_J\ol\omega_0^J\in\omz^{0,q}(\ol M_0)\cap \Dom{\Box_{\phi_0}^q}$. We can decompose the operator $\Box_{\phi_0}^q$ acting on $f$ as follows:
\begin{equation}
\begin{aligned}
\Box_{\phi_0}^qf=\sumprime_{n\notin J}\Box^\tau f_J\ol\omega_0^J+\sumprime_{n\in J}\Box^\nu f_J\ol\omega_0^J,
\end{aligned}
\end{equation}
where 
\begin{equation}\label{tna}
\Box^\tau=P^\tau-\frac{1}{2}\frac{\pa^2}{\pa r^2},\quad \Box^\nu=P^\nu-\frac{1}{2}\frac{\pa^2}{\pa r^2}.
\end{equation}
%Here, whether the effect is tangential ($\tau$) or normal ($\nu$) depends on the action of term  
The distinction between the tangential ($\tau$) and normal ($\nu$) action is determined by the term
\begin{equation}\label{acting}
\sum_{j,l=1}^{{n-1}}  \left( 2i \lambda_j \frac{\pa}{\pa \theta} + \mu_{l,j} \right)d \ol{\xi}_j \wedge \left( d \ol{\xi}_l \wedge \right)^*
\end{equation}
appearing in the expression for $\Box_{b,\phi_0}^q$. More precisely, the tangential (respectively, normal) action means that \eqref{tna} acts freely on the normal (respectively, tangential) part of a form. In other words, according to its action on the tangential and normal components of a form, the operator $\varTheta_{0}^\eta$ in $K(t, \mathbf{z}, \mathbf{w})$ can be decomposed into its tangential part $\varTheta_{0}^{\eta,\tau}$ and normal part $\varTheta_{0}^{\eta,\nu}$.

Write
$$
\begin{aligned}
A_{\phi_0}(t,\mathbf z, \mathbf w)&=A^\tau(t,\mathbf z, \mathbf w)+A^\nu(t,\mathbf z, \mathbf w)\\
:&=\Big (\sumprime_{\substack{|I|=|J|=q \\ n\notin J}}+\sumprime_{\substack{|I|=|J|=q \\ n\in J}}\Big)A_{\phi_0,I,J}(t,\mathbf z, \mathbf w)\,\ol\omega_0^I(\mathbf z)\otimes\ol\omega_0^J(\mathbf w).    
\end{aligned}
$$
From the boundary conditions \eqref{mbdc} for $\Box^q_{\phi_0}$, the heat equation can be split into the following two:
\begin{equation}\label{Di}
\begin{aligned}
\Big(\frac{\pa}{\pa t}+\Box^\nu\Big)A^{\nu}= 0
\mbox{ with $A^{\nu}|_{X_0}$=0}
\quad \mbox{and\quad$\lim_{t\to0}A^{\nu}(t)=\text{Id}$},
\end{aligned}
\end{equation}
\begin{equation}\label{Neu}
\begin{aligned}
\Big(\frac{\pa}{\pa t}+\Box^\tau\Big)A^{\tau}= 0
\mbox{ with $\ol Z_{n}A^{\tau}|_{X_0}$=0}
\quad \mbox{and\quad$\lim_{t\to0}A^{\tau}(t)=\text{Id}$}.
\end{aligned}
\end{equation}
As previously discussed, $K(t, \cdot, \mathbf{w})$ satisfies both heat equations, but without considering the boundary conditions. We now focus on determining the boundary correction terms $h^\nu(t)$ and $h^\tau(t)$ such that  
$$
A^\nu(t) = K(t) + h^\nu(t) \quad \text{and} \quad A^\tau(t) = K(t) + h^\tau(t).
$$

Note that the problem \eqref{Di}, involving the normal component, can essentially be treated as a Dirichlet problem. We then obtain $h^\nu(t)$ by reflection; namely, let
$$
h^\nu(t,\mathbf z, \mathbf w)=-K(t,(x,r),(y,-s))
$$
and
$$
A^\nu(t,\mathbf z, \mathbf w)=K(t,(x,r),(y,s))-K(t,(x,r),(y,-s)),
$$
where $\mathbf z=[x,r]$ and $\mathbf w=[y,s]$. Consequently, the operator $A^\nu(t)$, as defined by 
\begin{equation}\label{Nu}
\begin{aligned}
(A^\nu(t)f)(x,r)=\int_{\mathbb H_{n-1}\times \R}A^\nu(t,\mathbf z, \mathbf w)f(\mathbf w)dv(\mathbf w)
\end{aligned}
\end{equation}
is the desired operator for $(\frac{\pa}{\pa t}+\Box^\nu)$, satisfying
$$
\left(\frac{\pa}{\pa t}+\Box^\nu\right)A^{\nu}(t)f= 0
\mbox{ in $M_0$,$\quad A^{\nu}(t)f|_{X_0}=A^{\nu}(t)f|_{r=0}$=0}
\quad \mbox{and\quad$\lim_{t\to0}A^{\nu}(t)f=f$}.
$$
The proof is straightforward. In fact, it follows from $K(t,(x,r),(y,s))=K(t,(x,-r),(y,-s))$ that
$$
\begin{aligned}
(A^{\nu}(t)f)(x,0)&=\int K(t,(x,0),(y,s))f(y,s)dyds-\int K(t,(x,0),(y,-s))f(y,s)dyds\\
&=\int K(t,(x,0),(y,-s))f(y,s)dyds-\int K(t,(x,0),(y,-s))f(y,s)dyds\\
&=0.
\end{aligned}
$$
Moreover, $\lim_{t\to0}A^{\nu}(t)f=f$ since $f\in \omz^{0,q}_c(\mathbb H_{n-1}\times\R_-)$.

On the other hand, the problem \eqref{Neu}, involving the tangential component, is subject to the $\bar\partial$-Neumann boundary condition. We first construct a related operator $A^N$ satisfying the following Neumann boundary condition:
$$
\left(\frac{\pa}{\pa t}+\Box^\tau\right)A^{N}(t)f= 0,
\mbox{$\quad \frac{\pa}{\pa r}(A^{N}(t)f)|_{r=0}$=0}
\quad \mbox{and\quad$\lim_{t\to0}A^{N}(t)f=f$}.
$$
We claim that
$$
\begin{aligned}
A^N(t,\mathbf z, \mathbf w)=K(t,(x,r),(y,s))+K(t,(x,r),(y,-s)).
\end{aligned}
$$
Again, we notice that
$$
\left(\frac{\pa}{\pa r}K\right)(t,(x,r),(y,s))=-\left(\frac{\pa}{\pa r}K\right)(t,(x,-r),(y,-s)),
$$
which implies
\begin{equation*}
\begin{aligned}
&\frac{\pa}{\pa r}(A^{N}(t)f)(x,0)\\
&=\int \left(\frac{\pa}{\pa r}K\right)(t,(x,0),(y,s))f(y,s)dyds+\int \left(\frac{\pa}{\pa r}K\right)(t,(x,0),(y,-s))f(y,s)dyds\\
&=-\int \left(\frac{\pa}{\pa r}K\right)(t,(x,0),(y,-s))f(y,s)dyds+\int \left(\frac{\pa}{\pa r}K\right)(t,(x,0),(y,-s))f(y,s)dyds\\
&=0.
\end{aligned}
\end{equation*}
We then rewrite the solution of \eqref{Neu} as 
$$
A^{\tau}(t)=A^{N}(t)+h^{\tau}(t),
$$
where $h^{\tau}(t)$ is another correction term we want to find, and it satisfying
\begin{equation}\label{htau1}
\begin{aligned}
\left(\frac{\pa}{\pa t}+\Box^\tau\right)h^{\tau}(t)f= 0,
\quad \mbox{$\displaystyle\lim_{t\to0}h^{\tau}(t)f=0$},
\end{aligned}
\end{equation}
\begin{equation}\label{htau2}
\begin{aligned}
\mbox{and\quad }\ol Z_{n}h^{\tau}(t)f|_{X_0}=-\ol Z_{n}A^{N}(t)f|_{X_0}.
\end{aligned}
\end{equation}
Moreover, we aim to find its kernel function $h^\tau(t, (x, r), (y, s))$ that decays rapidly as $r \to -\infty$ or $s \to -\infty$, ensuring that the correction term primarily affects the boundary and does not introduce significant contributions away from it. From \eqref{htau2}, we have
\begin{equation}\label{htau3}
\begin{aligned}
\Big(\frac{\pa}{\pa r}-i\frac{\pa}{\pa \theta}\Big)h^{\tau}(t)f\Big |_{r=0}=i\frac{\pa}{\pa \theta}A^{N}(t)f\Big |_{r=0},
\end{aligned}
\end{equation}
which implies
\begin{equation*}
\begin{aligned}
&\int_{\mathbb H_{n-1}} \int_{-\infty}^0  \Big(\frac{\pa}{\pa r}-i\frac{\pa}{\pa \theta}\Big)h^{\tau}(t,(x,r),(y,s))f(y,s)dyds\bigg |_{r=0}\\
&=\int_{\mathbb H_{n-1}} \int_{-\infty}^0  i\frac{\pa}{\pa \theta}\Big[K(t,(x,r),(y,s))+K(t,(x,r),(y,-s))\Big]f(y,s)dyds\bigg |_{r=0}.
\end{aligned}
\end{equation*}
Note that $f\in\omz^{0,q}_c(\mathbb H_{n-1}\times\R_-)$ and $K(t,(x,r),(y,s))=K(t,(x,-r),(y,-s))$, the above equation holds if we can find 
a solution from the following:
\begin{equation}\label{htau4}
\begin{aligned}
\Big(\frac{\pa}{\pa r}-i\frac{\pa}{\pa \theta}\Big)h^{\tau}(t,(x,r),(y,s))=
2i\frac{\pa}{\pa \theta}K(t,(x,r),(y,-s))
\end{aligned}
\end{equation}
for all $x,y\in \mathbb H_{n-1}$, $r,s<0$, $t>0$.
 Taking the Fourier transform with respect to $\theta$ (where $x=(\xi,\theta)$) in \eqref{htau4}, we obtain
\begin{equation*}
\begin{aligned}
\Big(\frac{\pa}{\pa r}+u\Big)\hat h^{\tau}(t,(\xi,u,r),(y,s))=-2u \hat K(t,(\xi,u,r),(y,-s)),
\end{aligned}
\end{equation*}
which gives
\begin{equation*}
\begin{aligned}
\frac{\pa}{\pa r}\Big(e^{u r}\hat h^{\tau}(t,(\xi,u,r),(y,s))\Big)=-2u e^{u r}\hat K(t,(\xi,u,r),(y,-s)).
\end{aligned}
\end{equation*}
We will find $h^{\tau}(t,(x,r),(y,s))$ from \eqref{htau4} under the condition $\lim_{r\to-\infty}e^{u r}\hat h^{\tau}(t,(\xi,u,r),(y,s))=0$ for all $u\in\R$.
Taking integral to both sides from $-\infty$ to $r$ and then divided by $e^{u r}$, we obtain 
\begin{equation*}
\begin{aligned}
\hat h^{\tau}(t,(\xi,u,r),(y,s))&=e^{-u r}\int_{-\infty}^r2u e^{u \gamma}\hat K(t,(\xi,u,\gamma),(y,-s))d\gamma\\
&=2u\int_{-\infty}^0 e^{u \gamma}\hat K(t,(\xi,u,r+\gamma),(y,-s))d\gamma.
\end{aligned}
\end{equation*}
We then take the inverse Fourier transform with respect to $u$ and find that
\begin{equation*}
\begin{aligned}
h^{\tau}(t,(\xi,\theta,r),(y,s))=\frac{1}{\pi}\int_\R\int_{-\infty}^0 u e^{u (i\theta+\gamma)}\hat K(t,(\xi,u,r+\gamma),(y,-s))d\gamma du.
\end{aligned}
\end{equation*}
It follows from \eqref{ker2} that

\begin{equation*}
\begin{aligned}
&\hat K(t,(\xi,u,r),(y,-s) )\\
&=\frac{e^{\phi_0(\xi)/2-\phi_0(\zeta)/2}}{2\pi\sqrt{2\pi t}}e^{-|r+s|^2/(2t)}\lim_{\delta\to \infty}\int_{-\delta}^\delta e^{i<-\sigma,\eta>-t\eta^2/2}e^{-t\Box_\eta}(\xi,\zeta)\Big(\int_{\R} e^{i<\theta,\eta-u>}d\theta \Big)d\eta\\
&=\frac{e^{\phi_0(\xi)/2-\phi_0(\zeta)/2}}{\sqrt{2\pi t}}e^{-|r+s|^2/(2t)} e^{i<-\sigma,u>-tu^2/2}e^{-t\Box_{u}}(\xi,\zeta).
\end{aligned}
\end{equation*}
Hence we have
\begin{equation}\label{qker}
\begin{aligned}
&h^{\tau}(t,(\xi,\theta,r),(\zeta,\sigma,s))\\
&=\frac{e^{\phi_0(\xi)/2-\phi_0(\zeta)/2}}{\pi \sqrt{2\pi t}}\lim_{\delta\to \infty}\int_{-\delta}^\delta u e^{iu (\theta-\sigma)-tu^2/2}e^{-t\Box_u}(\xi,\zeta)\Big(\int_{-\infty}^0e^{u\gamma-|r+\gamma+s|^2/(2t)}d\gamma\Big) du\\
&=\frac{e^{\phi_0(\xi)/2-\phi_0(\zeta)/2}}{\pi \sqrt{2\pi t}}\lim_{\delta\to \infty}\int_{-\delta}^\delta u e^{iu (\theta-\sigma)-tu^2/2}e^{-t\Box_u}(\xi,\zeta)\Big(\int_{-\infty}^0e^{-(\gamma+r+s-tu)^2/(2t)-u(r+s)+tu^2/2}d\gamma\Big) du\\
&=\frac{e^{\phi_0(\xi)/2-\phi_0(\zeta)/2}}{\pi \sqrt{\pi }}\lim_{\delta\to \infty}\int_{-\delta}^\delta \eta e^{i\eta (\theta-\sigma)-\eta(r+s)}e^{-t\Box_{\eta}}(\xi,\zeta)\Big(\int_{-\infty}^{\frac{r+s-t\eta}{\sqrt{2t}}}e^{-\gamma^2}d\gamma\Big) d\eta.
\end{aligned}
\end{equation}
Combining the solutions to \eqref{Di} and \eqref{Neu}, we obtain the following theorem. In particular, it implies Theorem~\ref{hmd}.

\begin{thm}\label{heatkernelmodel}
Let $t\in\mathbb{R}_+$ and $f\in\Omega^{0,q}_c(M_0)$. Let $K(t,\mathbf z,\mathbf w)$ and $h^\tau(t,\mathbf z,\mathbf w)$ be as in \eqref{ker2} and \eqref{qker}, respectively.
Set
$$
A^\nu(t,\mathbf z, \mathbf w):=K(t,(x,r),(y,s))-K(t,(x,r),(y,-s))
$$
Then the operator $A^\nu(t)$ solves
\begin{equation}
\begin{aligned}
&\left(\frac{\pa}{\pa t}+\Box^\nu\right)A^{\nu}(t)f^\nu= 0,
\mbox{\quad $A^{\nu}(t)f^\nu|_{X_0}=0$},\\
&\mbox{\qquad$\lim_{t\to0}A^{\nu}(t)f^\nu=f^\nu$ in $L^2_{(0,q)}(M_0)$.}
\end{aligned}
\end{equation}
Set
$$
A^\tau(t,\mathbf z, \mathbf w):=K(t,(x,r),(y,s))+K(t,(x,r),(y,-s))+h^\tau(t,(x,r),(y,s))
$$
Then the operator $A^\tau(t)$ solves
\begin{equation}
\begin{aligned}
&\left(\frac{\pa}{\pa t}+\Box^\tau\right)A^{\tau}(t)f^\tau= 0,
\mbox{\quad $\ol Z_{n}(A^{\tau}(t)f^\tau)|_{X_0}=0$},\\
&\mbox{\qquad$\lim_{t\to0}A^{\tau}(t)f^\tau=f^\tau$ in $L^2_{(0,q)}(M_0)$.}
\end{aligned}
\end{equation}
Furthermore, 
$$e^{-t\Box^q_{\phi_0}}(\mathbf z,\mathbf w)=A^\nu(t,\mathbf z, \mathbf w)+A^\tau(t,\mathbf z, \mathbf w).$$
\end{thm}
To conclude this section, we present the explicit formulas for the tangential and normal components of the heat kernel near the origin. They are given by
\begin{equation}\nonumber
\begin{aligned}
A^\nu(t,(0,r), (0,r))&=\frac{1-e^{-2r^2/t}}{2\pi\sqrt{2\pi t}}\lim_{\delta\to \infty}\int_{-\delta}^\delta e^{-t\eta^2/2}e^{-t\Box_\eta}(0,0)d\eta\\
&=\frac{1-e^{-2r^2/t}}{(2\pi)^{n}\sqrt{2\pi t}}\lim_{\delta\to \infty}\int_{-\delta}^\delta e^{-t\eta^2/2}\dfrac{\det\dot R^\eta}{\det\big(1-e^{-t \dot R^{\eta}}\big)}
e^{-t\varTheta^{\eta,\nu}}d\eta
\end{aligned}
\end{equation}
and 
\begin{equation}\nonumber
\begin{aligned}
A^\tau(t,(0,r), (0,r))=&\frac{1+e^{-2r^2/t}}{(2\pi)^{n}\sqrt{2\pi t}}\lim_{\delta\to \infty}\int_{-\delta}^\delta e^{-t\eta^2/2}\dfrac{\det\dot R^\eta}{\det\big(1-e^{-t \dot R^{\eta}}\big)}
e^{-t\varTheta^{\eta,\tau}}d\eta\\
&+\frac{2}{(2\pi)^{n} \sqrt{\pi}}\lim_{\delta\to \infty}\int_{-\delta}^\delta \eta e^{-2\eta r} \dfrac{\det\dot R^\eta}{\det\big(1-e^{-t \dot R^{\eta}}\big)}
e^{-t\varTheta^{\eta,\tau}}\Big(\int^{\frac{2r-t\eta}{\sqrt{2t}}}_{-\infty}e^{-\gamma^2}d\gamma\Big)d\eta.
\end{aligned}
\end{equation}
These expressions will be applied to prove our main theorems in the next section.

%%%%%%%%%%%%%%%%%%%%%%%%%%%

\section{Heat kernel asymptotics on complex manifolds with boundary}\label{heatasymptotic}
In this section, we prove the asymptotic formulas for the heat kernel of the $\dbar$-Neumann Laplacian with values in $L^k$ as $k\to+\infty$. The main issue is the boundary regime, where the natural scale is anisotropic, and the operator must be compared with the model kernel constructed in Section~\ref{bdd}. Throughout this section, we use the notation from Section~\ref{sec:prelim} and assume that condition $Z(q)$ holds.

We regard $M$ as a domain with a smooth boundary $X$ in $M'$. Let $r\in \cali{C}^\infty(M')$ be a defining function such that
$$
\ol M=\{{\bf z}\in M':r({\bf z})\le 0\},\qquad
X=\{{\bf z}\in M':r({\bf z})=0\},
$$
and $|dr|=1$ on $X$. For every $k\in\N$, we consider the Hermitian metric $\langle\,\cdot\, \mid\,\cdot\, \rangle_k$ on $\C T^*M'$ defined in \eqref{e-gue260125yyd}. As in the introduction, the role of this metric is to capture the boundary contribution at the scale relevant for Morse inequalities.

The argument splits naturally into boundary and interior parts. In the interior, away from $X$, the analysis is essentially the same as in the compact case without boundary; see Section~\ref{innestimate}. Near the boundary, however, one must rescale the kernel, prove uniform estimates for the scaled operators, and identify the resulting limit with the model heat kernel from Section~\ref{bdd}. The next theorem is the main result of this section and gives the precise form of this boundary scaling limit.

\begin{thm}\label{main1}
Let $\ol M$ be the closure of a relatively compact open subset $M$ of an $n$-dimensional complex manifold $M'$ in which $\ol M$ has a smooth boundary $X$, and $(L^k,h^{L^k})$ be the $k$-th tensor power of a holomorphic line bundle $(L,h^L)$ over $M'$. Let $q\in\set{0,1,\ldots,n}$. Suppose that $M$ satisfies condition $Z(q)$. Then, in local boundary coordinates and after the anisotropic scaling introduced in this section, the rescaled heat kernel converges to the heat kernel of the model weighted $\dbar$-Neumann Laplacian on $\ol M_0$. More precisely, with the notation introduced below, one has
\begin{equation}\label{e-gue260126yyda}
\lim_{k\to\infty}A_{(k)}(t,\mathbf z,\mathbf w)=e^{-t\Box^q_{\phi_0}}(\mathbf z,\mathbf w)
\end{equation}
in the $\mathscr C^\infty(I\times\ol B_R\times\ol B_R, \Lambda^\bullet(\mathbb CT^*\C^n)\boxtimes(\Lambda^\bullet(\mathbb CT^*\C^n))^*)$ topology, for every $I\Subset\R_+$.

In particular, near the boundary $X$, the diagonal heat kernel admits the asymptotic formula
\begin{equation}\label{localasymoptptic}
\begin{aligned}
&\lim_{k\to\infty}k^{-(n+1)}e^{-\frac{t}{k}\Box_{k}^q}\Big([x,\frac{r'}{k}],[x,\frac{r'}{k}]\Big)=\frac{1+e^{-2r'^2/t}}{(2\pi)^{n}\sqrt{2\pi t}}\int_{\R}  \dfrac{e^{-t\eta^2/2}\cdot\det(\dot{\mathcal{R}}^L_{b,x}-2\eta\dot{\mathcal L}_x)}{\det\big(1-e^{-t (\dot{\mathcal{R}}^L_{b,x}-2\eta\dot{\mathcal L}_x)}\big)}
e^{-t\varTheta^{\eta,\tau}_x}d\eta \\
&\qquad \qquad \qquad \qquad \quad   +\frac{2}{(2\pi)^{n} \sqrt{\pi}}\int_{\R} \dfrac{\eta e^{-2r'\eta} \cdot\det(\dot{\mathcal{R}}^L_{b,x}-2\eta\dot{\mathcal L}_x)}{\det\big(1-e^{-t (\dot{\mathcal{R}}^L_{b,x}-2\eta\dot{\mathcal L}_x)}\big)}e^{-t\varTheta^{\eta,\tau}_x}\Big(\int^{\frac{2r'-t\eta}{\sqrt{2t}}}_{-\infty}e^{-\gamma^2}d\gamma\Big) d\eta \\
&\qquad \qquad \qquad \quad \qquad  +\frac{1-e^{-2r'^2/t}}{(2\pi)^{n}\sqrt{2\pi t}}\int_{\R} \dfrac{e^{-t\eta^2/2}\cdot\det(\dot{\mathcal{R}}^L_{b,x}-2\eta\dot{\mathcal L}_x)}{\det\big(1-e^{-t (\dot{\mathcal{R}}^L_{b,x}-2\eta\dot{\mathcal L}_x)}\big)}
e^{-t\varTheta^{\eta,\nu}_x}d\eta,
\end{aligned}
\end{equation}
where $\mathcal R^L_{b,x}$ is defined by \eqref{curvature2}, $\dot{\mathcal R}^L_{b,x}$ is the associated endomorphism defined by \eqref{ecur}, and $\varTheta^\eta_x$ is given by \eqref{bTheta}.
\end{thm}

\begin{rem}\label{r-gue260127yyd}
(i) The dependence on $q$ enters the formula through the actions of $\varTheta^{\eta,\tau}$ and $\varTheta^{\eta,\nu}$ on $(0,q)$-forms; see the discussion around \eqref{acting}.

(ii) The convergence in \eqref{e-gue260126yyda} is in the $\mathscr C^\infty$ sense on compact subsets in $(t,\mathbf z,\mathbf w)$.

(iii) The specific metric \eqref{e-gue260125yyd} is chosen in order to capture the boundary contribution relevant for Morse inequalities. More generally, the conclusion of Theorem~\ref{main1} remains valid as long as 
$$
\langle\,\dbar r\mid\dbar r\,\rangle_k\sim \frac1k \langle\,\dbar r\mid\dbar r\,\rangle
\quad\text{on}\quad
\big\{{\bf z}\in\ol M:-k^{-1}\lesssim r({\bf z})\le 0\big\}.
$$

(iv) Away from the boundary, where the metric is independent of $k$, one has the interior asymptotic formula \eqref{htin}; see Section~\ref{innestimate}.
\end{rem}

\subsection{Local coordinates near the boundary}\label{localc}
We now introduce the local coordinates and scaled kernels used in the proof of Theorem~\ref{main1}. Fix a point $p\in X$, and let $s$ be a local holomorphic trivializing section of $L$ on an open neighborhood $U$ of $p$ in $M'$, with local weight $\phi$. In other words, $|s({\bf z})|^2_{h^L}=e^{-\phi(\bf z)}$ for all ${\bf z}\in U$. We choose a special boundary chart $\mathbf z=(z,z_{n})=(z_1,\cdots,z_{n-1},z_{n})$ on $U$ such that
\begin{equation}\label{local1}
 \mathbf z(p)=0,\quad r(\mathbf z)=\text{Im}z_{n}+\sum_{j=1}^{{n-1}}\lambda_j|z_j|^2+O(|{\bf z}|^3),
\end{equation}
where $\lambda_j\in\R$, $j=1,\cdots,{n-1}$, and the metric satisfies
\begin{equation}\label{local2}
\begin{aligned}
&\langle\,d\ol z_j\mid d\ol z_l\,\rangle_k=\delta_{j,l}+O(|\mathbf z|), \quad
\langle\,d\ol z_{n}\mid\,d\ol z_{n}\,\rangle_k=\frac{2}{k}+O(\abs{z}),
 \end{aligned}
\end{equation}
for $j,l=1,\cdots,{n-1}$. Moreover, the weight function can be expressed as
\begin{equation*}
\begin{aligned}
\phi(\mathbf z)=\sum_{j,l=1}^{n}\mu_{j,l}z_j\ol z_l+O(|(z,z_{n})|^3),
\end{aligned}
\end{equation*}
where $\mu_{j,l}\in\C$, $\mu_{j,l}=\ol{\mu_{l,j}}$, $j, l=1,\ldots,n$. We proceed using the Heisenberg coordinates $[\xi, \theta, r]$ introduced in Section \ref{M0}. Under these coordinates, $\phi(\mathbf{z})$ can be expressed as
\begin{equation}\label{bphi}
\phi(\mathbf z)=\sum_{j,l=1}^{{n-1}}\mu_{j,l}\xi_j\ol \xi_l+O(|\xi||(\theta,r)|+|(\theta,r)|^2+|\mathbf z|^3).
\end{equation}
Let forms $\ol\omega^1,\cdots,\ol\omega^{n}$ be an orthogonal basis for $T^{*0,1}M'$ on $U$ so that 
\begin{equation}
\ol\omega^{n}=\sqrt 2\,\dbar r, \quad \langle\,\ol\omega^j(p)\mid\ol\omega^l(p)\,\rangle_k=\delta_{j,l}\quad\mbox{for $j,l=1,\cdots,{n-1}$}.
\end{equation}
Remark that $\langle\,\ol\omega^{n}\mid\ol\omega^{n}\,\rangle_k=\frac{1}{k}$ at $p$. Moreover, let $\{\ol Z_1,\cdots, \ol Z_{n}\}$ be a smooth orthogonal frame for $(0,1)$ vector fields which is dual to $\{\ol \omega^1,\cdots,\ol \omega^{n}\}$. 
On $U$, we can write 
\begin{equation}\label{bLj}
\ol Z_j=\frac{\pa}{\pa\ol \xi_j}+i\lambda_j \xi_j\frac{\pa}{\pa\theta}+O(|\mathbf z|^2),\quad j=1,\cdots,{n-1},
\end{equation}
and
\begin{equation}\label{bLn}
\ol Z_{n}=\frac{1}{\sqrt 2}\Big(\frac{\pa}{\pa r}-i\frac{\pa}{\pa \theta}\Big)+O(|\mathbf z|^2).
\end{equation}
%Note that $Z_j$ is tangential to the level surface of $r$, while $Z_n$ is a complex normal vector field. %i.e., 
%$$
%\begin{dcases}
%\,\ol Z_j(\rho)=0, \quad\text{when $z\in X\cap U$},\, j=1,\cdots,{n-1}-1, \\
%\,\ol Z_n(\rho)=1, \quad\text{when $z\in X\cap U$}.\\
%\end{dcases}
%$$

%\vspace{0.25cm}
Let $D:=U\cap M$, $\partial D:=U\cap X$, and $\ol D:=U\cap \ol M$. For $f,g\in \omz^{0,q}(\ol D)$, let
$$
(f ,  g)_{k\phi,D}=(f ,  g)_{k\phi}:=\int_{\ol D}\langle\,f \mid  g\,\rangle_k \,e^{-k\phi}dv_{M'}(\mathbf z).
$$
Here, $(\,\cdot\, ,\, \cdot\,)_{k\phi}$ is called the $L^2$ inner product on $\omz^{0,q}(\ol D)$ with weight $k\phi$. Denote by $L_{(0,q)}^2(D,k\phi)$ the completion of $\omz^{0,q}(\ol D)$ with respect to $(\,\cdot\, ,\, \cdot\,)_{k\phi}$.
Let $\dbar^{*,k\phi}$ be the formal adjoint of $\dbar$ with respect to $(\,\cdot\, ,\, \cdot\,)_{k\phi}$, then
$$
\Dom(\dbar^{*,k\phi})=\Big\{f=\sumprime_{J} f_J\ol\omega^J\in L_{(0,q)}^2(D,k\phi)\big| f_{J}=0 \mbox{ on $\partial D$ when $n\in J$}\Big\}. 
$$
Write
\begin{equation}\label{bdbarrho}
\dbar=\sum_{j=1}^{n}\bigr(\ol\omega^j\wedge \ol Z_j+(\dbar \ol\omega^j)\wedge (\ol\omega^j\wedge)^*\bigr),
\end{equation}
we have
\begin{equation}\label{bdbarstar}
\dbar^{*,k\phi}=\sum_{j=1}^{n}\bigr((\ol\omega^j\wedge)^*(- Z_j+kZ_j\phi+\alpha_j(\mathbf z))+\ol\omega^j\wedge (\dbar \ol\omega^j\wedge)^*\bigr),
\end{equation}
where $\alpha_j(\mathbf z)$ is a smooth function on $\ol D$, independent of $k$, for every $j=1,\ldots,n$. Set
\begin{equation}\label{be-gue210531yydu}
\Box_{k\phi}^q=\dbar^{*,k\phi}\dbar+\dbar\,\dbar^{*,k\phi}: \Dom(\Box_{k\phi}^q)\subset L_{(0,q)}^2(D,k\phi)\to L_{(0,q)}^2(D,k\phi).
\end{equation}
Let $f=\sumprime_{\abs{J}=q} f_J\ol\omega^J\in\omz^{0,q}(\ol D)$, then $f\in\Dom(\Box^q_{k\phi})$ gives that
\begin{equation}
\begin{cases}
\ f_J=0&\quad \text{on $\partial D$, if $n\in J$},\\
\ \ol Z_{n} f_J=0 &\quad \text{on $\partial D$, if $n\notin J$}.
%,\\Z_j(f_J)=\ol Z_j(f_J)=0 &\quad \text{on $\partial D$, if $j<n$ and $n\in J$}.
\end{cases}
\end{equation}
For $\hat f\in\omz^{0,q}_c(D,L^k)$, there exists a $ f\in \omz^{0,q}_c( D)$ such that 
$\hat f=s^k f$  on $D$. %Note that $\hat f=s^k f\in \Dom(\Box_k^q)$ if and only if $ f\in\Dom(\Box^q_{k\phi})$.  
Recall that $\dbar_{k}(s^k f)=s^k\dbar f$, we have
\begin{equation}\label{be-gue210303yyda}
\Box_{k}^q\hat f=s^k\Box_{k\phi}^q f.
\end{equation}
%From now on, on $D$, we identify $u$ with $\hat f$ and $\Box_{k}^q$ with $\Box_{k\phi}^q$.
%Until further notice, we assume that $Y(q)$ holds on $M$. 
Let
\begin{equation}\label{be-gue210305yyd}
A_{k\phi}(t,\mathbf z, \mathbf w)
:=\sumprime_{|I|=|J|=q} A_{k\phi,I,J}(t,\mathbf z, \mathbf w)\,\ol\omega^I(\mathbf z)\otimes\ol\omega^J(\mathbf w),
\end{equation}
where
$$
A_{k\phi,I,J}(t,\mathbf z, \mathbf w):=e^{\frac{k\phi(\mathbf z)}{2}}A_{k,s,I,J}(t,\mathbf z, \mathbf w)e^{-\frac{k\phi(\mathbf w)}{2}}
$$
and $A_{k,s}(t,\mathbf z, \mathbf w)$ is as in \eqref{be-gue230306yydII}. From \eqref{be-gue210303yydIII}, we have
\begin{equation}\label{kphik}
A_{k\phi}(t,\mathbf z,\mathbf z)=A_{k}(t,\mathbf z,\mathbf z)    
\end{equation}
on $D$. For $t>0$, let 
$$
A_{k\phi}(t): \mathscr E'( \ol D,T^{*0,q}M')\to\Omega^{0,q}( \ol D) 
$$
be the continuous operator with distribution kernel $A_{k\phi}(t,\mathbf z, \mathbf w)$ with respect to $dv_{M'}$, i.e.,
\begin{equation}\label{kbe-gue210305yydI}
A_{k\phi}(t) f({\bf z})=\int A_{k\phi}(t,\mathbf z, \mathbf w) f(\mathbf w)dv_{M'}(\mathbf w),\ \ \ f\in\Omega^{0,q}_c(D).
\end{equation} 
It follows from \eqref{be-gue210531ycda} and \eqref{be-gue210305yyd} that $A_{k\phi}(t) f$ satisfies the $\dbar$-Neumann boundary conditions of $\Box^q_{k\phi}$ on $\partial D$ and
\begin{equation}\label{kbe-gue210305yydII}
\begin{cases}
\left(\frac{\pa}{\pa t}+\Box^q_{k\phi}\right)A_{k\phi}(t)f=0,\\
\lim_{t\to0^+}A_{k\phi}(t)f=f \ \ \mbox{in $L^2_{(0,q)}(D,k\phi)$}.
\end{cases}
\end{equation}

\subsection{The scaling method}\label{scaledL} The primary goal of this section is to construct several scaled objects using scaling techniques and to study their properties. By doing so, we establish a local connection between the smooth kernel of the heat operators $e^{-\frac{t}{k}\Box_{k}^q}$ and $e^{-t\Box_{\phi_0}^q}$.

Let $B_R:=\{\mathbf z=[\xi,\theta,r]\in\C^{n}, |\xi_j|<R,|\theta|<R, -R<r<0, j=1,\cdots,n-1\}$, $\ol B_R:=\{\mathbf z=[\xi,\theta,r]\in\C^{n}, |\xi_j|<R,|\theta|<R, -R< r\le 0, j=1,\cdots,n-1\}$, and $\hat B_R:=\{\mathbf z=[\xi,\theta,r]\in\C^{n}, |\xi_j|<R,|\theta|<R, -R<r<R, j=1,\cdots,n-1\}$. 
$B_R$ and $\hat B_R$ can be identified with the subsets $D$ and $U$ respectively,  via the following holomorphic scaling map:
\begin{equation}\label{scalinging}
 \begin{split}
F_k: \mathbb C^{n}&\longrightarrow\mathbb C^{n},\\
[\xi,\theta,r]&\longmapsto\Big[\frac{\xi}{\sqrt{k}},\frac{\theta}{k},\frac{r}{k}\Big].
\end{split}   
\end{equation}
In fact, we can choose sufficiently large $k$ such that $F_k(B_{R})\subset D$ and $F_k(\hat B_R)\subset U$. %Note that the boundary of $ B_{\log k} $, denoted as $ bB_{\log k} $, is non-empty because of $ 0 \in \partial D \subset X $. 
 %In ambient coordinates $(z,z_{n})$, for $f=\sumprime f_Jd\ol z_J\in \omz^{0,q}(D)$, we define 
 %\begin{equation}
%f_{(k)}(\mathbf z):=\sumprime_{J}f_J\left(F_k\mathbf z\right)d\ol z_{J,(k)},
%\end{equation}
For $z\in\hat B_R$, set
\begin{equation}\label{psca}
\begin{split}
\ol \omega^j_{(k)}({\bf z})&:=\ol \omega^j(F_k {\bf z}),\quad j=1,\cdots,{n-1},\\
\ol \omega^{n}_{(k)}({\bf z})&:=\frac{1}{\sqrt k}\ol \omega^{n}(F_k {\bf z}).
\end{split}
\end{equation}
For $\mathbf z\in\hat B_R$, let 
\begin{equation}\label{scalfc}
T^{*0,q}_{(k),\mathbf z}M':=\left\{\sumprime_{|J|=q}a_J\ol\omega^J_{(k)}(\mathbf z)\,\bigg|\, a_J\in\C\right\}.
\end{equation}
Let $T_{(k)}^{*0,q}M'$ be the vector bundle over $\hat B_R$ with fiber $T^{*0,q}_{(k),\mathbf z}M'$ at ${\bf z}\in\hat B_R$. 
Define $\omz^{0,q}_{(k)}(\hat B_R)$ as the space of smooth sections of $T^{*0,q}_{(k)}M'$ over $\hat B_R$, and $\Omega^{0,q}_{(k),c}(\hat B_R)$ as a subspace of $\omz^{0,q}_{(k)}(\hat B_R)$ consisting of elements with compact support in $\hat B_R$.

Define 
\[\begin{split}
&\omz^{0,q}_{(k)}(\ol{B}_R)
:=\set{u|_{\ol B_R}: u\in\omz^{0,q}_{(k)}(\hat B_R)},\\
&\omz^{0,q}_{(k),c}(\ol{B}_R)
:=\set{u|_{\ol B_R}: u\in\omz^{0,q}_{(k),c}(\hat B_R)}.
\end{split}\]
For $f=\sumprime_{|J|=q}f_J\ol\omega^J\in \omz^{0,q}(F_k(\ol B_{R}))$, we define the scaled form $f_{(k)}(\mathbf z)$ by setting
 \begin{equation}
f_{(k)}(\mathbf z):=\sumprime_{J}f_J\left(F_k\mathbf z\right)\ol\omega^J_{(k)}({\bf z})\in \omz^{0,q}_{(k)}(\ol B_{R}).
\end{equation}
For a fixed $q\in\set{0,1,\ldots,n}$,  
we let $\langle\,\cdot\, \mid \,\cdot\,\rangle_{(k)}$  
be the scaled Hermitian metric on $T^{*0,q}_{(k)}M'
$ over $\hat B_{R}$, such that
\begin{equation*}
\begin{split}
&\langle\,\ol\omega_{(k)}^j\mid\ol\omega_{(k)}^l\,\rangle_{(k)}=\delta_{j,l},\quad j,l=1,\cdots,{n-1},\\
&\langle\,\ol\omega_{(k)}^{n}\mid\ol\omega_{(k)}^{n}\,\rangle_{(k)}=\frac{1}{k}.
\end{split}    
\end{equation*} 
Let $|\cdot|_{(k)}$ be the corresponding pointwise norm. On $\ol D$, we write the volume form  as
$dv_X(\mathbf z):=m(F_k\mathbf z)d\sigma(\mathbf z)$, where $m(\mathbf z)\in\cali{C}^\infty(D)$ and $d\sigma(\mathbf z):=2^{n}d x_1\wedge\cdots\wedge dx_{2n-2}\wedge d\theta\wedge dr$. Note that $k\phi_{(k)}\to \phi_0$ as $k\to\infty$, where $\phi_0$ is as in \eqref{mphi0}. Let $(\,\cdot \,, \,\cdot\,)_{(k)}$ be the $L^2$-inner product with weight $k\phi_{(k)}$ on $\Omega^{0,q}_{(k)}(\ol{B}_R)$, given by 
$$
(\,f ,  g\,)_{(k)}=\int_{\ol B_{R}}\langle\,f \mid g\,\rangle_{(k)}e^{-k\phi_{(k)}}m(F_k\mathbf z)d\sigma(\mathbf z),\ \ f, g\in \Omega^{0,q}_{(k)}(\ol B_{R}).
$$
Similarly, let $f, g\in\Omega^{0,q}_{(k)}(\hat B_R)$, we define 
$$
(\,f ,  g\,)_{(k),\hat B_R}:=\int_{\hat B_R}\langle\,f \mid g\,\rangle_{(k)}e^{-k\phi_{(k)}}m(F_k\mathbf z)d\sigma(\mathbf z).
$$  
Let 
$$ 
\ol Z_{j,(k)}=\frac{\pa}{\pa \ol \xi_j}+i\lambda_j\xi_j\frac{\pa}{\pa\theta}+\frac{1}{\sqrt k}O(|F_k\mathbf z|^2 ),\quad j=1,\cdots,{n-1},
$$
and
$$
\ol Z_{n,(k)}= \sqrt{\frac{k}{2}}\Big(\frac{\pa}{\pa r}-i\frac{\pa}{\pa \theta}\Big)+\frac{1}{\sqrt k}O(|F_k\mathbf z|^2 ).
$$
Let
$\dbar_{(k)}:\omz^{0,q}_{(k)}(\ol B_{R})\to \omz^{0,q+1}_{(k)}(\ol B_{R})$ be the scaled differential operator, given by
\begin{equation}\label{bdbarrhok}
\begin{aligned}
\dbar_{(k)}&=\sum_{j=1}^{n}\ol\omega^j_{(k)}\wedge \ol Z_{j,(k)}+\frac{1}{\sqrt k}\sum_{j=1}^{n}(\dbar\ol\omega^j)_{(k)}\wedge \big(\ol\omega^j_{(k)}\wedge\big)^*.
\end{aligned}
\end{equation}
Compare \eqref{bdbarrho} with \eqref{bdbarrhok}, one can check that
\begin{equation}\label{bk1}
\dbar_{(k)}f_{(k)}=\frac{1}{\sqrt k}(\dbar f)_{(k)},\quad \mbox{ for all }f\in\omz^{0,q}(F_k(\ol B_{R})).
\end{equation}
Let $\dbar_{(k)}^*$ be the formal adjoint of $\dbar_{(k)}$ with respect to $(\,\cdot\, , \,\cdot\,)_{(k)}$.
%Then
%$$
%\Dom(\dbar_{(k)}^*)=\{f\in \omz^{0,q}_{(k)}(\ol B_{\log k})\mid f_J=0 \mbox{ on $\partial B_{\log k}$ when $n\in J$}\},
%$$
We have 
\begin{equation}\label{bdbarstarrhok}
\dbar_{(k)}^*=\sum_{j=1}^{n}\big(\ol\omega^j_{(k)}\wedge\big)^*\Big(-Z_{j,(k)}+\sqrt k(Z_j\phi)_{(k)}+\frac{1}{\sqrt k}(\alpha_j)_{(k)}\Big)+\frac{1}{\sqrt k}\sum_{j=1}^{n}\ol\omega^j_{(k)}\wedge \left((\dbar \ol\omega^j)_{(k)}\wedge\right)^*, 
\end{equation}
where $\alpha_j=\alpha_j(\mathbf z)$ is a smooth function as in \eqref{bdbarstar}. Moreover, we have
\begin{equation}\label{bk2}
\dbar^{*}_{(k)}f_{(k)}=\frac{1}{\sqrt k}\big(\dbar^{*,k\phi}f\big)_{(k)},\quad\mbox{ for all } f\in\Omega^{0,q}(F_k(\ol B_{R}))\cap\dom(\dbar^{\star,k\phi}).
\end{equation}
We now define the scaled Kodaira Laplacian as
\begin{equation}\label{be-gue210614yydI}
\Box^q_{(k)}=\dbar_{(k)}^*\dbar_{(k)}+\dbar_{(k)}\dbar_{(k)}^*
\end{equation}
on $\Omega_{(k)}^{0,q}({\ol B}_{R})$.  
Then \eqref{bk1} and \eqref{bk2} yield that
\begin{equation}\label{k(k)}
\Box_{(k)}^qf_{(k)}=\frac{1}{k}(\Box_{k\phi}^qf)_{(k)},\quad\mbox{ for all } f\in\Omega^{0,q}(F_k(\ol B_{R})). 
\end{equation} 
For $t>0$, $\mathbf z,\mathbf w\in {\ol B}_{R}$, let 
\begin{equation}\label{be-gue210325yyd}
\begin{split}
&A_{(k)}(t,\mathbf z, \mathbf w)
:=\sumprime_{|I|=|J|=q} A_{(k),I,J}(t,\mathbf z, \mathbf w)\,\ol\omega^I_{(k)}(\mathbf z)\otimes\ol\omega^J_{(k)}(\mathbf w),\\
&\mbox{where\ \ }A_{(k),I,J}(t,\mathbf z, \mathbf w):=k^{-(n+1)}A_{k\phi,I,J}(\frac{t}{k},F_k\mathbf z, F_k\mathbf w),
\end{split}
\end{equation}
where $F_k\mathbf z=\big(\frac{z}{\sqrt k}, \frac{z_{n}}{k} \big) $, $F_k\mathbf w=\big(\frac{w}{\sqrt k}, \frac{w_{n}}{k} \big)$, $z,w\in\mathbb C^{n-1}$, $A_{k\phi}(t,\mathbf z, \mathbf w)$ is as in
\eqref{be-gue210305yyd}. 
Let 
$$A_{(k)}(t): \Omega^{0,q}_{(k),c}(B_{R})\to \Omega^{0,q}_{(k)}({\ol B}_{R})$$ 
be the continuous operator given by 
\begin{equation}\label{be-gue210325yydI}
(A_{(k)}(t)f)(\mathbf z)=\int A_{(k)}(t,\mathbf z, \mathbf w)f(\mathbf w)m(F_k\mathbf w)dv(\mathbf w),\ \ f\in \Omega^{0,q}_{(k),c}(B_{R}).
\end{equation}
Let $f\in \Omega^{0,q}_{(k),c}(B_{R})$ and $g\in\Omega^{0,q}_c(F_k(B_{R}))$ be such that $f=g_{(k)}$ on $B_{R}$, we have
\begin{equation}\label{be-gue230325yydI}
\begin{aligned}
&(A_{(k)}(t)f)(\mathbf z)\\
&=\frac{1}{k^{n+1}} \sumprime_{I,J}\int_{\ol B_{R}} A_{k\phi,I,J}(\frac{t}{k},F_k\mathbf z, F_k\mathbf w)g_J(F_k\mathbf w)\langle\ol\omega_{(k)}^J(\mathbf w)\mid\ol\omega_{(k)}^J(\mathbf w)\rangle_{(k)}m(F_k\mathbf w)d\sigma(\mathbf w)\,\ol\omega^I_{(k)}(\mathbf z)\\ 
&=\Big(\sumprime_{\substack{I,J \\ n\notin J}}+\frac{1}{k}\sumprime_{\substack{I,J \\ n\in J}}\Big)\int_{F_k(\ol B_{R})} A_{k\phi,I,J}(\frac{t}{k},F_k\mathbf z, \mathbf w)g_J(\mathbf w)\, m(\mathbf w)d\sigma(\mathbf w)\,\ol\omega^I_{(k)}(\mathbf z)\\
&=\sumprime_{I,J}\int_{F_k(\ol B_{R})} A_{k\phi,I,J}(\frac{t}{k},F_k\mathbf z, \mathbf w)g_J(\mathbf w)\langle\ol\omega^J(\mathbf w)\mid\ol\omega^J(\mathbf w)\rangle_k\, m(\mathbf w)d\sigma(\mathbf w)\,\ol\omega^I_{(k)}(\mathbf z)\\ 
&=\sumprime_{|I|=q}\big (A_{k\phi}(\frac{t}{k})g\big )_I(F_k\mathbf z)\,\ol\omega^I_{(k)}(\mathbf z)=(A_{k\phi}(\frac{t}{k})g\big )_{(k)}(\mathbf z).
\end{aligned}
\end{equation}
It then follows from \eqref{k(k)} that
\begin{equation}\label{shk}
 \Big(\frac{\pa}{\pa t}+\Box^q_{(k)}\Big)A_{(k)}(t)f=\Big(\frac{\pa}{\pa t}+\Box^q_{(k)}\Big)(A_{k\phi}(t/k)g)_{(k)}=\Big(\frac{\pa}{\pa t}+\frac{1}{k}\Box^q_{k\phi}\Big)A_{k\phi}(t/k)g=0.
\end{equation}
Recalling \eqref{bdbarrhok}, we observe that
\begin{equation}\label{bstaru3}
\ol Z_{j,(k)}=\frac{\pa}{\pa \ol \xi_j}+i\lambda_j\xi_j\frac{\pa}{\pa\theta}+\epsilon_kU_{j,k},\ \ j=1,\ldots,{n-1},
\end{equation}
and
\begin{equation}\label{bstaru3n}
\ol Z_{n,(k)}= \sqrt{\frac{k}{2}}\Big(\frac{\pa}{\pa r}-i\frac{\pa}{\pa \theta}\Big)+\epsilon_kU_{n,k}.
\end{equation}
on $\ol B_{R}$, where $\epsilon_k$ is a sequence tending to zero as $k\to\infty$, and $U_{j,k}$ is a first-order differential operator whose coefficients are uniformly bounded in $k$ on $\ol B_{R}$.
On the other hand, from \eqref{bphi}, we see that in \eqref{bdbarstarrhok},
\begin{equation}\label{bstaru4}
\begin{aligned}
-Z_{j,(k)}+\sqrt k(Z_j\phi)_{(k)}+\frac{1}{\sqrt k}(\alpha_j)_{(k)}=-\frac{\pa}{\pa \ol \xi_j}+i\lambda_j\xi_j\frac{\pa}{\pa\theta}+\sum_{l=1}^{n-1}\mu_{j,l}\ol  \xi_l+\delta_kV_{j,k},\ \ j=1,\ldots,{n-1}, 
\end{aligned}
\end{equation}
and
\begin{equation}\label{bstaru4n}
\begin{aligned}
-Z_{n,(k)}+\sqrt k(Z_{n}\phi)_{(k)}+\frac{1}{\sqrt k}(\alpha_{n})_{(k)}=-\sqrt{\frac{k}{2}}\Big(\frac{\pa}{\pa r}+i\frac{\pa}{\pa \theta}\Big)+\delta_kV_{n,k}
\end{aligned}
\end{equation}
where $\delta_k$ approaches zero as $k\to\infty$, and $V_{j,k}$ is a first-order differential operator with coefficients uniformly bounded in $k$ on $\ol B_{R}$. Combining this with \eqref{bdbarrhok}, \eqref{bdbarstarrhok}, and \eqref{bstaru3}, we obtain the following proposition:
\begin{prop}\label{boxsk}
We have that
\begin{equation}\label{bboxrhok}
\begin{aligned}
\Box_{(k)}^{q}=&\sum_{j=1}^{{n-1}}\left(-\frac{\pa}{\pa \xi_j}+i\lambda_j\ol \xi_j\frac{\pa}{\pa\theta}+\sum_{l=1}^{n-1}\mu_{j,l}\ol  \xi_l\right)\left(\frac{\pa}{\pa \ol \xi_j}+i\lambda_j \xi_j\frac{\pa}{\pa\theta}\right)-\frac{1}{2}\left(\frac{\pa^2}{\pa\theta^2}+\frac{\pa^2}{\pa r^2}\right)\\
&+\sum_{j,l=1}^{{n-1}}
\left(2i\lambda_j\frac{\pa}{\pa\theta}+\mu_{l,j}\right)\ol\omega^j_{(k)}\wedge \big( \ol\omega^l_{(k)}\wedge\big)^*+\varepsilon_kP_k^q
\end{aligned}
\end{equation}
on $\ol B_{R}$, where $\varepsilon_k$ is a sequence tending to zero as $k\to\infty$ and $P_k^q$ is a second order differential operator and all the 
derivatives of the coefficients of $P_k^q$ are uniformly bounded in $k$ on $\ol B_{R}$.
\end{prop}
\begin{proof}
Let $\ol Z_{j,(k)}^*:=-Z_{j,(k)}+\sqrt k(Z_j\phi)_{(k)}+\frac{1}{\sqrt k}(\alpha_j)_{(k)}$ for $j=1,\cdots,n$, we have
$$
\begin{aligned}
\Box_{(k)}^{q}&=\dbar_{(k)}^*\dbar_{(k)}+\dbar_{(k)}\dbar_{(k)}^*=\sum_{j,l=1}^{n}\left(\ol\omega^j_{(k)}\wedge\big(\ol\omega^l_{(k)}\wedge\big)^*+\big(\ol\omega^l_{(k)}\wedge\big)^*\ \ol\omega^j_{(k)}\wedge \right)\ol Z_{l,(k)}^*\ol Z_{j,(k)}\\
&\quad+\sum_{j=1}^{n}\ol\omega^j_{(k)}\wedge\big(\ol\omega^l_{(k)}\wedge\big)^*\big[\,\ol Z_{j,(k)},\ol Z_{l,(k)}^*\,\big]+\hat\eps_k\big(\,\ol Z_{(k)}+\ol Z_{(k)}^*+\mbox{zero order term}\big),
\end{aligned}
$$
where $\hat\varepsilon_k$ is a sequence tending to zero as $k\to\infty$.
Note that
$$
\ol\omega^j_{(k)}\wedge\big(\ol\omega^l_{(k)}\wedge\big)^*+\big(\ol\omega^l_{(k)}\wedge\big)^*\ \ol\omega^j_{(k)}\wedge=
\Big\{\begin{array}{lr}
        0, \quad \text{for } j\neq l \text{ and } j,l=1,\cdots, n\\
        1,  \quad \text{for } j=l\text{ and } j,l=1,\cdots, {n-1}
        \end{array}
$$
and 
$$
\ol\omega^{n}_{(k)}\wedge\big(\ol\omega^{n}_{(k)}\wedge\big)^*+\big(\ol\omega^{n}_{(k)}\wedge\big)^*\ \ol\omega^{n}_{(k)}\wedge=\frac{1}{k}.
$$
Using these results, we simplify the expression further:
$$
\begin{aligned}
\Box_{(k)}^{q}
&=\sum_{j=1}^{{n-1}}\ol Z_{j,(k)}^*\ol Z_{j,(k)}+\frac{1}{k}\ol Z_{n,(k)}^*\ol Z_{n,(k)}+\sum_{j,l=1}^{n}\ol\omega^j_{(k)}\wedge\big(\ol\omega^l_{(k)}\wedge\big)^*\big[\,\ol Z_{j,(k)},\ol Z_{l,(k)}^*\,\big]\\
&\quad+\eps_k\big(\,\ol Z_{(k)}+\ol Z_{(k)}^*+\mbox{zero order term}\big).
\end{aligned}
$$
The desired expression \eqref{bboxrhok} is then follows from \eqref{bstaru3}-\eqref{bstaru4n}.
\end{proof} 

Let $s\in\mathbb Z$. We denote by
$W^s_{(k)}\big(\hat B_R,T_{(k)}^{*0,q}M'\big)$
the Sobolev space of order $s$ of sections of
$T_{(k)}^{*0,q}M'$ over $\hat B_R$.
We use the same convention as in \eqref{sobolevnorm} in the rescaled
coordinates. The corresponding metric
and connection will be
denoted by $\langle\,\cdot\,,\,\cdot\,\rangle_{(k)}$ and $\nabla_{(k)}$, respectively. If $s\in\mathbb N_0$, we define
\begin{equation}
\begin{aligned}
\|f\|^2_{(k),s,\hat B_R}
:=\sum_{\ell=0}^s
\int_{\hat B_R}
\left|\nabla_{(k)}^\ell f({\bf z})\right|_{(k)}^2
e^{-k\phi_{(k)}({\bf z})}
m(F_k{\bf z})\,d\sigma({\bf z}),
\quad
f\in W^s_{(k)}\bigl(\hat B_R,T_{(k)}^{*0,q}M'\bigr).
\end{aligned}
\end{equation}
If $s\leq0$, the Sobolev norm of $f\in W^s_{(k)}\big(\hat B_R, T_{(k)}^{*0,q}M'\big)$ is given by
$$\|f\|_{(k),s,\hat B_R}:=\sup\set{\abs{(\,f\,,\,g\,)_{(k),\hat B_R}}: g\in W^{-s}_{(k)}\big(\hat {B}_R, T_{(k)}^{*0,q}M'\big), \norm{g}_{(k),-s,\hat B_R}=1}.$$
For $s\in\mathbb Z$, let 
$$W^s_{(k)}\big(\ol{B}_R, T_{(k)}^{*0,q}M'\big):=\set{u|_{\ol{B}_R}: u\in W^s_{(k)}\big(\hat B_R, T_{(k)}^{*0,q}M'\big)}.$$
The Sobolev norm of $f\in W^s_{(k)}\big(\ol{B}_R, T_{(k)}^{*0,q}M'\big)$ is given by
$$
\|f\|_{(k),s,\ol{B}_R}:=\inf\set{\|\tilde f\|_{(k),s\hat B_R}: \tilde f\in W^s_{(k)}\big(\hat B_R, T_{(k)}^{*0,q}M'\big), \tilde f|_{\ol{B}_R}=f}.
$$
%$W^s_{(k)}\big(\ol B_R, %T_{(k)}^{*0,q}B_{R}\big)$
%as the $L^2$-Sobolev space of order $s$ on the sections of $T_{(k)}^{*0,q}B_{R}$ over $\ol B_R$ with respect to $(\,\cdot \,, \,\cdot\,)_{(k)}$. Here the Sobolev norm of $f=\sumprime_{|J|=q}f_J\ol\omega^J_{(k)}\in W^s_{(k)}\big(\ol B_R, T_{(k)}^{*0,q}B_{R}\big)$ is given by
%\begin{equation}
%\begin{aligned}
%\|f\|^2_{(k),s,\ol B_R}=\sum_{|\alpha|\le s}\int_{\ol B_R}|\pa_{x}^\alpha f|_{(k)}^2e^{-k\phi_{(k)}}m(F_k{\bf z})d\sigma({\bf z}).
%\end{aligned}
%\end{equation}
%Let $s\in\mathbb N_0$. We define $W^s_{(k)}\big(\ol B_R, T_{(k)}^{*0,q}B_{R}\big)$
%as the $L^2$-Sobolev space of order $s$ on the sections of $T_{(k)}^{*0,q}B_{R}$ over $\ol B_R$ with respect to $(\,\cdot \,, \,\cdot\,)_{(k)}$. Here the Sobolev norm of $f=\sumprime_{|J|=q}f_J\ol\omega^J_{(k)}\in W^s_{(k)}\big(\ol B_R, T_{(k)}^{*0,q}B_{R}\big)$ is given by
%\begin{equation}
%\begin{aligned}
%\|f\|^2_{(k),s,\ol B_R}=\sum_{|\alpha|\le s}\int_{\ol B_R}|\pa_{z}^\alpha f|_{(k)}^2e^{-k\phi_{(k)}}m(F_k{\bf z})d\sigma({\bf z}).
%\end{aligned}
%\end{equation}
We also write $\|\cdot\|_{(k),\ol B_R}:=\|\cdot\|_{(k),0,\ol B_R}$.
From Proposition \ref{boxsk} and Kohn's $L^2$ estimate, we have
	
\begin{prop}\label{p-gue210511yyd}
Assume that condition $Z(q)$ holds. Let $s\in\mathbb N_0$. 
Fix $0<\delta\ll1$. There exists a constant $C_{R,s}>0$ that is independent of $k$ such that for all $f\in \omz_{(k)}^{0,q}(\ol {B}_{R+\delta})$, we have
\begin{equation}\label{subellipticest}
\begin{aligned}
\|f\|^2_{(k),s+1,\ol{B}_R}\le C_{R,s}\left(\|f\|^2_{(k), \ol {B}_{R+\delta}}+\|\Box_{(k)}^qf\|^2_{(k),s, \ol {B}_{R+\delta}}\right).
\end{aligned}
\end{equation}
\end{prop}

Our next goal is to apply the subelliptic estimates derived above to establish the following proposition: %It asserts that, for every compact interval $I\Subset\mathbb R_+$, the coefficients of $A_{(k)}(t,\mathbf z,\mathbf w)$ and all their derivatives on $I\times \ol B_R\times \ol B_R$ are uniformly bounded in $k$ and $t\in I$, with constants depending on $I$.

\begin{prop}\label{uniformly}
Let $I$ be a compact interval in $\mathbb R_+$. Assume that condition $Z(q)$ holds on $X$. Then, for every $\ell \in \mathbb{N}_0$, there exists a constant $C_{I,\ell,R} > 0$, independent of $k$ and $t\in I$, such that 
$$
\|A_{(k)}(t, \mathbf z, \mathbf w)\|_{\cali{C}^\ell(I \times \ol B_R \times \ol B_R, T_{(k)}^{*0,q}M' \boxtimes (T_{(k)}^{*0,q}M')^*)} \leq C_{I,\ell,R}.
$$
\end{prop}

\begin{proof}[Sketch of Proof]
The proof follows the approach in \cite[Proposition 3.6]{HZ23}, where a similar uniform estimate is derived for related operators. We outline the main steps:\\
Step 1. Let $f\in \Omega^{0,q}_{(k),c}(\ol{B}_R)$ and $g\in\Omega^{0,q}_c(F_k(\ol B_R))$ be such that $f=g_{(k)}$ on $\ol{B}_R$. Using the structure of $\Box^q_{(k)}$ and Lemma \ref{e-gue210331yyd}, we derive
\begin{equation}\label{e-gue210412yyd}
\begin{split}
&\|(\Box^q_{(k)})^\ell A_{(k)}(t)f\|^2_{(k),\ol{B}_R}=\Big\|k^{-\ell}\Bigr((\Box^q_{k\phi})^\ell A_{k\phi}(\frac{t}{k})g\Bigr)_{(k)}\Big\|^2_{(k),\ol{B}_R}\\
&=k^{-2\ell+n+1}\Big\|(\Box^q_{k\phi})^\ell A_{k\phi}(\frac{t}{k})g\Big\|^2_{k\phi,F_k(\ol B_R)}
= k^{-2\ell+n+1}\big\|( \Box^q_{k})^\ell e^{-\frac{t}{k} \Box^q_{k}}(s^kg)\big\|^2_{k}\\
&\leq\frac{k^{n+1}}{t^{2\ell}}C_\ell\|s^kg\|^2_{k}=\frac{k^{n+1}}{t^{2\ell}}C_\ell\|g\|^2_{k\phi,F_k(\ol B_R)}
=\frac{\hat C_{I,\ell}}{t^{2\ell}}\|f\|^2_{(k),\ol{B}_R},
\end{split}
\end{equation}
where $C_\ell > 0$ is independent of $k$ and $t$, and $\hat C_{I,\ell}>0$ is independent of $k$ and $t\in I$. Here we use the fact that $I\Subset\mathbb R_+$, so that $t^{-2\ell}$ is uniformly bounded on $I$. \\
Step 2. We define the following Sobolev spaces:
\[\begin{split}
&W^s_{(k),c}\big(\hat B_R, T_{(k)}^{*0,q}M'\big) := \{f \in W^s_{(k)}(\hat B_R, T_{(k)}^{*0,q}M') : {\rm supp\,}f \Subset\hat B_R\},\\
&W^s_{(k),c}\big(\ol{B}_R, T_{(k)}^{*0,q}M'\big) := 
\set{u|_{\ol{B}_R}: u\in W^s_{(k),c}\big(\hat B_R, T_{(k)}^{*0,q}M'\big)},\\
&W^s_{(k),{\rm loc}}\big(\hat B_R, T_{(k)}^{*0,q}M'\big) := \{f \in \mathscr{D}'(\hat B_R, T_{(k)}^{*0,q}M') : \chi f \in W^s_{(k)}(\hat B_R, T_{(k)}^{*0,q}M'), \, \forall\chi \in \cali{C}^\infty_c(\hat B_R)\},\\
&W^s_{(k),{\rm loc}}\big(\ol{B}_R, T_{(k)}^{*0,q}M'\big):=\set{u|_{\ol{B}_R}: u\in W^s_{(k),{\rm loc}}\big(\hat B_R, T_{(k)}^{*0,q}M'\big)}.
\end{split}
\]
Step 3. Using Proposition~\ref{p-gue210511yyd} and the bounds derived in Step 1, we show that: 
\begin{equation}\label{e-gue210412yydI}
A_{(k)}(t): W^0_{(k),{c}}(\ol{B}_R,T_{(k)}^{*0,q}M')\to W^{s}_{(k),{\rm loc}}(\ol{B}_R,T_{(k)}^{*0,q}M')
\end{equation}
is continuous for every $s\in\mathbb Z$, with the continuity being uniform in $t\in I$. By repeating the argument for $A_{(k)}(t)(\Box^q_{(k)})^{\ell_1}$ for any $\ell_1 \in \mathbb{N}_0$, and taking the adjoint, we further establish that $A_{(k)}(t)$ extends continuously to $W^{-\ell}_{(k),c}(\ol{B}_R, T_{(k)}^{*0,q}M')$, and the map
$$
A_{(k)}(t): W^{-\ell}_{(k),c}(\ol{B}_R, T_{(k)}^{*0,q}M') \to W^s_{(k),{\rm loc}}(\ol {B}_R, T_{(k)}^{*0,q}M')
$$
is continuous for every $s \in \mathbb{Z}$, with the continuity being uniform in $t \in I$.\\
Step 4. It remains to explain the estimates involving derivatives in $t$. From the heat equation \eqref{shk}, we have, for every $a\in\mathbb N_0$,
\[
\frac{\partial^a}{\partial t^a}A_{(k)}(t)=(-\Box^q_{(k)})^aA_{(k)}(t).
\]
Thus the estimates in Step 3, applied to $(\Box^q_{(k)})^aA_{(k)}(t)$, together with the same adjoint argument for the $\mathbf w$-variable, give uniform bounds for all mixed derivatives in $(t,\mathbf z,\mathbf w)$ on $I\times \ol B_R\times \ol B_R$. The constants depend on $I,\ell$ and $R$, but are independent of $k$ and $t\in I$.\\
Step 5. Combining steps 3 and 4 with the Sobolev embedding theorem,  we complete the proof. For further details, refer to \cite{HZ23}.
\end{proof}

We now turn to the convergence of $A_{(k)}(t)$. Recall that, for fixed $q$, $A_{\phi_0}(t,\mathbf z,\mathbf w)=e^{-t\Box_{\phi_0}^q}(\mathbf z,\mathbf w)$ denotes the heat kernel of the model Laplacian; see Section~\ref{M0}. To formulate the convergence precisely, we regard $A_{\phi_0}(t,\mathbf z,\mathbf w)$ and $A_{(k)}(t,\mathbf z,\mathbf w)$ as smooth sections over $\mathbb R_+\times\overline B_R\times\overline B_R$ with values in $T^{*0,q}\mathbb C^n\boxtimes (T^{*0,q}\mathbb C^n)^*$.
The precise statement is as follows.

\begin{thm}\label{t-gue210503yyd}
Assume that condition $Z(q)$ holds. Let $I\Subset\mathbb R_+$ be a compact interval. Then we have 
\begin{equation}\label{limits}
\lim_{k\to\infty}A_{(k)}(t,\mathbf z, \mathbf w)=A_{\phi_0}(t,\mathbf z, \mathbf w)  \end{equation}
in $\cali{C}^\infty(I\times \ol{B}_R\times \ol{B}_R,\Lambda^\bullet(\mathbb CT^*\mathbb C^n)\boxtimes(\Lambda^\bullet(\mathbb CT^*\mathbb C^n))^*)$ topology.
\end{thm} 
	
\begin{proof}
By combining Proposition~\ref{uniformly} with the Arzela-Ascoli theorem and the Cantor diagonalization argument, we can select a subsequence $\{k_1 < k_2 < \cdots\} \subset \mathbb{N}$ with $\lim_{j\to+\infty} k_j = +\infty$ such that
$$
\lim_{j\to+\infty} A_{(k_j)}(t,\mathbf z, \mathbf w) = Q(t,\mathbf z, \mathbf w)
$$
locally uniformly on compact subsets of $\mathbb{R}_+ \times \overline{M}_0 \times \overline{M}_0$ with respect to the $\cali{C}^\infty$ topology.   Set 
$$Q(t,\mathbf z, \mathbf w):=\sumprime_{|I|=|J|=q}Q_{I,J}(t,\mathbf z, \mathbf w)\ol\omega_0^I(\mathbf z)\otimes \ol\omega_0^J(\mathbf w),
$$ 
where $\ol\omega_0^j=d\ol \xi_j$ for $1\le j\le {n-1}$ and $\ol\omega_0^{n}=\sqrt 2\,\dbar r_0$ with $r_0=\text{Im}z_{n}+\sum_{j=1}^{{n-1}}\lambda_j|z_j|^2.
$
Then
$Q(t,\mathbf z, \mathbf w)\in\cali{C}^\infty(\mathbb R_+\times \ol M_0\times \ol M_0,T^{*0,q}\C^n\boxtimes(T^{*0,q}\C^n)^*)$. Consider the continuous operator
$Q(t): \Omega^{0,q}_c(M_0)\to\Omega^{0,q}(M_0)$ given by 
$$(Q(t)f)(\mathbf z)=\int Q(t,\mathbf z, \mathbf w)f(\mathbf w)dv(\mathbf w),\ \ f\in\Omega^{0,q}_c(M_0).$$
Observe that 
\begin{equation*}
Q_{I,J}(t,\mathbf z, \mathbf w)=\begin{dcases}
\displaystyle\lim_{j\to+\infty}A_{(k_j),I,J}(t,\mathbf z, \mathbf w)\quad&\text{when $n\notin I\cap J$},\\
\displaystyle\lim_{j\to+\infty}\frac{1}{\sqrt{k_j}}A_{(k_j),I,J }(t,\mathbf z, \mathbf w)&\text{when $n\in (I\cup J)/(I\cap J)$},\\
\displaystyle\lim_{j\to+\infty}\frac{1}{k_j}A_{(k_j),I,J}(t,\mathbf z, \mathbf w)&\text{when $n\in I\cap J$}.
\end{dcases}    
\end{equation*}
We claim that  for every $f\in\Omega^{0,q}_c(M_0)$ and $t>0$, we have
\begin{equation}\label{e-gue210503ycd}
\begin{split}
Q(t)f\in{\rm Dom\,}\Box^q_{\phi_0},\quad Q'(t)f+\Box^q_{\phi_0}Q(t)f=0, \quad \mbox{and } \lim_{t\to0^{+}} Q(t)f=f.
\end{split}
\end{equation}
By \eqref{shk}, we have
\begin{equation}\label{e-gue210504yyd}
A'_{(k)}(t)+\Box^q_{(k)}A_{(k)}(t)=0\ \ \mbox{on $\ol B_{R}$}.
\end{equation}
Using \eqref{mboxrhok}, \eqref{bboxrhok}, and passing to the limit $k\to\infty$ in \eqref{e-gue210504yyd} on compact subsets of $\mathbb R_+\times \ol M_0\times \ol M_0$, we obtain 
$$
Q'(t)f+\Box^q_{\phi_0}Q(t)f=0. 
$$
Write $f\in\Omega^{0,q}_c(M_0)$ as $f=\sum'_{\abs{J}=q}f_J({\bf z})\,\ol\omega_0^J$ with $f_J\in\cali{C}^\infty_c(M_0)$. Choose $R\gg1$ so that ${\rm supp\,}f\subset B_R$. We set
\begin{equation}\label{e-gue210504yydbc}
f_k:=\sideset{}{'}\sum_{n\notin J}f_J(\mathbf z)\overline\omega^J_{(k)}+\sqrt k\sideset{}{'}\sum_{n\in J}f_J(\mathbf z)\overline\omega^J_{(k)}, \quad k=1,2,\cdots.
\end{equation}
Then $f_k\in \Omega^{0,q}_{(k),c}(B_R)$ for sufficiently large $k$. %Since $\Box^q_{(k)}$ is a differential operator, we also have ${\rm supp\,}\Box^q_{(k)}f_k\subset B_R$ for all sufficiently large $k$. 
It can be verified that $\|f\|_{\phi_0,\ol B_R}^2=\displaystyle\lim_{k\to\infty}\|f_{k}\|_{(k),\ol B_R}^2$.
 Note that 
\begin{equation*}
A_{(k)}(t)f_{k}=\sumprime_{I}\int_{B_R} \Big(\sumprime_{n\notin J}+\frac{1}{\sqrt k}\sumprime_{n\in J}\Big) A_{(k),I,J}(t,\mathbf z, \mathbf w)f_J(\mathbf w)\, m(F_{k}\mathbf w)d\sigma(\mathbf w)\,\ol\omega^I_{(k)}(\mathbf z).
\end{equation*}
For every $t>0$, from the definition of $A_{(k)}(t)$, we have that
$$
\int_{B_R} \Big(\sumprime_{n\notin J}+\frac{1}{\sqrt k}\sumprime_{n\in J}\Big) A_{(k),I,J}(t,\mathbf z, \mathbf w)f_J(\mathbf w)\, m(F_{k}\mathbf w)d\sigma(\mathbf w)=0
$$
on $\hat B_R\cap X_0$ when $n\in I$, and
$$
\ol Z_{n,(k)}\int_{B_R} \Big(\sumprime_{n\notin J}+\frac{1}{\sqrt k}\sumprime_{n\in J}\Big) A_{(k),I,J}(t,\mathbf z, \mathbf w)f_J(\mathbf w)\, m(F_{k}\mathbf w)d\sigma(\mathbf w)=0
$$
on $\hat B_R\cap X_0$ when $n\notin I$. Observe that
\begin{equation*}
\begin{aligned}
 &Q(t)f=\sumprime_{I}\int_{B_R} \sumprime_{J} Q_{I,J}(t,\mathbf z, \mathbf w)f_J(\mathbf w)\, dv(\mathbf w)\,\ol\omega^I_{0}(\mathbf z)\\
&=\lim_{j\to+\infty}\Big(\sumprime_{n\notin I}+\frac{1}{\sqrt {k_j}}\sumprime_{n\in I}\Big)\int_{B_R} \Big(\sumprime_{n\notin J}+\frac{1}{\sqrt{k_j}}
\sumprime_{n\in J}\Big) A_{(k_j),I,J}(t,\mathbf z, \mathbf w)f_J(\mathbf w)\, dv(\mathbf w)\,\ol\omega^I_{0}(\mathbf z)   
\end{aligned}
\end{equation*}
Hence $Q(t)f$ satisfies the $\dbar$-Neumann boundary condition on $X_0$. Moreover, for each $\ell\in\mathbb N_0$, we have 
\begin{equation}\label{e-gue210504yydI}
\|(\Box^q_{\phi_0})^\ell Q(t)f\|^2_{\phi_0,\ol B_R}=\lim_{j\to+\infty}\|(\Box^q_{(k_j)})^\ell A_{(k_j)}(t)f_{k_j}\|^2_{(k_j),\ol B_R}.
\end{equation}
From \eqref{e-gue210412yyd} and \eqref{e-gue210504yydI}, we conclude that there is a constant $C_\ell>0$ independent of $t$ and $R$ such that 
\begin{equation}\label{e-gue210504yydII}
\|(\Box^q_{\phi_0})^\ell Q(t)f\|_{\phi_0,\ol B_R}\leq\frac{C_\ell}{t^\ell}\|f\|_{\phi_0,\ol B_R}, 
\end{equation}
which implies
$\|(\Box^q_{\phi_0})^\ell Q(t)f\|_{\phi_0,\ol B_R}\leq\frac{C_\ell}{t^\ell}\|f\|_{\phi_0}$ for every $R\gg1$. Let $R\to+\infty$, we get 
\begin{equation}\label{e-gue210504yydIII}
\|(\Box^q_{\phi_0})^\ell Q(t)f\|_{\phi_0}\leq\frac{C_\ell}{t^\ell}\|f\|_{\phi_0},
\end{equation}
which gives  $Q(t)f\in{\rm Dom\,}\Box^q_{\phi_0}$.

To complete the claim \eqref{e-gue210503ycd}, it remains to show $\lim_{t\to0^{+}} Q(t)f=f$. Let $f\in\Omega^{0,q}_c(M_0)$ and let $f_k\in \Omega^{0,q}_{(k),c}(B_R)$ be as defined in \eqref{e-gue210504yydbc}. For every $t>0$, we have
\begin{equation}\label{e-gue210504yydb}
A_{(k)}(t)f_k-f_k=\int^t_0A'_{(k)}(s)f_kds=-\int^t_0\Box^q_{(k)}A_{(k)}(s)f_kds=-\int^t_0A_{(k)}(s)\Box^q_{(k)}f_kds.
\end{equation}
and \begin{equation}\label{cset}
\|A_{(k)}(s)\Box^q_{(k)}f_k\|_{(k)}
\leq\|\Box^q_{(k)}f_k\|_{(k)},
\quad 0\leq s\leq t.
\end{equation}
Here, \eqref{cset} and the last equality of \eqref{e-gue210504yydb} follow from $A_{k\phi}(s/k)\Box_{k\phi}^q=\Box_{k\phi}^qA_{k\phi}(s/k)$, \eqref{k(k)}, and \eqref{be-gue210325yyd}. Moreover,
from \eqref{mboxrhok}, \eqref{bboxrhok}, and \eqref{e-gue210504yydbc}, we have $\Box^q_{(k)}f_k\to\Box^q_{\phi_0}f$
in $\cali C^\infty$ on compact subsets as $k\to+\infty$. In particular, there exists a constant $C>0$, independent of $k$, such that
\begin{equation*}
\|\Box^q_{(k)}f_k\|_{(k)}\leq C.
\end{equation*}
Therefore, from \eqref{e-gue210504yydb} and \eqref{cset}, we get
\begin{equation*}
\|A_{(k)}(t)f_k-f_k\|_{(k)}
\leq \int^t_0\|A_{(k)}(s)\Box^q_{(k)}f_k\|_{(k)}ds
\leq tC.
\end{equation*}
Let $\hat R>0$ be arbitrary. Passing to the limit $k_j\to+\infty$ on $\ol B_{\hat R}$, we get
\begin{equation*}
\|Q(t)f-f\|_{\phi_0,\ol B_{\hat R}}
\leq tC.
\end{equation*}
Letting $\hat R\to+\infty$, we have $\|Q(t)f-f\|_{\phi_0}\leq tC$, and hence
\begin{equation*}
\lim_{t\to0+}Q(t)f=f.
\end{equation*}
We now complete the proof of the claim \eqref{e-gue210503ycd}.
Applying the uniqueness of the heat kernel for non-negative self-adjoint operators, see for example \cite[Proposition 2.17]{BGV}, we get that
\[
Q(t)=e^{-t\Box^q_{\phi_0}}.
\] 

Based on the preceding discussion, it is clear that for any subsequence of $A_{(k)}(t)$, we can always find a further sub-subsequence that converges locally uniformly to the same $e^{-t\Box^q_{\phi_0}}$. As a result, we conclude that $\lim_{k\to\infty}A_{(k)}(t,\mathbf z, \mathbf w)=e^{-t\Box^q_{\phi_0}}(\mathbf z, \mathbf w)$ exhibits local uniform convergence on $\mathbb R_+\times \ol M_0\times \ol M_0$ in the $\cali{C}^\infty$ topology. Hence, we have successfully established the desired result.

\end{proof}

%\noindent Recall that for $(\mathbf z,\mathbf w)\in \ol M_0\times \ol M_0$, we have
\subsection{Heat kernel asymptotics on the boundary region}\label{ddd}
In this section, we study the asymptotic behavior of $e^{-\frac{t}{k}\Box^q_{k}}({\bf z},{\bf z})$  as $k \to +\infty$ for ${\bf z}\in \ol M$ near the boundary, and prove Theorem \ref{main1}. Our analysis assumes that condition $Z(q)$ holds, which guarantees subelliptic estimates for $\Box^q_{k}$.

Before proceeding to do so, we shall digress for the moment to i
llustrate the relationship between $\dot R^\eta$ (see \eqref{Reta}), 
the curvature of the line bundle $L$, and the Levi form on the boundary $X$. 

%\begin{defin}\label{d-gue210524yyd}
%Let $L$ be a holomorphic line bundle over $M'$ with a Hermitian metric $h^L$. 
%In a local trivialization, let $\phi$ be the weight of $h^L$. 
%The Chern curvature of $(L, h^L)$ is the Hermitian form 
%$\mathcal R_{\bf z}^{L}$ on $T_{\bf z}^{1,0}M$ given by  
%\begin{equation}\label{curvature1}
%\mathcal{R}_{\bf z}^{L}(U,\ol V)=\frac{1}{2}
%\left\langle d(\dbar\phi-\pa\phi)({\bf z}), U\wedge\ol V\right\rangle,
%\quad U,V\in T^{1,0}_{\bf z} M,
%\end{equation}
%This definition is independent of the choice of $\phi$.  
%If ${\bf z} = [x,0] \in X$ with $x$ a local coordinate on $\partial D$, 
%we define CR curvature of $(L, h^L)$ by the Hermitian form 
%$\mathcal{R}_{x}^{\phi}$ on $T_{x}^{1,0}X$:
%\begin{equation}\label{curvature2}
%\mathcal R_{x}^{\phi}(U,\ol V)=\frac{1}{2}\left\langle d(\dbar_b\phi-\pa_b\phi)(x), 
%U\wedge\ol V\right\rangle,\quad U,V\in T^{1,0}_xX,
%\end{equation}
%where $\dbar_b$ is the tangential Cauchy–Riemann operator.  
%\end{defin}

\begin{defin}\label{d-gue210524yyd}
Let $L$ be a holomorphic line bundle over $M'$ endowed with a Hermitian metric $h^L$. 
In a local trivialization, let $\phi$ be the weight of $h^L$. 
The Chern curvature of $(L,h^L)$ is the Hermitian form 
$\mathcal R_{\bf z}^L$ on $T_{\bf z}^{1,0}M$ given by
\begin{equation}\label{curvature1}
\mathcal R_{\bf z}^{L}(U,\ol V)=\frac{1}{2}
\left\langle d(\dbar\phi-\pa\phi)({\bf z}),U\wedge \ol V\right\rangle,
\quad U,V\in T_{\bf z}^{1,0}M.
\end{equation}
This definition is independent of the choice of the local weight $\phi$.

If ${\bf z}=[x,0]\in X$, we denote by
\(\mathcal R^L_{b,x}\)
the Hermitian form on $T_x^{1,0}X$ induced by the curvature of $(L,h^L)$ 
on the CR tangent space $T_x^{1,0}X$. Equivalently, if $\phi$ is a local weight of $h^L$ near $x$, then
\begin{equation}\label{curvature2}
\mathcal R^L_{b,x}(U,\ol V)
=\frac{1}{2}\left\langle d(\dbar_b\phi-\pa_b\phi)(x),
U\wedge \ol V\right\rangle,\quad U,V\in T_x^{1,0}X,
\end{equation}
where $\pa_b$ and  $\dbar_b$ are tangential Cauchy-Riemann operators. $\mathcal R^L_{b,x}$ is intrinsic and independent of the choice of the local trivialization. 
%When a local weight $\phi$ is fixed, we may also write\(\mathcal R^\phi_x:=\mathcal R^L_{b,x}\)\,.
\end{defin}

For ${\bf z} \in M$, define $\dot{\mathcal{R}}_{\bf z}^L \in 
\text{End}(T^{1,0}_{\bf z}M)$ and $\varTheta_{\bf z}$ by 
\begin{equation}\label{icur}
\langle\,\dot{\mathcal{R}}^L_{\bf z}U\mid V\,
\rangle=\mathcal{R}^L_{\bf z}(U,\ol V),
\quad \varTheta_{\bf z} = \sum_{i,j=1}^{n} 
\mathcal{R}_{\bf z}^L(Z_i, \overline{Z}_j) \overline{w}^i \wedge 
(\overline{w}^j \wedge)^*,
\end{equation}
where $U, V \in T^{1,0}_{\bf z}M$, $\{Z_j\}_{j=1}^{n}$ 
is a local orthonormal frame for $T^{1,0}_{\bf z}M$, 
$\{w^j\}_{j=1}^{n}$ is the dual frame, and 
$(\overline{w}^j \wedge)^* = i_{\overline{Z}_j}$ denotes contraction. 
Similarly, for $x\in X$, we let $\dot{\mathcal{R}}^L_{b,x}$, 
$\dot{\mathcal{L}}_x\in \text{End}(T^{1,0}_xX)$ be linear maps given by 
\begin{equation}\label{ecur}
\langle\,\dot{\mathcal{R}}^L_{b,x}U\mid V\,\rangle=\mathcal{R}^L_{b,x}(U,\ol V),
\quad \langle\,\dot{\mathcal{L}}_xU\mid V\,\rangle=\mathcal{L}_{x}(U,\ol V),  \end{equation}
where $U, V \in T^{1,0}_xX$,  $\mathcal{L}_x$ is the Levi form of $X$ at $x$. 
We now translate the model quantities $\dot R^\eta$ and $\varTheta^\eta$
defined in \eqref{Reta} and \eqref{omega} into intrinsic geometric terms.
Under the identification induced by the scaling coordinates, the model
endomorphism $\dot R^\eta$ corresponds to
$\dot{\mathcal{R}}^L_{b,x}-2\eta\dot{\mathcal{L}}_x$, and
$\varTheta^\eta$ corresponds to $\varTheta^\eta_x$
defined below. For every $\eta \in \mathbb{R}$, let  
$$
\det(\dot{\mathcal{R}}^L_{b,x} - 2\eta \dot{\mathcal{L}}_x) = 
\mu_1(x) \cdots \mu_{n-1}(x),
$$  
where $\mu_j(x)$, $j=1,\ldots,{n-1}$, are the eigenvalues of 
$\dot{\mathcal{R}}^L_{b,x} - 2\eta \dot{\mathcal{L}}_x$
with respect to $\langle\,\cdot\mid\cdot\,\rangle$. Let  
\begin{equation}\label{bTheta}
\varTheta_x^\eta := \sum_{i,j=1}^{n-1} (\mathcal{R}^L_{b,x} - 
2\eta \mathcal{L}_x)(Z_i, \overline{Z}_j) 
\overline{\omega}^i \wedge (\overline{\omega}^j \wedge)^*: 
T^{*0,q}_xX \to T^{*0,q}_xX,    
\end{equation} 
where $\{Z_j\}_{j=1}^{n-1}$ is an orthonormal frame for $T^{1,0}_xX$, 
and $\{\omega^j\}_{j=1}^{n-1}$ is the dual frame in $T^{*1,0}_xX$.

Note that from \eqref{be-gue210325yyd}, we have
\begin{equation}\label{(k)kcon}
\begin{aligned}
 & A_{(k)}(t,\mathbf z, \mathbf w)=\sumprime_{|I|=|J|=q} 
 k^{-(n+1)}A_{k\phi,I,J}(\frac{t}{k},F_k\mathbf z, F_k\mathbf w)\,
 \ol\omega^I_{(k)}(\mathbf z)\otimes\ol\omega^J_{(k)}(\mathbf w)\\
&=\Big(\sumprime_{n\notin I}+\frac{1}{\sqrt{k}}
\sumprime_{n\in I}\Big)\Big(\sumprime_{n\notin J}+\frac{1}{\sqrt k}
\sumprime_{n\in J}\Big)\,k^{-(n+1)}A_{k\phi,I,J}(\frac{t}{k},F_k\mathbf z, 
F_k\mathbf w)\,\ol\omega^I(F_k \mathbf z)\otimes\ol\omega^J(F_k \mathbf w)
\end{aligned}   
\end{equation}
on $\mathbb R_+\times \ol{B}_{R}\times \ol B_{R}$. 
Let $p\in \ol M\cap\{r\ge -R/k\}$ be near to $X$. 
Choose local coordinates $[x,r]$ such that $p=[0,r_p]$ and 
\eqref{local1}-\eqref{bphi} hold on $\ol M\cap\{r\ge -R/k\}$. 
We can write $r_p=\dfrac{r}{k}$ with $r\ge -R$. 
It then follows from \eqref{limits} and \eqref{(k)kcon} that
\begin{equation*}
A_{\phi_0,I,J}(t,(0,r), (0,r))=\displaystyle\lim_{k\to\infty}\begin{dcases}
k^{-(n+1)}A_{k\phi,I,J}(t,p,p)\quad&\text{when $n\notin I\cap J$},\\
k^{-({n-1}+5/2)}A_{k\phi,I,J}(t,p,p)&\text{when $n\in (I\cup J)/(I\cap J)$},\\
k^{-({n-1}+3)}A_{k\phi,I,J}(t,p,p)&\text{when $n\in I\cap J$}.
\end{dcases}    
\end{equation*}
Applying this procedure for each point $\mathbf z\in \ol M\cap U$ by replacing $[0,\dfrac{r}{k}]$ with $[x,\dfrac{r}{k}]$, we obtain
\begin{equation}\label{htbdn}
 \begin{aligned}
  &A_{\phi_0,I,J}(t,(x,r), (x,r))\\
&=\displaystyle\lim_{k\to\infty}\begin{dcases}
k^{-(n+1)}A_{k\phi,I,J}\big(t,[x,\frac{r}{k}],[x,\frac{r}{k}]\big)\quad&\text{when $n\notin I\cap J$},\\
k^{-({n-1}+5/2)}A_{k\phi,I,J}\big(t,[x,\frac{r}{k}],[x,\frac{r}{k}]\big)&\text{when $n\in (I\cup J)/(I\cap J)$},\\
k^{-({n-1}+3)}A_{k\phi,I,J}\big(t,[x,\frac{r}{k}],[x,\frac{r}{k}]\big)&\text{when $n\in I\cap J$}.
\end{dcases}    
 \end{aligned}   
\end{equation}

%\begin{rem}\label{rem7}
 %   Let $z = [x, r] \in \ol M \cap U$. By choosing local coordinates such that $x = 0$, and using the mapping in \eqref{scalinging}, we observe that for sufficiently large $k$, the radius $R_k$ for which $F_k(B_{R_k}) \subset M \cap U$ can be taken not only as $B_{\log k}$, but can also be relaxed to $R_k = k^{1-\eps}$ for any $\eps>0$. Hence, the $R$ in \eqref{htbdn} can be extended to $k^{1-\eps}$ for $\eps>0$.
%end{rem}

\begin{proof}[Proof of Theorem \ref{main1}] 

The result follows directly from Theorem \ref{hmd} in combination with \eqref{kphik}, \eqref{psca}, and \eqref{htbdn}. The only point that requires clarification is how to drop the limit involving $\delta$ in front of the integral in \eqref{modelasymoptptic}. Since 
$$\lim_{\eta \to -\infty} \int^{\frac{2r-t\eta}{\sqrt{2t}}}_{-\infty} e^{-\tau^2} \, d\tau = \sqrt{\pi}$$ for fixed $t>0$ and $r$, it follows from Remark \ref{rem:zq} that, when condition $Z(q)$ is satisfied, it suffices to consider the uniform boundedness of the integrand for $\eta \gg 0$. 

Note that both the first and the third terms on the right-hand side 
of \eqref{modelasymoptptic} contain the factor $e^{-t\eta^2}$, 
which decays rapidly as $\eta\to+\infty$. 
Hence, these terms cause no difficulty. 
For the second term, observe that 
$$
\int_{-\infty}^{\frac{2r-t\eta}{\sqrt{2t}}} e^{-\tau^2}\, d\tau
\sim
-\frac{\sqrt{t/2}}{2r-t\eta}\,
e^{-\frac{(2r-t\eta)^2}{2t}}
\qquad \text{as }\eta\to +\infty.
$$
Therefore, the second term also does not pose any problems, as it decays to zero when multiplied by the other $\eta$-dependent terms for $\eta \gg 0$.
\end{proof}

\subsection{Heat kernel asymptotics in the inner region}\label{innestimate}
We now consider the inner region $\{r<-\frac{1}{\sqrt{k}}\}$. Our asymptotic estimate \eqref{htin} below is essentially the same as in the case where $M$ is a compact manifold without boundary (see \cite[Theorem~1.5]{Bi87}, \cite[Theorem~1.6.1]{MM07}). 

Let $I\subset\mathbb \R_+$  be a compact interval and let $K\subset M\cap\{r<-\frac{1}{\sqrt{k}}\}$ be a compact set. Then, there is a constant $C>0$ independent of $k$ such that, for all $\mathbf z\in K$, $t\in I$, $0\leq q\leq n$, we have
\begin{equation}\label{eq_ub}
\abs{k^{-n}e^{-\frac{t}{k} \Box^q_{k\phi}}(\mathbf z,\mathbf z)}_{\mathscr L(T^{*0,q}_{\mathbf z}M',T^{*0,q}_{\mathbf z}M')}\leq C.
\end{equation} 
To analyze the asymptotic behavior, let $R\in \mathbb{R}$. Define the ball $D_R := \{{\bf z} \in \mathbb{C}^{n} : |{\bf z}| < R\}$, which can be identified as a subset of $M$ via the scaling map:
\begin{equation}\label{eq_scaling}
\begin{split}
F_k: \mathbb C^{n}\to\mathbb C^{n},\quad
F_k({\bf z})=\frac{{\bf z}}{\sqrt{k}}.
\end{split}
\end{equation} 
For any point $\mathbf{z} \in M$ that can be expressed as $\mathbf{z} = \hat{\mathbf{z}}/\sqrt{k}$ with $\hat{\mathbf{z}} \in D_{R}$, we have 
\begin{equation}\label{htin}
\lim_{k\to\infty}k^{-n}
e^{-\frac{t}{k}\square_{k\phi}^q}(\mathbf z,\mathbf z)=e^{-t\mathring{\Box}^q_{\phi_0}}(\mathbf z,\mathbf z) 
=\frac{1}{(2\pi)^n}\frac{\det(\dot{\mathcal{R}}_{\mathbf z}^L)\exp(-t\varTheta_{\mathbf z})}{\det(1-\exp(-t\dot{\mathcal{R}}_{\mathbf z}^L)),}
\end{equation}   		
Here, $\mathring{\Box}^q_{\phi_0}$ is the model Laplacian introduced in \eqref{modelLa}, while $\dot{\mathcal{R}}_{\mathbf z}^L$ and $\varTheta_{\mathbf z}$ are defined in \eqref{icur}.  %Here if an eigenval

\section{Morse inequalities and semi-classical Weyl law}\label{morseinequality}

In this section, we apply the heat kernel asymptotics established in Section~\ref{heatasymptotic} to derive holomorphic Morse inequalities for complex manifolds with boundary. The final formulas contain the usual interior contribution coming from the curvature of $L$, together with an additional boundary term expressed in terms of the Levi form and the boundary weight. This yields a heat-kernel proof of the weak and strong Morse inequalities in the present setting. Finally, we explain how the same heat trace asymptotics lead to a semi-classical Weyl law for the $\dbar$-Neumann Laplacian.

\begin{thm}\label{thmmorse}
Let $M$ be a relatively compact open subset with a smooth boundary $X$ of $n$-dimensional complex manifold $M'$. 
Let $(L,h^{L})$ be a holomorphic Hermitian line bundle over $M'$. Let $0\le q\le n$. Suppose $Z(q)$ holds. Then, as $k\to\infty$, we have 
\begin{equation}\label{weak}
\begin{aligned}
\dim H^{q}(\ol M,L^k) &\le\frac{k^{n}}{(2\pi)^{n}}\Big(\int_{M(q)}\big|\det(\dot{\mathcal{R}}^L_{\mathbf z})\big|dv_{M'}(\mathbf z)+\int_X\int_{\R_{x}(q)}\big|\det(\dot{\mathcal{R}}^L_{b,x}-2\eta\dot{\mathcal L}_x)\big|\,d\eta\,dv_X(x)\Big)\\
&\quad+o(k^{n}).
\end{aligned}
\end{equation}
%where
%\begin{equation}\label{rjz}
%M(q)=\{\mathbf z\in M \mid \dot{\mathcal R}^L_{\mathbf z}\,\text{is non-degenerate and has exactly $q$ negative eigenvalues}\},
%\end{equation}
%and
%\begin{equation}\label{rj}
%\begin{aligned}
%\R_x(q)=\{\eta<0 \mid &\dot{\mathcal R}^L_{b,x}-2\eta\dot{\mathcal L}_x\,\text{has exactly $q$ negative eigenvalues} \\
%&\text{and $n-q-1$ positive eigenvalues}\}.
%\end{aligned}
%\end{equation}
Suppose that condition $Z(j)$ holds for all $j=0,1,\cdots, q$. Then, as $k\to\infty$, we have 
\begin{equation}\label{strong}
\begin{split}
\sum_{j=0}^q(-1)^{q-j}\dim H^{j}(\ol M,L^k)
&\le\frac{k^{n}}{(2\pi)^{n}}\sum_{j=0}^q(-1)^{q-j}\Big(\int_{M(j)}\big|\det(\dot{\mathcal{R}}^L_{\mathbf z})\big|dv_{M'}(\mathbf z)\\
&\quad+\int_X\int_{\R_{x}(j)}\big|\det(\dot{\mathcal{R}}^L_{b,x}-2\eta\dot{\mathcal L}_x)\big|\,d\eta\,dv_X(x)\Big)+o(k^{n}).
\end{split}
\end{equation}
\end{thm}
Recall that the index sets $M(j)$ and $\R_{x}(j)$ were defined in 
\eqref{rjz} and \eqref{rj}.
\begin{rem}
The inequalities above were previously obtained by Berman \cite{Be05} using Bergman kernel methods. The interior term is the usual Demailly contribution, written here in determinant notation. The boundary term is also the same as Berman's: his set $T(q)_{\rho,x}$ \cite[p.~1056]{Be05} corresponds to our $\mathbb R_x(q)$ after the linear change of parameter determined by $\dot{\mathcal R}^L_{b,x}-2\eta\dot{\mathcal L}_{x}$. Thus \eqref{weak} and \eqref{strong} give a heat-kernel proof of Berman's holomorphic Morse inequalities in our notation.
\end{rem}
\subsection{Proof of Theorem \ref{thmmorse}}
We now prove Theorem~\ref{thmmorse}. The argument follows 
the standard heat-kernel strategy: we compare the dimensions 
of the cohomology groups with traces of the heat operators, 
and then insert the interior and boundary asymptotic formulas obtained earlier.

Fix $q\in\set{0,1,\ldots,{n}}$ and assume that $Z(q)$ holds on $M$. Let
\begin{equation}\label{e-gue210529yydI}
\Tr_q\left(e^{-\frac{t}{k}\Box_{k}^q}(\mathbf z,\mathbf z)\right):=
\sum^d_{j=1}\big\langle\,e^{-\frac{t}{k}\Box_{k}^q}(\mathbf z,\mathbf z)v_j(\mathbf z) 
\mid v_j(\mathbf z)\,\big\rangle_k,
\end{equation}
where $\set{v_j}^d_{j=1}$ is an orthonormal basis for $T^{*0,q}_{\mathbf z}M$. Let 
\begin{equation}\label{trace}
\Tr_q\left(e^{-\frac{t}{k}\Box_{k}^q}\right):=\int_{\ol M} 
\Tr_q\left(e^{-\frac{t}{k}\Box_{k}^q}(\mathbf z,\mathbf z)\right)dv_{M'}(\mathbf z).
\end{equation}
It is well-known (cf. \cite[Lemma 1.7.2]{MM07}) that 
\begin{equation}\label{e-gue210614yyd}
\dim H^{q}(M,L^k) \le\Tr_q\left(e^{-\frac{t}{k}\Box_{k}^q}\right),\ \ 
\mbox{for every $t>0$};
\end{equation}
and
\begin{equation}\label{morse}
\sum_{j=0}^q(-1)^{q-j}\dim H^{j}(\ol M,L^k) 
\le\sum_{j=0}^q(-1)^{q-j}\Tr_j\left(e^{-\frac{t}{k}\Box_{k}^j}\right),\ \ 
\mbox{for every $t>0$}. 
\end{equation}
Note that
\begin{equation}\label{htbd}
\begin{aligned}
&\Tr_q\big(e^{-\frac{t}{k}\square_{k}^{q}}(\mathbf z, \mathbf z)\big)=
\sumprime_{n\notin I}A_{k,I,I}(\frac{t}{k},\mathbf z, \mathbf z)+
\chi^{R,\eps}_k(r)\sumprime_{n\in I}A_{k,I,I}(\frac{t}{k},\mathbf z, \mathbf z),
\end{aligned}
\end{equation}
since $\langle\,\ol\omega^{n}\mid\ol\omega^{n}\,\rangle_k=
\chi^{R,\eps}_k(r)$ with ${\bf z}=[\xi,\theta,r]$.
From \eqref{morse} and \eqref{htbd},  
for every $t>0$, we have 

\begin{equation}\label{weakmorse}
\begin{split}
&\limsup_{k\to\infty}k^{-n}\dim H^{q}(\ol M,L^k)
\leq\limsup_{k\to\infty}k^{-n}\Tr_q(e^{-\frac{t}{k}\Box^q_k})\\%1
&=\limsup_{k\to\infty}\int_{\ol M} k^{-n}\Tr_q(e^{-\frac{t}{k}
\Box^q_k}(\mathbf z,\mathbf z))dv_{M'}(\mathbf z)\\%2
&=\limsup_{k\to\infty}k^{-n}\Big(\int_{\ol M\cap\{r\le-\frac{1}{\sqrt{k}}\}}+
\int_{\ol M\cap\{-\frac{1}{\sqrt{k}}\le r\le \frac{\varepsilon-1}{\sqrt{k}}\}}+
\int_{\ol M\cap\{ \frac{\varepsilon-1}{\sqrt{k}} \le r\le-\frac{R+\varepsilon}{k}\}}\\
&\quad
+\int_{\ol M\cap\{-\frac{R+\varepsilon}{k}\leq r\le-\frac{R}{k}\}}+
\int_{\ol M\cap\{r\ge-\frac{R}{k}\}}\Big)
\Tr_q(e^{-\frac{t}{k}\Box^q_k}(\mathbf z,\mathbf z))dv_{M'}(\mathbf z),
\end{split}
\end{equation}
where $R\gg1$, $0<\varepsilon\ll1$ are as in \eqref{e-gue260125yyd}. 
Note that in \eqref{weakmorse}, according to the construction of the 
Hermitian metric in \eqref{newmetric}, we decompose the domain of 
integration into five regions in the normal direction.
The reason for this decomposition is that, in Sections~\ref{heatasymptotic}, 
we established the heat kernel asymptotics only in the boundary region 
$\{r\ge -\frac{R}{k}\}$ and the interior region $\{r\le -\frac{1}{\sqrt{k}}\}$, 
whereas the integral is taken over the whole set $\{r\le 0\}$. 
For the three transition regions, we shall show that their contributions tend 
to zero as $R\to\infty$ and $\varepsilon\to0$.
To do this, we use the scaling map
\begin{equation}
\begin{split}
F_k: \mathbb C^{n}&\longrightarrow\mathbb C^{n},\\
(z,z_n)&\longmapsto \Big(\dfrac{z}{\sqrt k},\dfrac{z_n}{\sqrt{k/\chi_k(r)}}\Big),
\end{split}   
\end{equation}
which coincides with \eqref{scalinging} when $r\ge- \frac{R}{k}$ and with \eqref{eq_scaling} when $r\le- \frac{1}{\sqrt{k}}$. We now repeat the scaling argument used in Section~\ref{scaledL}. It follows that there exists a constant $C>0$, independent of $R$, $\varepsilon$, and $k$, such that
\begin{equation}\label{e-gue260205yydb}
\Tr_q\big(e^{-\frac{t}{k}\Box^q_k}(\mathbf z,\mathbf z)\big)\leq C k^n/\chi_k^{R,\varepsilon}(r).
\end{equation}
Applying Fatou's lemma, we obtain
\begin{equation}\label{e-gue260205yydI}
\begin{aligned}
&\limsup_{k\to\infty}k^{-n}\int_{\ol M\cap\{-\frac{R+\varepsilon}{k}\leq r\le-\frac{R}{k}\}}\Tr_q(e^{-\frac{t}{k}\Box^q_k}(\mathbf z,\mathbf z))dv_{M'}(\mathbf z)\leq C\varepsilon,\\
&\limsup_{k\to\infty}k^{-n}\int_{\ol M\cap\{\frac{\varepsilon-1}{\sqrt{k}}\leq r\le-\frac{R+\varepsilon}{k}\}}\Tr_q(e^{-\frac{t}{k}\Box^q_k}(\mathbf z,\mathbf z))dv_{M'}(\mathbf z)\leq\frac{\hat C}{R},\\
&\limsup_{k\to\infty}k^{-n}\int_{\ol M\cap\{-\frac{1}{\sqrt{k}}\leq r\le\frac{\varepsilon-1}{\sqrt{k}}\}}\Tr_q(e^{-\frac{t}{k}\Box^q_k}(\mathbf z,\mathbf z))dv_{M'}(\mathbf z)\leq \lim_{k\to\infty} C\eps/\sqrt k=0,
\end{aligned}
\end{equation}
where $\hat C>0$ is a constant independent of $R$. 

Let $R\to+\infty$ and $\varepsilon\to0$ in the right side of \eqref{weakmorse}, and from \eqref{localasymoptptic}, \eqref{htin}, \eqref{e-gue260205yydI}, it is straightforward to see that 
\begin{equation}\label{weakmorsez}
\begin{split}
&\limsup_{k\to\infty}k^{-n}\dim H^{q}(\ol M,L^k)\\
&\le \frac{1}{(2\pi)^{n}}\int_{M}\frac{\det\dot{\mathcal{R}}_{\mathbf z}^L}{\det(1-e^{-t\dot{\mathcal{R}}_{\mathbf z}^L})}\Tr_qe^{-t\varTheta_{\mathbf z}}dv_{M'}(\mathbf z)
\\
&\quad+\frac{1}{(2\pi)^{n}}\int_X\int^0_{-\infty}\frac{1+e^{-2{r}^2/t}}{\sqrt{2\pi t}}\Big(\int_{\R} \dfrac{e^{-t\eta^2/2}\cdot\det(\dot{\mathcal{R}}^L_{b,x}-2\eta\dot{\mathcal L}_x)}{\det\big(1-e^{-t (\dot{\mathcal{R}}^L_{b,x}-2\eta\dot{\mathcal L}_x)}\big)}
\Tr_qe^{-t\varTheta^{\eta,\tau}_x}d\eta\Big) dr dv_X(x)\\%2
&\quad +\frac{2}{(2\pi)^{n} \sqrt{\pi}}\int_X\int^0_{-\infty}\int_{\R} \dfrac{\eta e^{-2\eta r} \cdot\det(\dot{\mathcal{R}}^L_{b,x}-2\eta\dot{\mathcal L}_x)}{\det\big(1-e^{-t (\dot{\mathcal{R}}^L_{b,x}-2\eta\dot{\mathcal L}_x)}\big)}\Tr_qe^{-t\varTheta^{\eta,\tau}_x}\Big(\int^{\frac{2r-t\eta}{\sqrt{2t}}}_{-\infty}e^{-\gamma^2}d\gamma\Big) d\eta drdv_X(x)\\%3
&\quad +\frac{1}{(2\pi)^{n}}\int_X\int^0_{-\infty}\frac{1-e^{-2{r}^2/t}}{\sqrt{2\pi t}}\Big(\int_{\R} \dfrac{e^{-t\eta^2/2}\cdot\det(\dot{\mathcal{R}}^L_{b,x}-2\eta\dot{\mathcal L}_x)}{\det\big(1-e^{-t (\dot{\mathcal{R}}^L_{b,x}-2\eta\dot{\mathcal L}_x)}\big)}
\Tr_qe^{-t\varTheta^{\eta,\nu}_x}d\eta\Big) drdv_X(x).
\end{split}
\end{equation}

For every $j=0,1,\ldots,q$, set
\begin{equation}\label{rqz}
\begin{aligned}
M(j)=\{\mathbf z\in M &\mid \dot{\mathcal R}^L_{\mathbf z}\,\text{has exactly $j$ negative eigenvalues} \\&\text{and $n-j$ positive eigenvalues}\}.
\end{aligned}
\end{equation}
and
\begin{equation}\label{rq}
\begin{aligned}
\R_x(j)=\{\eta<0 &\mid \dot{\mathcal R}^L_{b,x}-2\eta\dot{\mathcal L}_x\,\text{has exactly $j$ negative eigenvalues} 
	\\&\text{and $n-j-1$ positive eigenvalues}\}.
\end{aligned}
\end{equation}
Note that since $Z(q)$ holds at each point of $X$, $\R_{x}(q)$ is bounded for all local weights of $L$. In fact, for $\eta<0$,
$$
\dot{\mathcal R}^L_{b,x}-2\eta\dot{\mathcal L}_x
=(-2\eta)\left(\dot{\mathcal L}_x-\frac{1}{2\eta}\dot{\mathcal R}^L_{b,x}\right),
$$
so as $\eta\to-\infty$ its index is the same as that of $\dot{\mathcal L}_x$.
The condition $Z(q)$ excludes the possibility that $\dot{\mathcal L}_x$ 
has exactly $q$ negative and $n-q-1$ positive eigenvalues. Hence $\mathbb R_x(q)$ is bounded.
It is straightforward to check that 
\begin{equation}\label{I1}
\lim_{t\to+\infty}\frac{\det\dot{\mathcal{R}}_{\mathbf z}^L}{\det(1-e^{-t\dot{\mathcal{R}}_{\mathbf z}^L})}\Tr_qe^{-t\varTheta_{\mathbf z}}
=(-1)^q\chi_{M(q)}(\mathbf z)\det(\dot{\mathcal{R}}^L_{\mathbf z})=\chi_{M(q)}(\mathbf z)\big|\det(\dot{\mathcal{R}}^L_{\mathbf z})\big|,
\end{equation}
\begin{equation}\label{III1}
\begin{split}
&\lim_{t\to+\infty}\frac{(1+e^{-2{r}^2/t})e^{-t\eta^2/2}}{\sqrt{2\pi t}}
\dfrac{\det(\dot{\mathcal{R}}^L_{b,x}-2\eta\dot{\mathcal L}_x)}{\det\big(1-e^{-t (\dot{\mathcal{R}}^L_{b,x}-2\eta\dot{\mathcal L}_x)}\big)}\Tr_qe^{-t\varTheta^{\eta,\tau}_x}=0,
\end{split}
\end{equation}
\begin{equation}\label{III2}
\begin{split}
&\lim_{t\to+\infty}\frac{(1-e^{-2{r}^2/t})e^{-t\eta^2/2}}{\sqrt{2\pi t}}
\dfrac{\det(\dot{\mathcal{R}}^L_{b,x}-2\eta\dot{\mathcal L}_x)}{\det\big(1-e^{-t (\dot{\mathcal{R}}^L_{b,x}-2\eta\dot{\mathcal L}_x)}\big)}\Tr_qe^{-t\varTheta^{\eta,\nu}_x}=0,
\end{split}
\end{equation}
and
\begin{equation}\label{III3}
\begin{split}
&\lim_{t\to+\infty}\dfrac{\det(\dot{\mathcal{R}}^L_{b,x}-2\eta\dot{\mathcal L}_x)\Tr_qe^{-t\varTheta^{\eta,\tau}_x}}{\det\big(1-e^{-t (\dot{\mathcal{R}}^L_{b,x}-2\eta\dot{\mathcal L}_x)}\big)}\cdot\eta e^{-2\eta r} \int^{\frac{2r-t\eta}{\sqrt{2t}}}_{-\infty}e^{-\gamma^2}d\gamma \\
&= (-1)^q\chi_{\R_x(q)}(\eta)\det(\dot{\mathcal{R}}^L_{b,x}-2\eta\dot{\mathcal L}_x)\cdot\begin{cases}
 \sqrt\pi \eta e^{-2\eta r}\quad &\eta< 0,\\
 0\quad &\eta\ge 0,\\
\end{cases}
\end{split}
\end{equation}
where $\chi_{M(q)}(\mathbf z)$ and $\chi_{\R_x(q)}(\eta)$ are the characteristic functions of $M(q)$ and $\R_x(q)$, respectively. Let $t\to+\infty$ in \eqref{weakmorsez} and using \eqref{weakmorse}-\eqref{III3}, we obtain the following weak Morse inequality \eqref{weak}: 
\begin{equation}\label{e-gue210530yydI}
\begin{split}
&\limsup_{k\to\infty}k^{-n}\dim H^{q}(\ol M,L^k)\\
&\le\frac{1}{(2\pi)^{n}}\Big(\int_{M(q)}\big|\det(\dot{\mathcal{R}}^L_{\mathbf z})\big|dv_{M'}(\mathbf z)+\int_X\int_{\R_{x}(q)}\big|\det(\dot{\mathcal{R}}^L_{b,x}-2\eta\dot{\mathcal L}_x)\big|\,d\eta\,dv_X(x)\Big),
\end{split}
\end{equation}
where the second term in the final expression is obtained by
$$
\int^{0}_{-\infty}e^{-2\eta r}dr=\dfrac{1}{2\eta} \quad\text{for $\eta <0$}.
$$

For the proof of the strong Morse inequalities, starting from \eqref{morse} and repeating the proof of \eqref{e-gue210530yydI} with minor changes, we get \eqref{strong}, and hence Theorem \ref{thmmorse} follows.

\subsection{Morse inequalities for \texorpdfstring{$q$}{q}-convex 
and \texorpdfstring{$q$}{q}-concave manifolds}\label{5.2}
In this section, we establish Morse inequalities for $q$-convex and 
$q$-concave manifolds as a direct consequence of Theorem \ref{thmmorse} and its proof. 
The concept of $q$-convexity and $q$-concavity, originally introduced 
by Andreotti and Grauert in \cite{AG:62}, is among the 
fundamental tools used in analyzing the geometry of non-compact complex spaces. 
We begin by introducing their definitions and discussing their relation to condition $Z(q)$.

\begin{defin}[{\cite{AG:62}}]
Let $M$ be an $n$-dimensional connected complex manifold
and $q\in\{1,\ldots,n\}$.

(i) The manifold $M$ 
is called \emph{$q$-convex} if there exists a smooth function 
$\varphi:M\to[a,b)$, where $a\in\R$, $b\in\R\cup\{+\infty\}$, 
such that the \emph{sublevel set}
$M_c=\{x\in M: \varphi(x)<c\}$ is relatively compact in $M$ for all $c\in[a,b)$ and 
$\sqrt{-1}\partial\bar{\partial} \varphi$ has at least $n-q+1$ positive eigenvalues outside 
a compact set $K$ (called an exceptional set). 

(ii) The manifold $M$ is called \emph{$q$-concave}
if there exists a smooth function 
$\varphi:M\to(a,b]$, where $\,a\in\R\cup\{-\infty\}$ and $b\in\R$, 
such that the \emph{superlevel set}
$M_c=\{x\in M: \varphi(x)>c\}$ is relatively compact in $M$ for all $c\in(a,b]$ and 
$\sqrt{-1}\partial\bar{\partial}\varphi$ has at least $n-q+1$ positive eigenvalues 
outside a compact set $K\subset M$, called
an exceptional compact set. In the sequel, we also use $\rho:=-\varphi$ as 
the sublevel exhaustion function and then $M_c=\{x\in M: \rho(x) < c\} \Subset M$ for all $c$.
\end{defin}  

The connection to the condition $Z(q)$ is as follows.

(1) Let $M$ be an $n$-dimensional $q$-convex manifold ($n\geq1$ and $1\leq q\leq n$) 
with exhaustion function $\varphi$
and an exceptional set $K$. For any regular value $c>\sup_K\varphi$
of $\varphi$, the sublevel set $M_c=\{\varphi<c\}$ is a relatively compact open set
with smooth boundary satisfying condition $Z(j)$ for all $j\in\{q,\ldots,n\}$.

(2) Let $M$ be an $n$-dimensional $q$-concave manifold ($n\geq2$ and $1\leq q\leq n-1$)
with exhaustion function $\varphi$
and an exceptional set $K$. For any regular value $c<\inf_K\varphi$
of $\varphi$, the superlevel set $M_c=\{\varphi>c\}$ is a relatively compact open set
with smooth boundary satisfying condition $Z(j)$ for all $j\in\{0,\ldots,n-q-1\}$.

A fundamental result about $q$-convex and $q$-concave manifolds is the 
finiteness theorem of Andreotti-Grauert \cite[Th\'eor\`eme 14]{AG:62}. 
Namely, if $M$ is an $n$-dimensional $q$-convex manifold for $1\leq q\leq n$ 
(resp.\ $q$-concave for $n\geq2$ and $1\leq q\leq n-1$), 
then for every coherent analytic sheaf $\mathscr{F}$ on $M$, 
the cohomology groups $H^j(M,\mathscr{F})$ are finite dimensional
for $j\ge q$ (resp.\ $j\le n-q-1$).
Moreover, under the same assumptions, the following isomorphism theorem holds: 
for any $c \in \mathbb{R}$ such that $M_c$ contains the exceptional compact set, 
the restriction morphisms
$
H^j(M,\mathscr{F}) \to H^j(M_c,\mathscr{F}),
$
are isomorphisms for $j \ge q$ if $M$ is $q$-convex
(resp.\ $j \le n - q - 1$ if $M$ is $q$-concave).

On a compact complex manifold $M$ of dimension $n$, Demailly's holomorphic
Morse inequalities \cite{De85} give asymptotic bounds for the sheaf cohomology $H^j(M,L^k)$,
$j=0,1,\ldots,n$,
as $k\to\infty$ in terms of curvature integrals of $L$.
The natural question of extending this result for $q$-convex and $q$-concave
manifolds for the range of $j$ where the cohomology is finite dimensional
($j\ge q$ resp.\ $j\le n-q-1$) was carried out  under various hypotheses
in \cite{Bou:89,Mar96,Be05,MM07,LMW25}.
The results presented in \cite{Bou:89,Mar96, MM07} were derived under
certain positivity assumptions on the curvature of $L$. 
Consequently, the curvature bounds employed in the Morse inequalities 
were restricted to integration within the interior of $M$.
By contrast, in \cite{Be05,LMW25}, no positivity assumption was imposed, 
resulting in bounds that also incorporated a boundary curvature integral.
We will now derive holomorphic Morse inequalities for
$q$-convex and $q$-concave manifolds modeled on Theorem~\ref{thmmorse},
where we thus exhibit a boundary curvature integral.

To deal with sheaf cohomology, we recall the following facts.
Let $M$ be a complex manifold and let $E$ be a holomorphic
vector bundle on $M$. Let us denote by $L^2_{0,j}(M,E,\textrm{loc})$
the space of locally $L^2$ $(0,j)$-forms with values in $E$.
We denote by $H^{0,j}_{(2)}(M,E,\textrm{loc})$
the quotient space of $\{f\in L^2_{0,j}(M,E,\textrm{loc}):\dbar f=0\}$
by $L^2_{0,j}(M,E,\mathrm{loc})\cap 
\{\dbar g:g \in L^2_{0,j-1}(M,E,\mathrm{loc})\}$.
By the Dolbeault theorem, there is a natural isomorphism
between the space $H^{0,j}_{(2)}(M,E,\textrm{loc})$ and the $j$-th cohomology group $H^{j}(M,E)$ of $M$ with values in the
sheaf of holomorphic sections of $E$.
Assume now that $M$ is a relatively compact open set with a smooth boundary in a complex manifold of dimension $n$, and it satisfies the conditions
$Z(j)$ and $Z(j+1)$ for some $j\in\{0,\ldots,n-1\}$. By \cite[(3.1.14), (4.3.1)]{FK72} and 
\cite[Theorem 3.4.9]{Hor:65}, the restriction morphism 
$H^{j}(\ol M,E)\to H^{0,j}_{(2)}(M,E,\textrm{loc})$ is an isomorphism,
and thus the canonical morphism $H^{j}(\ol M,E)\to H^{j}(M,E)$
is an isomorphism. This also holds for $j=n$ by the same arguments,
since any $(0,n+1)$-form vanishes.

We first state the holomorphic Morse inequalities for $q$-convex manifolds.
To write only one interior integral, we introduce the
following notation: For $r\in\{0,\ldots,n\}$, we set $M(\geq r)=\bigcup_{j=r}^n M(j)$.
We also set $\R_x(n)=\varnothing$.
Let us denote by $c_1(L,h^L)=\frac{i}{2\pi}\mathcal R^L$ the Chern curvature form of $(L,h^L)$.
%===
\begin{thm}\label{miqconv}
Let $M$ be a $q$-convex manifold of dimension $n$.
Let $c$ be a regular value of the exhaustion function of $M$
such that $M_c$ contains the exceptional compact set. Set $X_c=\partial M_c$.
Let $(L,h^L)$ be a holomorphic Hermitian line bundle over $M$.
Then we have for all $r\ge q$, as $k\to\infty$,
\begin{equation}\label{strongq}
\begin{split}
\sum_{j=r}^{n}(-1)^{j-r}&\dim{}  H^{j}(M_c,L^k)
\le\frac{k^{n}}{n!}
\int_{M_c(\geq r)}(-1)^r c_1(L,h^L)^n\\
&\quad+\frac{k^{n}}{(2\pi)^{n}}\sum_{j=r}^{n}(-1)^{j-r}\int_{X_c}\int_{\R_{x}(j)}
\big|\det(\dot{\mathcal{R}}^L_{b,x}-
2\eta\dot{\mathcal L}_x)\big|\,d\eta\, dv_{X_c}(x)+o(k^{n}),
\end{split}
\end{equation}
%\begin{equation}\label{strongq}
%\begin{split}
%\sum_{j=r}^{n}(-1)^{j-r}\dim H^{j}(M_c,L^k)
%&\le\frac{k^{n}}{(2\pi)^{n}}\sum_{j=r}^{n}(-1)^{j-r}
%\Big(\int_{M_c(s)}\big|\det(\dot{\mathcal{R}}^L_{\mathbf z})\big|dv_{M'}(\mathbf z)\\
%&\quad+\int_{X_c}\int_{\R_{x}(s)}\big|\det(\dot{\mathcal{R}}^L_{b,x}-
%2\eta\dot{\mathcal L}_x)\big|\,d\eta\,dv_{X_c}(x)\Big)+o(k^{n}),
%\end{split}
%\end{equation}
Moreover, we have for all $j\ge q$, as $k\to\infty$,
\begin{equation}\label{weakq}
\begin{aligned}
\dim H^{j}(M_c,L^k) &\le\frac{k^{n}}{n!}\int_{M_c(j)}
(-1)^jc_1(L,h^L)^n+
\frac{k^{n}}{(2\pi)^{n}}\int_{X_c}\int_{\R_{x}(j)}\big|\det(\dot{\mathcal{R}}^L_{b,x}-
2\eta\dot{\mathcal L}_x)\big|\,d\eta\,dv_{X_c}(x)\\
&\quad+o(k^{n}).
\end{aligned}
\end{equation}
%\begin{equation}\label{weakq}
%\begin{aligned}
%\dim H^{j}(M_c,L^k) &\le\frac{k^{n}}{(2\pi)^{n}}\Big(\int_{M(j)}
%\big|\det(\dot{\mathcal{R}}^L_{\mathbf z})\big|dv_{M'}(\mathbf z)+
%\int_X\int_{\R_{x}(j)}\big|\det(\dot{\mathcal{R}}^L_{b,x}-
%2\eta\dot{\mathcal L}_x)\big|\,d\eta\,dv_X(x)\Big)\\
%&\quad+o(k^{n}).
%\end{aligned}
%\end{equation}
The same estimates hold if on the left-hand side of \eqref{strongq}
and \eqref{weakq} we replace 
$H^{j}(M_c,L^k)$ by $H^{j}(M,L^k)$, for $j\ge q$.
\end{thm}
%===
\begin{proof}
A sublevel set $M_c$ as in the statement satisfies the conditions
$Z(j)$ for all $j\in\{q,\ldots,n\}$; thus, for any holomorphic
vector bundle $E$ on $M$, we have canonical isomorphisms
$$H^{j}(\ol M_c,E)\xrightarrow{\sim} H^{j}(M_c,E), \quad j\in\{q,\ldots,n\}.$$ 
%for all $j\in\{q,\ldots,n\}$.
Thus, the inequalities \eqref{strongq} and \eqref{weakq} follow
immediately from Theorem \ref{thmmorse}. 
\end{proof}
%===

%===
\begin{cor}\label{miqconvsp}
Let $M$ be a $q$-convex manifold of dimension $n$.
Let $c$ be a regular value of the exhaustion function of $M$
such that $M_c$ contains the exceptional compact set. %Set $X_c=\partial M_c$.
Let $(L,h^L)$ be a holomorphic Hermitian line bundle over $M$ that is 
semipositive on $\overline{M}_c$.
Then we have for all 
$j\ge q$, as $k\to\infty$,
\begin{equation}\label{weakqsp1}
\begin{aligned}
\dim H^{j}(M,L^k)=\dim H^{j}(M_c,L^k)=o(k^{n}).
\end{aligned}
\end{equation}
\end{cor}
%===
\begin{proof}
We apply the inequalities \eqref{weakq} and observe that for 
all $j\geq q$ we have $M_c(j)=\varnothing$, and for all $j\geq q$ and $x\in X_c$
we have $\R_x(j)=\varnothing$,
since $c_1(L,h^L)$ is semipositive on $\overline{M}_c$ and
$\dot{\mathcal L}_x$ has at least $n-q$ positive eigenvalues.
\end{proof}
%===
Corollary \ref{miqconvsp} extends the following vanishing theorem 
of Andreotti-Tomassini \cite[Theorem 1]{AnTo:69} (see also \cite{LMW25}
for a differential-geometric proof): If $M$ is a $q$-convex manifold 
and $L\to M$ is a positive line bundle, then there exists an integer 
$k_0$ such that for all $k\geq k_0$ and $j\geq q$
we have $H^j(M,L^k)=0$.
%===

In the case of $1$-convex manifolds, there are more precise vanishing theorems,
due to the special structure of these manifolds.
For this purpose, we recall the following
analytic characterization of $1$-convex manifold $X$ 
(see e.\,g.\ \cite{AG:62}): There exists a Stein space $Y$, 
a proper holomorphic surjective map $\rho:X\to Y$ 		
satisfying $\rho_*\mathcal{O}_X = \mathcal{O}_Y$, 		
and a finite set $A\subset Y$ such that the induced map 		
$X\setminus\rho^{-1}(A) \to Y\setminus A$ is biholomorphic. 		
The Stein space $Y$ is called the Remmert reduction of $X$ and 		
$\Sigma\coloneqq\rho^{-1}(A)$ the exceptional analytic set of $X$.
We have the following vanishing result of Schneider \cite[Corollar 2.12]{Schn78}.
Let $M$ be a $1$-convex manifold
%complex space with bounded embedding dimension. 
and a holomorphic line bundle $L$ on $M$,
whose restriction to the exceptional set $\Sigma$ of $M$ is
positive in the sense of Grauert, then there exists an integer 
$k_0$ such that for all $k\geq k_0$ and $j\geq 1$
we have $H^j(M,L^k)=0$.
This is certainly the case if the holomorphic line bundle $L$ admits
a Hermitian metric $L$ such that $c_1(L,h^L)$ is positive in the
neighborhood of $\Sigma$.
The analog of this result for semipositive bundles is as follows.
%===
\begin{cor}\label{mi1convsp}
Let $M$ be a $1$-convex manifold of dimension $n$.
Let $(L,h^L)$ be a holomorphic Hermitian line bundle over $M$ which is 
semipositive in the neighborhood of the exceptional analytic set $\Sigma$
of $M$.
Then we have for all 
$j\ge 1$, as $k\to\infty$,
\begin{equation}\label{weakqsp2}
\begin{aligned}
\dim H^{j}(M,L^k)=o(k^{n}).
\end{aligned}
\end{equation}
\end{cor}
%===
\begin{proof}
For the proof, we need the following construction of
a special strictly plurisubharmonic exhaustion function for $M$.
Consider a strictly plurisubharmonic smooth exhaustion function $\varphi_Y$ of $Y$, 
such that $\varphi_Y\geq0$ and $\{\varphi_Y=0\}=A$ (see \cite[p. 563]{Col98}). 
This is constructed by embedding $Y$ into a Euclidean space
$\C^N$ and constructing such a strictly psh exhaustion function
on $\C^N$.
Then $\varphi=\varphi_Y\circ\rho$ is a smooth plurisubharmonic exhaustion 
function of $X$, such that $\varphi\geq0$, $\{\varphi=0\}=\Sigma$ 
and $\varphi$ is strictly plurisubharmonic on $X\setminus\Sigma$.

Let $U$ be an open set containing $\Sigma$, such that $(L,h^L)$
is semipositive on $U$.
For the exhaustion function $\varphi$ above, there exist $c>0$
such that $\overline{M}_c\subset U$. Otherwise, there exists a sequence
$(x_m)$ in $M$ such that $\varphi(x_m)\leq\frac1m$ and $x_m\notin U$.
By extracting a convergent subsequence to $x\in M$, we have $\varphi(x)=0$,
thus $x\in\Sigma$, which leads to a contradiction to $x_m\notin U$.

The conclusion follows now immediately from Corollary \ref{miqconvsp}.
\end{proof}
%===

\noindent
Next, we consider the case of $q$-concave manifolds.
For $r\in\{0,\ldots,n\}$ set $M(\leq r)=\bigcup_{j=0}^r M(j)$.
%===
\begin{thm}\label{miqconc}
Let $M$ be a $q$-concave manifold of dimension $n$.
Let $c$ be a regular value of the exhaustion function of $M$
such that $M_c$ contains the exceptional compact set. Set $X_c=\partial M_c$.
Let $(L,h^L)$ be a holomorphic Hermitian line bundle over $M$.
Then we have for all $r\le n-q-1$, as $k\to\infty$,
\begin{equation}\label{strongqcconc}
\begin{split}
\sum_{j=0}^{r}(-1)^{r-j}&\dim{}  H^{j}(M_c,L^k)
\le\frac{k^{n}}{n!}
\int\limits_{M_c(\leq r)}(-1)^r c_1(L,h^L)^n\\
&+\frac{k^{n}}{(2\pi)^{n}}\sum_{j=0}^{r}(-1)^{r-j}\int\limits_{X_c}\,
\int\limits_{\R_{x}(j)}\big|\det(\dot{\mathcal{R}}^L_{b,x}-2\eta\dot{\mathcal L}_x)\big|
\,d\eta\,dv_{X_c}(x)+o(k^{n}),
\end{split}
\end{equation}
Moreover, we have for all $j\leq n-q-1$, as $k\to\infty$,
\begin{equation}\label{weakqconc}
\begin{aligned}
\dim H^{j}(M_c,L^k) \le&\frac{k^{n}}{n!}\int\limits_{M_c(j)}
(-1)^jc_1(L,h^L)^n+
\frac{k^{n}}{(2\pi)^{n}}\int\limits_{X_c}\,\int\limits_{\R_{x}(j)}\!
\big|\det(\dot{\mathcal{R}}^L_{b,x}-
2\eta\dot{\mathcal L}_x)\big|\,d\eta\,dv_{X_c}(x)\\ &+o(k^{n}).
\end{aligned}
\end{equation}
The same estimates hold if on the left-hand side of \eqref{strongqcconc}
and \eqref{weakqconc} we replace 
$H^{j}(M_c,L^k)$ by $H^{j}(M,L^k)$, for $j\leq n-q-1$.
\end{thm}
%===
\begin{proof}
A sublevel set $M_c$ as above satisfies the conditions
$Z(j)$ for all $j\in\{0,\ldots,n-q-1\}$; thus, for any holomorphic
vector bundle $E$ on $M$, we have canonical isomorphisms
$H^{j}(\ol M_c,E)\to H^{j}(M_c,E)$ for all $j\in\{0,\ldots,n-q-2\}$.
Moreover, by \cite[Theorem 3.4.9]{Hor:65}, for $E=L^k$ we have
\begin{equation}
	\dim H^{n-q-1}(M_c,E)\leq \dim H^{n-q-1}(\ol M_c,L^k).
    \label{eq:5.26}
    \end{equation}
Thus, the inequalities \eqref{strongqcconc} and \eqref{weakqconc} 
follow from Theorem \ref{thmmorse}.
\end{proof}

%===
\begin{cor}\label{miqconcsn}
Let $M$ be a $q$-concave manifold of dimension $n$.
Let $c$ be a regular value of the exhaustion function of $M$
such that $M_c$ contains the exceptional compact set. %Set $X_c=\partial M_c$.
Let $(L,h^L)$ be a holomorphic Hermitian line bundle over $M$
which is seminegative on $\overline{M}_c$.
Then we have for all $j\le n-q-1$, as $k\to\infty$,
\begin{equation}\label{weakqconcsn}
\begin{aligned}
\dim H^{j}(M,L^k)=\dim H^{j}(M_c,L^k)=o(k^{n}).
\end{aligned}
\end{equation}
\end{cor}
%===
\begin{proof}
We apply the inequalities \eqref{weakqconc} and observe that for 
all $j\leq n-q-1$ we have $M_c(j)=\varnothing$ and for all $j\leq n-q-1$ and $x\in X_c$
we have $\R_x(j)=\varnothing$,
since $c_1(L,h^L)$ is seminegative on $\overline{M}_c$ and
$\dot{\mathcal L}_x$ has at least $n-q$ negative eigenvalues.
\end{proof}
%===
Corollary \ref{miqconcsn} extends the following vanishing theorem 
of Andreotti-Tomassini \cite[Theorem 2]{AnTo:69} (see also \cite{LMW25}
for a differential-geometric proof): If $M$ is a $q$-concave manifold 
and $L\to M$ is a negative line bundle, then there exists an integer 
$k_0$ such that for all $k\geq k_0$ and $j\leq n-q-1$
we have $H^j(M,L^k)=0$.
%===

We now apply the 
holomorphic Morse inequalities on manifolds with boundary 
to obtain asymptotic lower bounds for the
dimension of the spaces of holomorphic sections
of a semipositive line bundle over a $1$-concave manifold
(cf.\ also \cite{Be05,LMW25,Mar96, Mar16}). 
This will also yield a compactification result for a $1$-concave manifold
to a Moishezon manifold.
Recall that by a classical result of Moishezon, any compact Moishezon manifold $X$ 
admits a proper modification $\widetilde{X}\to X$, obtained as a finite 
sequence of blow-ups along smooth centers, such that the resulting manifold 
$\widetilde{X}$ is projective algebraic (see \cite[Theorem 2.2.16]{MM07}).
%===
\begin{cor}\label{cor:one-concave-positive}
Let $M$ be a $1$-concave complex manifold of dimension $n\geq 3$, and
let $(L,h^L)$ be a semipositive holomorphic Hermitian line bundle over 
$M$.
Let $c$ be a regular value of the exhaustion function of $M$
such that $M_c$ contains the exceptional compact set. 
Set $X_c=\partial M_c$.
Assume that along $X_c$, the curvature of $L$ and the Levi form of $X_c$
are conformally equivalent in the sense that there exists a smooth 
positive function $g\in C^\infty(X_c)$ such that
\begin{equation}\label{eq:conformal}
\dot{\mathcal R}^L_{b,x}
=-g(x)\dot{\mathcal L}_{x},
\qquad x\in X_c,
\end{equation}
as Hermitian forms on \(T^{1,0}X_c\cap T^{1,0}M\). 
Then, as $k\to\infty$,
\begin{equation}\label{eq:h0-one-concave}
\begin{split}
\dim H^0(M_c,L^k)
&=
\frac{k^n}{n!}\int_{M_c}c_1(L,h^L)^n+\frac{k^n}{(2\pi)^n}
\int_{X_c}\int_{\mathbb R_x(0)}
\big|\det(\dot{\mathcal R}^L_{b,x}
-2\eta\dot{\mathcal L}_{x})\big|\,d\eta\,dv_{X_c}(x)\\
&\qquad+o(k^n),
\end{split}
\end{equation}
\begin{equation}\label{eq:hj-one-concave}
\dim H^j(M_c,L^k)=o(k^n), \quad\text{for any $j\in\{1,\ldots,n-2\}$}.
\end{equation}
The asymptotics \eqref{eq:h0-one-concave} and \eqref{eq:hj-one-concave} 
also hold with $H^j(M_c,L^k)$ replaced by $H^j(M,L^k)$. 
Furthermore, the manifold $M$ can be compactified to
a Moishezon manifold; that is, there exists a compact Moishezon 
$\widehat{M}$ such that $M$ is biholomorphic to an open set of $\widehat{M}$.
\end{cor}
%===
\begin{proof}
Let us first remark that since $(L,h)$ is semipositive, we have 
$M_c(j)=\emptyset$ for $j\geq1$
and $\int_{M_c}c_1(L,h^L)^n=\int_{M_c(0)}c_1(L,h^L)^n$.

Since $M$ is $1$-concave, $M_c$ satisfies condition $Z(j)$ for
$j=0,\ldots,n-2$. In particular, the preceding Morse inequalities apply in
degrees $0$ and $1$. The weak Morse inequality in degree $0$ gives the upper
bound in \eqref{eq:h0-one-concave}. For the lower bound, 
we apply the strong Morse inequality with $r=1$.
Since $L$ is semipositive, the interior contribution of index $1$ is empty. Moreover,
under the conformal equivalence assumption
$\dot{\mathcal R}^L_{b,x}=-g(x)\dot{\mathcal L}_{x}$, 
the boundary contribution
appearing in this lower bound coincides with the boundary term in the weak degree $0$ upper bound.
%, as in Berman's conformally equivalent boundary curvature case \cite[Corollary~6.7]{Be05}. 
Hence, the upper and lower bounds
coincide, and \eqref{eq:h0-one-concave} follows. 
The asymptotic formula for $H^0(M,L^k)$ follows
from the Andreotti-Grauert isomorphism theorem.

The leading factor in the asymptotic formula \eqref{eq:h0-one-concave}
satisfies 
$$\frac{1}{n!}\int_{M_c}c_1(L,h^L)^n
+
\frac{1}{(2\pi)^n}
\int_{X_c}\int_{\mathbb R_x(0)}
\big|\det(\dot{\mathcal R}^L_{b,x}
-2\eta\dot{\mathcal L}_{x})\big|\,d\eta\,dv_{X_c}(x)>0,$$
since the interior integral is non-negative and the boundary integral is positive. 
Hence, the dimension of the space $H^0(M,L^k)$ grows like $k^n$,
as $k\to\infty$ by \eqref{eq:h0-one-concave}.
By a theorem of Rossi \cite[Theorem 3, p.\,245]{Ro:65} and Andreotti-Siu \cite[Proposition 3.2]{AS:70}, any $1$-concave manifold of dimension greater than three can be compactified. 
Since $\dim H^0(M,L^k)\sim k^n$, as $k\to\infty$, \cite[Theorem 5.2]{Mar96} shows that $M$ can be compactified to a Moishezon manifold.
\end{proof}

For a compact complex manifold $M$ of dimension $n$ 
and a holomorphic line bundle
$L$ over $M$, the volume of $L$ is defined by 
\begin{equation}\label{vol-def}
{\rm vol}(L):=\limsup_{k\to\infty} n!\,k^{-n}\dim  H^0 (M, L^k).
\end{equation}
If $(L,h^L)$ is positive, we have
\begin{equation}\label{vol-ample}
\operatorname{vol}(L)=\int_Mc_1(L,h^L)^n.
\end{equation}
The line bundle $L$ is called big if its Kodaira-Iitaka dimension
equals the dimension of $M$ (see \cite[Definition 2.2.5]{MM07}). 
The line bundle $L$ is big if and only if ${\rm vol}(L)>0$ (see 
\cite[Theorem 2.2.7]{MM07}). 
The concepts of Kodaira–Iitaka dimension, bigness of a line bundle, 
and the volume of a line bundle extend in a natural 
manner to the setting of \(1\)-concave complex manifolds 
(see \cite[\S 3.4]{MM07}).
Thus, under the hypotheses of Corollary \ref{cor:one-concave-positive}, 
we obtain the following expression for the volume of $L$.
\begin{cor} 
Under the hypotheses of Corollary \ref{cor:one-concave-positive},
we have 
\begin{equation}\label{term1}
{\rm vol}(L)=\int_{M_c}c_1(L,h^L)^n
+\frac{n!}{(2\pi)^n}
\int_{X_c}\int_{\mathbb R_x(0)}
\big|\det(\dot{\mathcal R}^L_{b,x}
-2\eta\dot{\mathcal L}_{x})\big|\,d\eta\,dv_{X_c}(x).
\end{equation}
In particular, if $g\equiv 1$, then
\begin{equation}\label{term2}
{\rm vol}(L)=\int_{M_c}c_1(L,h^L)^n
+\frac{(n-1)!}{2(2\pi)^n}\int_{X_c}\big|\det(\dot{\mathcal R}^L_{b,x})\big|dv_{X_c}(x).
\end{equation}
\end{cor}
\begin{proof}
Formula \eqref{term1} follows directly from Corollary~\ref{cor:one-concave-positive}. We prove \eqref{term2} in the case $g\equiv1$. It follows that $\dot{\mathcal R}^L_{b,x}-2\eta\dot{\mathcal L}_{x}=(1+2\eta)\dot{\mathcal R}^L_{b,x}$.
Since $\dot{\mathcal R}^L_{b,x}$ is positive, the Hermitian form
$(1+2\eta)\dot{\mathcal R}^L_{b,x}$ has no negative eigenvalues precisely when
$1+2\eta>0$. By the definition of \(\mathbb R_x(0)\) in \eqref{rj}, we have
$$
\mathbb R_x(0)=\left(-1/2,0\right).
$$
Therefore
$$
\begin{aligned}
\int_{\mathbb R_x(0)}
\big|\det(\dot{\mathcal R}^L_{b,x}-2\eta\dot{\mathcal L}_{x})\big|\,d\eta
=\int_{-1/2}^{0}(1+2\eta)^{n-1}\big|\det(\dot{\mathcal R}^L_{b,x})\big|\,d\eta
=\frac{1}{2n}\big|\det(\dot{\mathcal R}^L_{b,x})\big|.
\end{aligned}
$$
Substituting this into \eqref{term1}, we obtain \eqref{term2}.
\end{proof}
The formula for ${\rm vol}(L)$ consists of an interior term 
given by the integral of $c_1(L,h^L)^n$, as in the compact case, 
together with an additional boundary term that reflects the specific 
geometric features of the $1$-concave manifold under consideration. 
Moreover, the line bundle $L$ is big.

\subsection{Geometric examples of the \texorpdfstring{$1$}{1}-concave setting}
We conclude with two geometric applications of the preceding results. 
On the one hand, we construct natural \(1\)-concave examples arising from the 
regular locus of compact normal complex spaces with isolated singularities. 
On the other hand, we use the positivity of the leading coefficient, 
which is stable under small deformations, leading to 
pseudoconcave Moishezon manifolds in the deformed setting.

We now describe a natural class of examples to which 
Corollary~\ref{cor:one-concave-positive} applies.
This will be constructed by starting with a compact normal complex space $X$ 
with at most isolated singularities, and letting $(L,h^L)$ be a holomorphic Hermitian 
line bundle on $X$.
We first recall some necessary notions for this context.

Let $Z$ be a (reduced) normal complex space and denote by $Z_\reg$, 
resp.\ $Z_\sing$, the set of regular, resp.\ singular points of $X$.
A chart $(U,\iota,V)$ on $Z$  is a triple consisting of an open set 
$U \subset Z$, a closed complex space 
$V \subset G \subset \C^N$ in an open set $G$ of $\C^N$, and a biholomorphic 
map $\iota:\ U \to V$. We call $\iota:\ U \to G \subset \C^N$ a local 
embedding of $X$.
A strictly plurisubharmonic (psh) function on $Z$ 
is a function $\varphi:Z\to[-\infty,\infty)$ such that for every 
$x\in Z$, there exists a local embedding $\tau:U\to G\subset\C^N$ 
with $x\in U$ and a strictly psh function 
$\widetilde\varphi:G\to[-\infty,\infty)$ such that 
$\varphi|_U=\widetilde\varphi\circ\tau$. If $\widetilde\varphi$ 
can be chosen smooth, then $\varphi$ is 
called a smooth strictly psh function. 
The definition is independent of the chart. 
%===
\begin{prop}\label{prop:isolated-sing-example}
Let $Z$ be a compact normal complex space of dimension $n\geq 3$ 
with only isolated singularities, and let $(L,h^L)$ be a holomorphic Hermitian 
line bundle on $Z$. Assume that
\begin{enumerate}
\item $c_1(L,h^L)\geq 0$ on $Z_{\reg}$,
\item $c_1(L,h^L)>0$ in a neighborhood of $Z_{\sing}$,
\item for every singular point $p\in Z_{\sing}$ there exists an open neighborhood 
$U$ of $p$ such that $L|_{U}$ is trivial and a local
 weight $\varphi\in \mathscr C^\infty(U)$ of $h^L$ 
that attains a minimum at $p$ and $\{\varphi_p<d\}\Subset U$
for $d>\varphi(p)$ in a neighborhood of $\varphi(p)$.
%which is a strictly psh exhaustion function that attains a minimum at $p$.
%
%for some local weight 
%$\varphi_p\in \mathscr C^\infty(U_p)$ of $h^L$, the strictly psh
%function $\varphi_p$ attains a minimum at $p$.
\end{enumerate}
Consider the $1$-concave manifold \(M:=Z_{\reg}\). 
Then there exists a smooth superlevel exhaustion
function $\psi:M\to[0,\infty)$ and $d>0$ such that $\psi$ is
strictly psh on $\{\psi<d\}$ and for any regular value $c\in(0,d)$, 
the superlevel set \(M_c=\{\psi>c\}\Subset M\)
has a smooth boundary \(X_c:=\partial M_c\), and along \(X_c\) 
the curvature of \(L\) and the Levi form of \(X_c\) are conformally equivalent
in the sense of \eqref{eq:conformal}, 
in fact with \(g\equiv 1\). Consequently, Corollary~\ref{cor:one-concave-positive} 
applies to \(M\) and \(L\).
\end{prop}
%===
\begin{proof}
For simplicity, we first assume that \(Z\) has only one singular point \(p\). 
Let \(|e_U|_{h^L}^2=e^{-2\varphi}\)
where \(\varphi\in\mathscr C^\infty(U)\) is as in hypothesis (3) 
%strictly plurisubharmonic
%local weight of \(h^L\) on \(U\) which attains a minimum at \(p\). 
After adding a constant, we may assume that
\(\varphi(p)=0.\)
Choose \(d>0\) so small that
\(U_d:=\{x\in U:\varphi(x)<d\}\) is relatively compact in \(U\).

We now construct a smooth function
\(
\psi:M=Z_{\reg}\to[0,\infty)
\)
such that
\[
\psi=\varphi \quad \text{on } U_d,
\qquad
\psi\ge d \quad \text{on } M\setminus U_d.
\]
Since \(Z\) is compact and \(p\) is the unique singular point, 
it follows that for every \(c\in(0,d)\), the superlevel set 
\(M_c:=\{x\in M:\psi(x)>c\}\)
is relatively compact in \(M\). 
Moreover, \(\psi\) is strictly plurisubharmonic on 
\(\{\psi<d\}\). Thus \(\psi\) is a smooth superlevel exhaustion function 
of \(M\) having the required properties.

Let \(c\in(0,d)\) be a regular value of \(\psi\), and set
\(
X_c:=\partial M_c=\{x\in M:\psi(x)=c\}.
\)
Since \(X_c\subset U_d\), we have \(\psi=\varphi\) in a neighborhood of \(X_c\). 
Therefore, the defining function
\(
\rho_c:=c-\psi
\)
satisfies
\[
\partial\bar\partial \rho_c=-\partial\bar\partial \psi
=-\partial\bar\partial \varphi
\]
near \(X_c\). On the other hand,
\(
\sqrt{-1}\partial\bar\partial \varphi_p=c_1(L,h^L)
\)
on \(U\cap Z_{\reg}\). Hence, along \(Z_c\), the Levi form of \(X_c\) 
and the curvature of \(L\) are conformally equivalent in the sense of 
\eqref{eq:conformal}, in fact with \(g\equiv 1\).
By assumption \(c_1(L,h^L)\ge0\) on \(M=Z_{\reg}\), \((L,h^L)\)
is semipositive on \(M\). Therefore, all hypotheses of Corollary~\ref{cor:one-concave-positive} 
are satisfied, and the conclusion follows.

If \(Z\) has finitely many isolated singular points, the proof is the same: 
one chooses pairwise disjoint neighborhoods \(U_p\) of the singular points, 
corresponding strictly plurisubharmonic local weights \(\varphi_p\), 
and patches them together into a smooth function \(\psi\) on \(Z_{\reg}\) 
which agrees with \(\varphi_p\) near each singular point and is bounded 
below away from these neighborhoods.
\end{proof}

\begin{cor}\label{cor:projective-isolated-sing}
Let $Z\subset \mathbb P^N$ be a compact normal complex space of dimension 
$n\geq 3$ with only isolated singularities, and let
\(L=\mathcal{O}(1)|_Z\)
be endowed with the induced Fubini-Study metric. 
Then the assumptions of Proposition~\ref{prop:isolated-sing-example} 
are satisfied. In particular, the regular locus $Z_{\reg}$ is a \(1\)-concave 
complex manifold to which Corollary~\ref{cor:one-concave-positive} applies.
\end{cor}

\begin{proof}
The bundle $\mathcal O(1)|_Z$ is positive everywhere. 
It remains to verify condition~(3) in 
Proposition~\ref{prop:isolated-sing-example}. Let $p\in Z_{\sing}$. 
Since $\operatorname{Aut}(\mathbb P^N)=\operatorname{PGL}(N+1,\mathbb C)$, 
after composing with a projective automorphism we may assume that $p=[1,0,\ldots,0]$; 
moreover, such an automorphism preserves the hyperplane bundle, so 
$f^*\mathcal O(1)\simeq \mathcal O(1)$.
In the affine chart $\{Z_0\neq 0\}\simeq \C^N$, with coordinates 
$z=(z_1,\ldots,z_N)$, the Fubini--Study metric on $\mathcal O(1)$ has local weight
\(\varphi(z)=\log(1+|z|^2),\)
which is smooth and strictly plurisubharmonic and attains its minimum at $z=0$, 
corresponding to the point $p$. Hence, all assumptions of 
Proposition~\ref{prop:isolated-sing-example} are fulfilled.
\end{proof}

The deformation theory of complex structures on \(1\)-concave manifolds is 
closely related to the deformation and embeddability theory of the induced 
CR structures on their pseudoconvex boundaries. This point of view, developed 
in particular in the work of Epstein, Lempert, Epstein–Henkin, and others, 
shows that small deformations of the boundary CR structure can often 
be studied by extending them to the pseudoconcave side, where positivity of divisors or line bundles provide an effective tool for proving 
embeddability, stability, and Moishezon-type results; see for instance 
\cite{Eps98,EpHe00,Le3,Le2}. 

In \cite{Mar16}, a similar strategy was used 
in the presence of a positive divisor, following ideas of Lempert. 
In the present setting, the key observation is slightly different: 
although the conformality assumption along the boundary need not 
be stable under deformation, the leading coefficient in the Morse inequality 
depends continuously on the complex structure and on the corresponding 
Hermitian data. Hence strict positivity of this coefficient is an open condition, 
and this suffices to deduce the persistence of the Moishezon property under 
sufficiently small deformations.
%====
\begin{thm}
Let $M$ be a $1$-concave complex manifold of dimension $n\ge 3$, and let
$(L,h^L)$ be a positive holomorphic Hermitian line bundle over $M$
given by a smooth divisor $Z\subset M$ with $L=[Z]$.
Let $c$ be a regular value of the exhaustion function of $M$ such that
$M_c\Subset M$ contains the exceptional compact set and $Z\subset M_c$.
Set $X_c=\partial M_c$.
Assume that the 
leading coefficient in the asymptotic formula \eqref{eq:h0-one-concave},
\[
\frac{1}{n!}\int_{M_c}c_1(L,h^L)^n
+
\frac{1}{(2\pi)^n}\int_{X_c}\int_{\mathbb R_x(0)}
\big|\det(\dot{\mathcal R}^{L}_{b,x}-2\eta\dot{\mathcal L}_{x})\big|\,d\eta\,dv_{X_c}(x)
\]
is strictly positive.
Then any sufficiently small deformation of the complex structure of $M_c$
leaving $T(Z)$ invariant is a pseudoconcave Moishezon manifold.
If $K_M$ is positive, the same result remains valid under any sufficiently small deformation of the complex structure of $M_c$.
\end{thm}

\begin{proof}
Let $J$ be the given complex structure on $M_c$, and let $J'$ 
be a sufficiently small deformation of $J$ leaving $T(Z)$ invariant. 
Denote by $M_c'$ the manifold $M_c$ endowed with the complex structure $J'$.
Since $T(Z)$ is $J'$-invariant, the divisor $Z$ is also a smooth complex 
divisor in $M_c'$, and hence determines a holomorphic line bundle
\(L'=[Z]\to M_c'\).
By Lemma~4.1 and Lemma~4.2 in \cite{Mar16}, 
which are based on the corresponding construction of Lempert \cite[Lemma~4.1]{Le1}, 
the bundle $L'$ may be represented near $M_c$ by transition functions arbitrarily close 
to those of $L$, and it admits a Hermitian metric $h^{L'}$ whose curvature form 
is arbitrarily close to $c_1(L,h^L)$ in the $\mathscr C^\infty$ topology, provided that
 $J'$ is sufficiently close to $J$. Since $L$ is positive, positivity being an open condition, 
 it follows that $L'$ is again positive for every sufficiently small deformation.

Consider now the leading coefficient in the weak Morse inequality in degree \(0\):
\[
\mathcal A(J,L,h^L)
:=
\frac{1}{n!}\int_{M_c}c_1(L,h^L)^n
+
\frac{1}{(2\pi)^n}
\int_{X_c}\int_{\mathbb R_x(0)}
\bigl|\det(\dot{\mathcal R}^{L}_{b,x}-2\eta\dot{\mathcal L}_{x})\bigr|\,d\eta\,dv_{X_c}(x).
\]
By hypothesis, this quantity is strictly positive. Since the curvature form, the Levi form, the 
set $\mathbb R_x(0)$, and the induced boundary volume element all depend continuously on the complex structure and the metric, the corresponding coefficient
\(
\mathcal A(J',L',h^{L'})
\)
for the deformed structure remains positive for $J'$ sufficiently close to $J$.
Applying the weak Morse inequality in degree \(0\) to the deformed manifold \(M_c'\), we obtain
\[
\dim H^0(M_c',L'^k)\geq \mathcal A(J',L',h^{L'})\,k^n+o(k^n),
\qquad k\to\infty.
\]
In particular,
\[
\limsup_{k\to\infty} \, k^{-n}\dim H^0(M_c',L'^k)>0,
\]
so \(L'\) is a big line bundle on \(M_c'\). Hence \(M_c'\) is Moishezon.

Since \(M_c\) is a relatively compact domain with pseudoconcave boundary, 
the same holds for the nearby deformed manifold \(M_c'\), so \(M_c'\) is a pseudoconcave Moishezon manifold.

If \(K_M\) is positive, then the same argument applies to the canonical bundle. 
Indeed, by Lemma~4.3 of \cite{Mar16}, again relying on the local deformation 
result of Lempert \cite{Le1}, any sufficiently small deformation \(J'\) induces 
a holomorphic canonical bundle \(K_{M_c'}\) carrying a Hermitian metric 
whose curvature is arbitrarily close to that of \(K_{M_c}\). 
Since positivity is open, \(K_{M_c'}\) remains positive for all sufficiently 
small deformations, and the preceding argument shows that \(M_c'\) is Moishezon.
\end{proof}

\subsection{Semi-classical Weyl law}\label{5.3}

As another application of the heat kernel asymptotics obtained above, we derive a
semi-classical Weyl law for the $\dbar$-Neumann Laplacian $P_k^q:=\frac{1}{k}\Box_k^q$, where $\Box_k^q$ is defined with respect to the $k$-dependent Hermitian metric introduced in \eqref{e-gue260125yyd}. Throughout this section, we assume that $\ol M$ is compact and that condition $Z(q)$ holds. Then $\Box^q_k$ has a discrete spectrum.
For $\lambda\ge0$, we define the counting function
$$
N^q_k(\lambda):=\Tr_q\mathbf 1_{[0,\lambda]}(P_k^q)=
\sum_{\mu\in\operatorname{Spec}(P_k^q)\cap[0,\lambda]}
\dim E_\mu(P_k^q),
$$
where $E_\mu(P_k^q)$ denotes the eigenspace of $P_k^q$ corresponding to the eigenvalue $\mu$. Thus $N^q_k(\lambda)$ counts the eigenvalues of $\frac{1}{k}\Box^q_k$ not exceeding $\lambda$, counted with multiplicities.

We first write the heat trace asymptotics in a form suitable for the Weyl law. For $t>0$, put
\begin{equation}\label{inweyl}
H^M_q(t,\mathbf z):=\frac{1}{(2\pi)^n}\frac{\det\dot{\mathcal R}^L_{\mathbf z}}{\det(1-e^{-t\dot{\mathcal R}^L_{\mathbf z}})}\Tr_q e^{-t\varTheta_{\mathbf z}},
\qquad \mathbf z\in M.
\end{equation}
For $x\in X$ and $r\le0$, define the boundary heat density by
\begin{equation}\label{bweyl}
\begin{aligned}
H^X_q(t,x,r):=&\frac{1+e^{-2r^2/t}}{(2\pi)^n\sqrt{2\pi t}}\int_{\R}\frac{e^{-t\eta^2/2}\det(\dot{\mathcal R}^L_{b,x}-2\eta\dot{\mathcal L}_x)}{\det\big(1-e^{-t(\dot{\mathcal R}^L_{b,x}-2\eta\dot{\mathcal L}_x)}\big)}\Tr_q e^{-t\varTheta^{\eta,\tau}_x}\,d\eta
\\&+\frac{2}{(2\pi)^n\sqrt{\pi}}\int_{\R}\frac{\eta e^{-2r\eta}\det(\dot{\mathcal R}^L_{b,x}-2\eta\dot{\mathcal L}_x)}{\det\big(1-e^{-t(\dot{\mathcal R}^L_{b,x}-2\eta\dot{\mathcal L}_x)}\big)}\Tr_q e^{-t\varTheta^{\eta,\tau}_x}
\left(\int^{\frac{2r-t\eta}{\sqrt{2t}}}_{-\infty}e^{-\gamma^2}\,d\gamma\right)d\eta
\\&+\frac{1-e^{-2r^2/t}}{(2\pi)^n\sqrt{2\pi t}}\int_{\R}\frac{e^{-t\eta^2/2}\det(\dot{\mathcal R}^L_{b,x}-2\eta\dot{\mathcal L}_x)}{\det\big(1-e^{-t(\dot{\mathcal R}^L_{b,x}-2\eta\dot{\mathcal L}_x)}\big)}\Tr_q e^{-t\varTheta^{\eta,\nu}_x}\,d\eta.
\end{aligned}
\end{equation}
In this context, the terms $\mu/(1-e^{-t\mu})$ occurring in \eqref{inweyl} and \eqref{bweyl} are defined at $\mu=0$ by taking their continuous extension. Set
\begin{equation}\label{Thetaq}
\Theta_q(t):=\int_M H^M_q(t,\mathbf z)\,dv_{M'}(\mathbf z)+\int_X\int_{-\infty}^{0}H^X_q(t,x,r)\,dr\,dv_X(x).
\end{equation}
Integrating the diagonal heat kernel asymptotics and using the same cut-off estimates as in the proof of Theorem~\ref{thmmorse}, we obtain, for every $t>0$,
\begin{equation}\label{heat-trace}
\Theta_q(t)=\lim_{k\to\infty}k^{-n}
\Tr_q\left(e^{-\frac{t}{k}\Box^q_k}\right).
\end{equation}
Let $\nu^M_{q,\mathbf z}$ and $\nu^X_{q,x,r}$ be the local spectral measures characterized, respectively, by
\begin{equation}\label{nu-MX}
\int_0^\infty e^{-t\lambda}\,d\nu^M_{q,\mathbf z}(\lambda)=H^M_q(t,\mathbf z)\text{\ \  and \  }
\int_0^\infty e^{-t\lambda}\,d\nu^X_{q,x,r}(\lambda)=H^X_q(t,x,r).    
\end{equation}

\begin{thm}[Semi-classical Weyl law]\label{thm:scwl}
Assume that $\ol M$ is compact and that condition $Z(q)$ holds. Then there exists a unique positive Borel measure $\mu^q$ on $[0,\infty)$ such that
\begin{equation}\label{mu-q}
\int_0^{\infty}e^{-t\lambda}\,d\mu^q(\lambda)=\Theta_q(t),\qquad t>0.
\end{equation}
Moreover, the measure $\mu^q$ is given by
\begin{equation}\label{mu-l}
\mu^q([0,\lambda])=\int_M \nu^M_{q,\mathbf z}([0,\lambda])\,dv_{M'}(\mathbf z)
+\int_X\int_{-\infty}^{0}\nu^X_{q,x,r}([0,\lambda])\,dr\,dv_X(x)    
\end{equation}
at every continuity point $\lambda$ of the distribution function $\lambda\mapsto\mu^q([0,\lambda])$. In particular,
for every $\lambda\ge0$ such that $\mu^q(\{\lambda\})=0$, we have
\begin{equation}\label{scweyl}
N^q_k(\lambda)=k^n\mu^q([0,\lambda])+o(k^n),\qquad k\to\infty.
\end{equation}
\end{thm}

\begin{proof}
For each $k$, define the normalized spectral measure of $P_k^q$ by
\begin{equation*}
\mu^q_k:=k^{-n}\sum_{\mu\in\operatorname{Spec}(P_k^q)}\dim E_\mu(P_k^q)\,\delta_\mu,
\end{equation*}
where $\delta_\mu$ denotes the Dirac measure at $\mu$. Then $\mu^q_k$ is a positive Borel measure on $[0,\infty)$ and
$$
\mu^q_k([0,\lambda])=k^{-n}N^q_k(\lambda).
$$
Moreover, by the spectral theorem and the definition of $\mu^q_k$, for every $t>0$ we have
\begin{equation*}
\begin{aligned}
\int_0^{\infty}e^{-t\lambda}\,d\mu^q_k(\lambda)
=k^{-n}\operatorname{Tr}_q(e^{-tP_k^q})
=k^{-n}\operatorname{Tr}_q\left(e^{-\frac{t}{k}\Box^q_k}\right).
\end{aligned}
\end{equation*}
Hence, by \eqref{heat-trace},
\begin{equation}\label{conweyl}
\lim_{k\to\infty}\int_0^{\infty}e^{-t\lambda}\,d\mu^q_k(\lambda)
=\Theta_q(t),\qquad t>0.
\end{equation}
We first note that the measures $\mu^q_k$ are locally uniformly bounded. Indeed, fix $\Lambda>0$ and $t_0>0$. Since $e^{-t_0\lambda}\ge e^{-t_0\Lambda}$ for $0\le\lambda\le\Lambda$, we have
$$
\mu^q_k([0,\Lambda])\le e^{t_0\Lambda}\int_0^{\infty}e^{-t_0\lambda}\,d\mu^q_k(\lambda)
= e^{t_0\Lambda}k^{-n}\operatorname{Tr}_q\left(e^{-\frac{t_0}{k}\Box^q_k}\right).
$$
The right-hand side is uniformly bounded in $k$ by the heat trace asymptotics \eqref{heat-trace}. Hence, the family $\{\mu^q_k\}_{k\ge1}$ is locally uniformly bounded. Therefore, after passing to a subsequence, there exists a positive Borel
measure $\nu$ on $[0,\infty)$ such that
$$
\int_0^\infty f(\lambda)\,d\mu^q_k(\lambda)\to
\int_0^\infty f(\lambda)\,d\nu(\lambda)
$$
for every continuous function $f$ with compact support in $[0,\infty)$. 

We claim that the Laplace transform of $\nu$ is $\Theta_q(t)$. Since $e^{-t\lambda}$
is not compactly supported, we use a standard tail estimate. For $R>0$, we have
$$
\int_R^\infty e^{-t\lambda}\,d\mu^q_k(\lambda)
\le e^{-\frac{tR}{2}}\int_0^{\infty}e^{-\frac{t}{2}\lambda}\,d\mu^q_k(\lambda).
$$
The last integral is uniformly bounded in $k$ by \eqref{conweyl} with $t$ replaced by $t/2$. Hence, the tail is uniformly small as $R\to\infty$. Together with the convergence of measures on compact subsets, this implies
$$
\int_0^{\infty}e^{-t\lambda}\,d\nu(\lambda)
=\Theta_q(t),\qquad t>0.
$$
By the uniqueness theorem for Laplace transforms of positive Borel measures on $[0,\infty)$, the subsequential limit $\nu$ is unique. We denote this limit by $\mu^q$. Hence, every subsequence has the same limit $\mu^q$, and $\mu^q$ is uniquely characterized by \eqref{mu-q}.

We next identify $\mu^q$ in terms of the local model spectral measures. Let $\widetilde\mu^q$ be the positive Borel measure defined by
$$
\widetilde\mu^q([0,\lambda])=\int_M \nu^M_{q,\mathbf z}([0,\lambda])\,dv_{M'}(\mathbf z)+\int_X\int_{-\infty}^{0}\nu^X_{q,x,r}([0,\lambda])\,dr\,dv_X(x).
$$
where $\nu^M_{q,\mathbf z}$ and $\nu^X_{q,x,r}$ are given by \eqref{nu-MX}. For every $t>0$, the Laplace transform of $\widetilde\mu^q$ is
$$
\begin{aligned}
\int_0^\infty e^{-t\lambda}\,d\widetilde\mu^q(\lambda)
&=\int_M H^M_q(t,\mathbf z)\,dv_{M'}(\mathbf z)+\int_X\int_{-\infty}^{0}H^X_q(t,x,r)\,dr\,dv_X(x)=\Theta_q(t).
\end{aligned}
$$
Thus $\widetilde\mu^q$ has the same Laplace transform as $\mu^q$. By the uniqueness of Laplace transforms, $\widetilde\mu^q=\mu^q$. This proves \eqref{mu-l}.
	
Now, let $\lambda\ge0$ be such that $\mu^q(\{\lambda\})=0$. Then the convergence of measures on compact subsets gives
$$
\lim_{k\to\infty}\mu^q_k([0,\lambda])=\mu^q([0,\lambda]).
$$
Since $\mu^q_k([0,\lambda])=k^{-n}N^q_k(\lambda)$, we obtain \eqref{scweyl}. This proves the theorem.
\end{proof}

\begin{rem}
The argument is in the same spirit as the proof of the asymptotic distribution of eigenvalues in \cite[\S3.2.2]{MM07} for the
Kodaira Laplacian with Dirichlet boundary conditions. In that case, see
\cite[(3.2.15)]{MM07}, contains only the interior contribution:
\begin{equation*}\label{mu-q1}
\int_0^{\infty}e^{-t\lambda}\,d\mu^q(\lambda)=
\int_M H^M_q(t,\mathbf z)\,dv_{M'}(\mathbf z)=
\frac{1}{(2\pi)^n}\int_M\frac{\det\dot{\mathcal R}^L_{\mathbf z}}{\det(1-e^{-t\dot{\mathcal R}^L_{\mathbf z}})}\Tr_q e^{-t\varTheta_{\mathbf z}}\,dv_{M'}(\mathbf z),\quad t>0.
\end{equation*}
For the interior term, the distribution function 
$\nu^M_{q,\mathbf z}([0,\lambda])$ can be written explicitly 
in terms of the eigenvalues of the curvature of $L$, 
whose expression first appeared in 
Demailly \cite[Th\'eor\`eme~0.6]{De85}; 
see also Ma–Marinescu \cite[(3.2.37)]{MM07}.
In the present $\dbar$-Neumann setting, the semi-classical Weyl formula 
\eqref{scweyl} contains, in addition, the contribution of the boundary 
model operator.
\end{rem}

\begin{rem}
Letting $\lambda\to 0+$, we have 
\begin{equation}\label{mu-limit}
\lim_{\lambda\to0+}\mu^q([0,\lambda])
=\mu^q(\{0\})=\lim_{t\to+\infty}\Theta_q(t).    
\end{equation}
In fact, since $\Theta_q(t)=\int_0^\infty e^{-t\lambda}\,d\mu^q(\lambda)$ and the integrand $e^{-t\lambda}$ converges pointwise to the indicator function $\mathbf 1_{\{0\}}(\lambda)$ as $t\to+\infty$, the last equality in \eqref{mu-limit} follows directly from the monotone convergence theorem.
By the computation of the large-time limit in the proof of Theorem~\ref{thmmorse}, the quantity in \eqref{mu-limit} is given explicitly by
\begin{equation*}
\mu^q(\{0\})=\frac{1}{(2\pi)^{n}}\Big(\int_{M(q)}\big|\det(\dot{\mathcal{R}}^L_{\mathbf z})\big|dv_{M'}(\mathbf z)+\int_X\int_{\R_{x}(q)}\big|\det(\dot{\mathcal{R}}^L_{b,x}-2\eta\dot{\mathcal L}_x)\big|\,d\eta\,dv_X(x)\Big).
\end{equation*}
This is exactly the coefficient of $k^n$ in the weak Morse inequality \eqref{weak}. For fixed $\lambda>0$, however, $\mu^q([0,\lambda])$ in general does not simplify to the curvature-sign determinant expression appearing in the Morse inequalities: the latter only detects the zero-energy part of the model operators, while the Weyl law involves their full spectral distribution below $\lambda$.
\end{rem}

\bibliographystyle{alpha}
\bibliography{heat}
\setlength{\itemsep}{5pt}

\end{document}